\documentclass{article}
\usepackage{graphicx,color}
\usepackage{amsmath,graphicx,amsthm,amsfonts,color,float,tikz}
\graphicspath{{figures/}}
\title{Ahlfors Regular Conformal Dimension of Metrics on Infinite Graphs and Spectral Dimension of the Associated Random Walks}
\author{K\^{o}hei Sasaya\thanks{Research Institute for Mathematical Sciences, Kyoto University, Kyoto 606-8502, Japan. JSPS Research Fellow (DC1). E-mail: \texttt{ksasaya@kurims.kyoto-u.ac.jp}}}

\makeatletter

\@addtoreset{equation}{section}

\@addtoreset{figure}{section}
\makeatother

\newtheorem{lem}{Lemma}[section]
\newtheorem{prop}[lem]{Proposition}
\newtheorem{thm}[lem]{Theorem}
\newtheorem{cor}[lem]{Corollary}
\theoremstyle{definition}
\newtheorem{defi}[lem]{Definition}
\newtheorem{ex}[lem]{Example}
\theoremstyle{remark}
\newtheorem*{rem}{Remark}

\newcommand{\qs}{\underset{\mathrm{QS}}{\sim}}
\newcommand{\gen}{\underset{\mathrm{GE}}{\sim}}
\newcommand{\ard}{\dim_\mathrm{AR}}
\newcommand{\diam}{\mathrm{diam}}
\newcommand{\id}{\mathrm{id}}
\newcommand{\ol}[1]{\overline{#1}}
\newcommand{\ul}[1]{\underline{#1}}
\newcommand{\Mc}[1]{\mathcal{#1}}
\newcommand{\Mb}[1]{\mathbb{#1}}
\newcommand{\Mf}[1]{\mathfrak{#1}}
\newcommand{\Cup}{\bigcup}

\newcommand{\tr}{\triangle}
\newcommand{\epkw}[1]{\mathcal{E}_{p,k,#1}(N_1,N_2,N)}
\newcommand{\etkw}[1]{\mathcal{E}_{2,k,#1}(N_1,N_2,N)}
\newcommand{\ccdd}[5]{\left(\frac{d_{\Mf{C}}(#2,#3)}{d(#1,#3)}\bigvee \frac{d_{\Mf{C}}(#3,#4)}{d(#3,#5)} \right)}
\newcommand{\crt}{\frac{1}{2}+\frac{\sqrt{3}}{2}i}
\newcommand{\kig}[1]{the counterpart of \cite[#1]{Kig18}}

\begin{document}
\maketitle

\begin{abstract}
Quasisymmetry is a well-studied property of homeomorphisms between metric spaces, and Ahlfors regular conformal dimension is a quasisymmetric invariant. In the present paper, we consider the Ahlfors regular conformal dimension of metrics on infinite graphs, and show that this notion coincides with the critical exponent of $p$-energies. Moreover, we give a relation between the Ahlfors regular conformal dimension and the spectral dimension of a graph.
\end{abstract}

\tableofcontents

\section{Introduction}
Quasisymmetry is a well-studied property of homeomorphisms between metric spaces, and roughly speaking, means that the homeomorphism in question preserves ratios of distances. The Ahlfors regular conformal dimension is a quasisymmetric invariant of metric spaces, which gives a measure of the simplest (in a certain sense) quasisymmetrically equivalent space. The purpose of this paper is to study the Ahlfors regular conformal dimension of discrete unbounded metric spaces, and show relations between the Ahlfors regular conformal dimensions and spectral dimensions of such spaces.\\
Quasisymmetry was introduced by Tukia and V\"ais\"ala in~\cite{TV} to generalize the notion of quasiconformal mappings on the complex plane. In~\cite{TV}, quasisymmetry was given as a property of a homeomorphism between two metric spaces. A specialization of this was given by Kigami~\cite{Kig12}, for the comparison of metrics on the same underlying space. This is the definition we will use. 

\begin{defi}[Quasisymmetry, Kigami's]
 Let $X$ be a set and $d,\rho$ be metrics on $X,$ and let $\theta:[0,\infty)\to[0,\infty)$ be a homeomorphism. Then we say $d$ is $\theta$-quasisymmetric to $\rho$ if for any $x,y,z\in X$ with $x\ne z,$
 \[ \frac{\rho(x,y)}{\rho(x,z)}\le\theta\left(\frac{d(x,y)}{d(x,z)}\right ).\]
 Moreover, if $d$ is $\theta$-quasisymmetric to $\rho$ for some $\theta,$ then we say that $d$ is quasisymmetric to $\rho$ and write $d\qs\rho.$
\end{defi}
For example, $d\qs d^\alpha$ for any $\alpha\in(0,1)$ and any metric space $(X,d).$\\
 This notion was also called ``quasisymmetrically related by the identity map'' in~\cite{Hei}, and ``quasisymmetrically equivalent'' in~\cite{Pia13}. Note that if there exists a quasisymmetric map $f:(X,d)\to(Y,\rho)$, then $d$ is quasisymmetric to the pull-back metric $\rho^*$ in the sense of Kigami's definition and we can identify $(Y,\rho)$ with $(X,\rho^*).$  
\\Quasisymmetry has been studied in various fields. For example, a quasi-\hspace{0pt}isometric map (the definition is in~\cite[Definition 3.2.11]{MT}, for example) between Gromov hyperbolic spaces induces a quasisymmetric map (see~\cite[Theorem 3.2.13 and Section 3.6]{MT}, or~\cite{Pau}). There is also much research about quasisymmetry and Gromov hyperbolic spaces (see ~\cite{MT} for example). Quasisymmetry is a weaker notion of bi-Lipschitz equivalence, which has been studied extensively for decades, see~\cite{Hei} or~\cite{See} for example.
From the viewpoint of global analysis, it is notable that quasisymmetric modification preserves the volume doubling property, which plays an important role in heat kernel estimates. This idea is used in~\cite{Kig12},~\cite{Kig14}, and there is a recent application to circle packing graphs in~\cite{Mur}.\\
Ahlfors regular conformal dimension is a relatively new quasisymmetric invariant. It was introduced by Bourdon and Pajot \cite{BP} (see also Bonk and Kleiner~\cite{BK}), and is defined as follows:

\begin{defi}[Ahlfors regularity]
 Let $(X,d)$ be a metric space, $\mu$ be a Borel measure on $(X,d)$ and $\alpha>0$. We say $\mu$ is $\alpha$-Ahlfors regular with respect to $(X,d)$ if there exists $C>0$ such that 
 \[C^{-1}r^\alpha\le \mu(B_d(x,r)) \le Cr^\alpha\hspace{5pt}\text{ for any }x\in X\text{ and }r_x\le r\le\diam(X,d)\]  
 where $r_x=r_{x,d}=\inf_{y\in X\setminus\{x\}}d(x,y),$ and $B_d(x,r)$ is the open ball $\{y\in X\mid d(x,y)<r\}.$ The space $(X,d)$ is called $\alpha$-Ahlfors regular if there exists a Borel measure $\mu$ such that $\mu$ is $\alpha$-Ahlfors regular with respect to $(X,d).$
\end{defi}

\begin{defi}[Ahlfors regular conformal dimension]
  Let $(X,d)$ be a metric space. The Ahlfors regular conformal dimension (or ARC dimension in short) of $(X,d)$ is defined by 
 \begin{multline*}
  \ard(X,d)=\inf\{\alpha\mid \text{ there exists a metric $\rho$ on $X$}\\
  \text{such that $\rho$ is $\alpha$-Ahlfors regular and }d\qs \rho\},  
 \end{multline*}
 where $\inf\emptyset=\infty.$
\end{defi}

The ARC dimension is related to the conformal dimension, another well-known quasisymmetric invariant, as introduced by Pansu~\cite{Pan} in 1989. \\
 In this paper we will extend the notion of ARC dimension to discrete metric spaces. Note that the ARC dimension has mainly been studied on bounded metric spaces without isolated points, in which case $r_x=0$ and $\diam(X,d)<\infty$.\\
ARC dimension is related to the well-known Cannon's conjecture, which claims that for any hyperbolic group $G$ whose boundary is homeomorphic to 2-\hspace{0pt}dimensional sphere, there exists a discrete, cocompact and isometric action of $G$ on the hyperbolic space $\Mb{H}^3.$ Bonk and Kleiner~\cite{BK} proved Cannon's conjecture is equivalent to the following: If $G$ is a hyperbolic group whose boundary is homeomorphic to 2-dimensional sphere, then there exists a metric that attains the value of the ARC dimension of the boundary.\\
It is not easy to calculate the ARC dimension in general. Motivated by~\cite{Pia13,Pia14}, Kigami~\cite{Kig18} gave a method to calculate ARC dimension as a critical exponent of a $p$-energy, which is defined by successive division of the original metric space. Furthermore,~\cite{Kig18} gives inequalities between ARC dimension and $p$-spectral dimensions.\\
In this paper, we extend the results of~\cite{Kig18} to infinite graphs and give a relation between spectral dimensions and ARC dimensions. Our main results need a lot of notions, so we postpone detailed definitions to Sections 2 and 4, and explain the main results through examples.\\
\begin{figure}[tbp]
 \centering
\begin{tikzpicture}
	\foreach \y in {0,...,3}
	\draw (0,\y) -- (5.5,\y);
	\foreach \va in {0,...,5}
	\draw[white] (\va-0.05,1) -- (\va+0.05,1);
	\foreach \vb in {0,2,4}
	\draw[white] (\vb-0.05,2) -- (\vb+0.05,2);
	\foreach \vc in {0,4}
	\draw[white] (\vc-0.05,3) -- (\vc+0.05,3);
	
	\foreach \a in {1,3,5}
	\node (h\a) at (\a,1) {};
	\node (h2) at (2,2) {};
	\node (h4) at (4,3) {};
	\node (h0) at (0,4) {};
	\foreach \x in {0,...,5}
	\draw[dashed] (\x,0) -- (h\x);
	
	\foreach \t in {1,...,5}
	\node (a\t) [label=below:{(0,\t)}, fill=black, draw, circle, inner sep=1] at (\t-0.5,1){};
	\node (b1) [label=above left:{(1,1)}, fill=black, draw, circle, inner sep=1] at (1,2){};
	\node (b2) [label=above right:{(1,2)}, fill=black, draw, circle, inner sep=1] at (3,2){};
	\node (b3) [label=above left:{(1,3)}, fill=black, draw, circle, inner sep=1] at (5,2){};
	\node (c1) [label=above left:{(2,1)}, fill=black, draw, circle, inner sep=1] at (2,3){};
	\foreach \Saisu in {0,...,5} 
	\node (\Saisu) [label=below:\Saisu, fill=white, draw, circle, inner sep=1]at (\Saisu,0){};
	
	\draw (a1) -- (b1) -- (c1) -- (4,4);
	\draw (a2) -- (b1);
	\draw (a3) -- (b2) -- (c1);
	\draw (a4) -- (b2);
	\draw (a5) -- (b3) -- (5.5,2.5);
	\draw (b3) -- (5.45,1.1);
\end{tikzpicture}
 \caption{A partition of $\Mb{Z}_+$}
\label{figintro0}
\end{figure}
In our study, it will be useful to consider partitions of graphs that arise as edges are successively unified. One of the simplest cases is the unification of vertices of $\Mb{Z}_+=\{n\in\Mb{Z}\mid n\ge0\}.$ For $a\in\Mb{N}$ and $n\in\Mb{Z}_+,$ we identify $2^n$ edges $\{(2^n(a-1),2^n(a-1)+1), (2^n(a-1)+1,2^n(a-1)+2),\cdots,(2^na-1,2^na)\}=:K_{(n,a)}$ and consider unified graphs $\{G_n,E_n\}_{n\ge0}$ where $G_n=\{(n,a)\mid a\in\Mb{N}\}$ and $E_n$ is the set of links between $(n,a)$ and $(n,a+1).$
 Let $(n,a)\sim(m,b)$ if $n-m=1$ and $K_{(n,a)}\supseteq K_{(m,b)},$ or $m-n=1$ and $K_{(n,a)}\subseteq K_{(m,b)}.$ Consider $T:=\Cup_{n,a}(n,a)$ as a tree by $\sim,$ then we obtain a correspondence between $\{G_n,E_n\}_{n\ge0}$ and $T$ (see Figure \ref{figintro0}).  We call such a correspondence between unified graphs and a tree, a partition (see Definition \ref{defPG}, and note that we construct $K$ by unification of edges but we treat $K$ as a subsets of vertices because of technical reasons). \\
In this paper, we characterize the ARC dimension with a partition. For a given partition, we can define an upper $p$-energy $\ol{\Mc{E}}_p$ of the partition as a certain limit of $p$-energies on unification graphs, see Definition \ref{ciro}, which is based on definitions of~\cite{Kig18}. The $p$-energy enjoys a phase transition when $p$ varies, that is, there exists a $p_0>0$ such that $\ol{\Mc{E}}_p>0$ if $p<p_0$ and $\ol{\Mc{E}}_p=0$ if $p>p_0.$ We can also define a lower $p$-energy $\ul{\Mc{E}}_p,$ as well.\\
Our main result of this paper is the following. 
   
\begin{thm}[Theorem \ref{Smain1} (1)]\label{thintro1}
Let $(G,E)$ be a graph and $d$ be a metric on $G$. Under some conditions about $d$ and for partitions within a certain class,  
\[\ard(G,d)=\inf\{p\mid\ol{\Mc{E}}_p=0\}=\inf\{p\mid\ul{\Mc{E}}_p=0\}\]
\end{thm}
For detailed conditions, see Theorem \ref{Smain1}. Let us give an interesting example for ARC dimension for an unbounded metric space.

\begin{ex}\label{exintro1}
   Let $f(n):\Mb{Z}_+\to\Mb{Z}_+$ be such that $f(n)\le n$ for any $n.$  For $n\ge0,$ divide $[2^n,2^{n+1}]\times[0,2^n]$ into $2^{f(n)}\times 2^{f(n)}$ blocks and call them $G_n,$ and consider $G=\cup_{n\ge0}G_n\cup\{(0,0)\}$ as a subgraph of $\Mb{Z}^2$ (see Figure \ref{figintro1}, and Example~\ref{exthe1} 
for precise definition). 
\end{ex}

\begin{figure}[tbp]
 \centering
 \includegraphics[width=0.8\columnwidth]{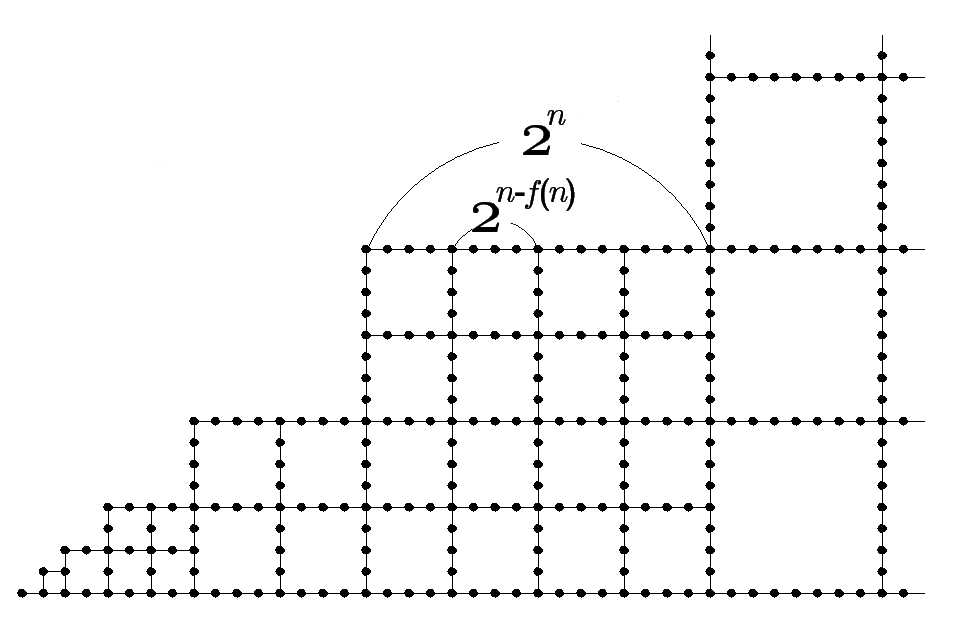}
 \caption{Example \ref{exintro1}}
\label{figintro1}
\end{figure}

Using Theorem \ref{thintro1}, we obtain the following:

\begin{prop}[Proposition \ref{dimgauge}]\hspace{0pt}\samepage
  \begin{enumerate}
  \item If  $\limsup_{n\to\infty} f(n)=\infty,$ then $\ard(G,d)=2.$
  \item If  $\limsup_{n\to\infty} f(n)<\infty,$ then $\ard(G,d)=1.$
  \end{enumerate}
\end{prop}

It is remarkable that only $\limsup_{n\to\infty} f(n)=\infty$ implies $\ard(G,d)=2$, although the size of boxes $2^{n-f(n)}$ may diverge.\\
We also compare the Ahlfors conformal dimension with the spectral dimension. For $p>0$, we can define the upper and lower $p$-spectral dimension, $\ol{d}^S_p$ and $\ul{d}^S_p$ of a partition (see Definition \ref{ciro2}).
We can further obtain the following.
\begin{thm}[Theorem \ref{Smain1}(2) and (3)]
Let $(G,E)$ be a graph and $d$ is a metric on $G$. Under the same conditions as in Theorem \ref{thintro1},  
\item If $\ard(G,d)<p,$ then
\[\ard(G,d)\le \ul{d}^S_p \le \ol{d}^S_p <p. \]
\item If $\ard(G,d)\ge p,$ then
\[\ard(G,d)\ge \ol{d}^S_p \ge \ul{d}^S_p \ge p. \]
\end{thm}
When $p=2$, the $p$-spectral dimension coincides in many example with the notion of the spectral dimension of random walks, where the latter is defined as follows:
\begin{align*}
\ol{d}_S(G)&= 2\limsup_{n\to\infty}\frac{-\log p_{2n}(x,x)}{\log n}, &\ul{d}_S(G)&= 2\liminf_{n\to\infty}\frac{-\log p_{2n}(x,x)}{\log n},
\end{align*}
where $p_n(x,y)$ is the transition density of the associated random walk. Hence, where this occurs, the latter theorem will also relate the ARC dimension and the spectral dimension of random walks. See Theorem \ref{Smain2} and Corollary \ref{Smain3} for a sufficient condition that the $2$-spectral dimension and the spectral dimension of random walks coincide. However, we prove that they can also be different - see Example \ref{exthe2} for an example where this is the case.\\

The outline of this paper is as follows. In Section 2, we introduce basic notation, and give the framework and results of ~\cite{Kig18} for compact spaces without isolated points, on which the main result of this paper is based. In Section 3, we extend the results of \cite{Kig18} to perfect $\sigma$-compact spaces. Section 4 is the main part of this paper, devoted to proving our main results. In Section 5, we give examples that illustrate properties of ARC dimension of graphs. In Section 6, we give a  proof of a result of a heat kernel estimate that is used in Section 4.

\subsection*{Acknowledgements}
I would like to thank Professor Takashi Kumagai, my master's thesis supervisor for fruitful discussions and checking earlier versions of this paper. I also thank Professor David Croydon for helpful comments and for checking my English of the introduction.\\
This work was supported by JSPS KAKENHI Grant Number JP20J23120. 

\section{Notation and Kigami's results for compact metric spaces}
As the preparation of this paper, we introduce notation used in this paper, and results of \cite{Kig18} on which the results of this paper are based. For this purpose, we first give notation of graphs.

\subsection{Basic notation}
We use the following notation in this paper.

\begin{itemize}
	\item Let $A$ be a set and $F$ be a map on $A$ to itself. Then $F^n$ denotes $\overbrace{F\circ\cdots\circ F}^{n}$ for any $n>0$ and $\mathrm{id}_A$ for $n=0.$ Moreover, $A^n$ denotes the product $\overbrace{A\times\cdots\times A}^n .$
	\item Let $\{A_\lambda\}_{\lambda\in\Lambda}$ be a family of sets, then $\sqcup_{\lambda\in\Lambda}A_\lambda$ denotes $\cup_{\lambda\in\Lambda}A_\lambda$ with $A_\lambda\cup A_\tau=\emptyset$ for any $\lambda,\tau\in\Lambda$ with $\lambda\ne\tau.$ 
	\item Let $f$ and $g$ be functions with variables $x_1,...,x_n$. We use ``$f\asymp g$ for any $(x_1,...x_n)\in A$'' if there exists $C>0$ such that
	\begin{equation*}
		C^{-1}f(x_1,...,x_n)\le g(x_1,...,x_n) \le C f(x_1,...x_n)
	\end{equation*}
	for any $(x_1,...,x_n)\in A.$
	\item Let $(X,\Mc{O})$ be a topological space and $A\subseteq X,$ then $A^o$ (resp. $A^c$) denote the interior (resp. complement).  
	\item Let $(X,d)$ be a metric space and $\mu$ is a (Borel) measure on $(X,d)$, then we write
	\[ B_d(x,r)=\{y|d(x,y)<r\} \text{ and } V_{d,\mu}(x,r)=\mu(B_d(x,r)). \]
	Moreover, let $\rho$ be a distance on $X$, then we write
	\[ \overline{d}_\rho(x,r)=\sup_{y\in B_\rho(x,r)}d(x,y).\]
	We also use the notation $\diam(X,d)=\sup\{d(x,y)\mid x,y\in X\}.$
	If no confusion may occur, we omit $d,\rho$ or $\mu$ in these notation.
	\item We also use the notation $[n,m]_\Mb{Z}=\{k\in\mathbb{Z}\mid n\le k\le m\}$ for $n,m\in\mathbb{Z},$ $a\vee b=\max\{a,b\}$ and $\ a\wedge b=\min\{a,b\}.$  
\end{itemize}
In the following, whenever $d\qs\rho$ for metrics $d,\rho$ on a space, $\theta$ denotes a homeomorphism such that $d$ is $\theta$-quasisymmetric to $\rho$, if no confusion may occur.\\   
We will use the same notation to a vertex and its equivalent class.

The following definitions are basic notions of graph, but we write that for confirmation of our notations.

\begin{defi}[Graph, tree]
Let $T$ be an (at most) countable set and let $\Mc{A}\subseteq T\times T$ such that 
\begin{itemize}
\item for any $w\in T,$ $(w,w)\not\in \Mc{A}.$
\item $(w,v)\in A$ if $(v,w)\in \Mc{A}.$
\end{itemize}
We call $(T,A)$ a simple graph. We write $w\sim v$ if $(w,v)\in\Mc{A}.$
\begin{enumerate}
\item $(T,A)$ is called locally finite if $\#(\{y\mid y\sim x\})<\infty$ for any $x\in T.$ $(T,A)$ is called bounded degree if $\sup_{x\in T}\#(\{y\mid y\sim x\})<\infty.$ 
\item Let $n\ge0.$ We call $(w_0,w_1,...,w_n)\in T^n$ a $n$-path (between $w_0$ and $w_n$) if $w_i\sim w_{i-1}$ for any $i\in [1,n]_\Mb{Z}.$ Especially, we call $(w_0,w_1,...,w_n)$ a $n$-simple path (between $w_0$ and $w_n$) if it is a $n$-path and $w_i\ne w_j$ whenever $i\ne j.$ \\$(w_0,w_1,...w_n)$ is called a path if it is a $n$-path for some $n\ge0,$ and called a simple path if it is a $n$-simple path for some $n\ge0.$ 
\item We call $(T,\Mc{A})$ connected if there exists a path between $w$ and $v$ for any $w,v\in T.$ Moreover, we call $(T,\Mc{A})$ a tree if there exists an unique simple path between $w$ and $v$ for any $w,v\in T.$ 
\item Let $(T,\Mc{A})$ be a simple graph. We define $l_\Mc{A}$ by 
\[l_\Mc{A}(w,v)=\min\{n\mid\text{ there exists an }n\text{-path between }w\text{ and }v\}.\]
If $(T,\Mc{A})$ is connected, then $l_\Mc{A}$ is called the graph metric of $(T,\Mc{A}).$
\end{enumerate}
\end{defi}
In this paper, we will consider only simple graphs.

\begin{defi}[Rooted tree]
Let $(T,\Mc{A})$ be a tree and $\phi\in T.$ We call the triple $(T,\Mc{A},\phi)$ a rooted tree.
\begin{enumerate}
\item Define $|w|=l_\Mc{A}(\phi,w)$ and $(T)_n=\{w\mid|w|=n\}$ for any $n\ge 0,$ and define $\pi:T\times T$ by  
\[\hspace{-1pt} \pi(w)=\pi_{(T,A,\phi)}(w)=\begin{cases}
 w_{n-1},\hspace{-4pt}& \text{if }w\ne\phi\text{ and }(\phi=w_0,...,w_{n-1},w_n=w)\text{ is}\\
&\text{the unique simple path between }\phi\text{ and }w,\\
 \phi, & \text{if }w=\phi.
\end{cases} \]
S denote the inverse of $\pi$ (excluding $\phi$), namely
\[S(A)=\{w\in T\setminus\{\phi\}\mid\pi(w)\in A\} \]
for any $A\subseteq X,$ and we write $S(w)$ instead of $S(\{w\}).$ Moreover, we define subtree $T_w=\{v\in T\mid\pi^n(v)=w\text{ for some }n\ge0\}.$ 
\item Define geodesics of $T$ (from $\phi$) by  
\[\Sigma=\{\omega=(\omega_n)_{n\ge0}\mid\omega_n\in (T)_n,\ \pi(\omega_{i+1})=\omega_i\text{ for all }i\ge0\},\]
and geodesics passing through $w$ by $\Sigma_w=\{\omega\in\Sigma|\omega_{|w|}=w\}$ for any $w\in T.$
\end{enumerate}
\end{defi}

\begin{rem}
In \cite{Kig18}, Kigami used the terminology "tree with a reference point" instead of "rooted tree". However, we want to use the term ``a reference point'' to distinguish trees, which have no root but have a reference point of height, introduced in Section 3.
\end{rem}

\subsection{Kigami's results for compact metric spaces}
Throughout this section, $\Mc{T}=(T,\Mc{A},\phi)$ is a locally finite rooted tree. 

\begin{defi}[Partition]\label{copa}
Let $(X,\Mc{O})$ be a compact metrizable space having no isolated points, and let $\Mc{C}(X,\Mc{O})$ be the collection of nonempty compact subsets of $(X,\Mc{O})$ without single points. A map $K:T\to\Mc{C}(X,\Mc{O})$, where we write $K_w$ instead of $K(w)$ for ease of notation, is called a partition of $(X,\Mc{O})$ parametrized by $\Mc{T}$ if it satisfies the following conditions.
\begin{itemize}
\item[(P1)] $K_\phi=X$ and for any $w\in T,$
\[ \Cup_{v\in S(w)}K_v=K_w.\]
\item[(P2)] For any $\omega\in\Sigma,$ $\cap_{m\ge0}K_{\omega_m}$ is a single point.
\end{itemize}
\begin{enumerate}
\item Let $K$ be a partition of $X.$ We define $O_w$ by
\[O_w=K_w\setminus\left(\Cup_{v:v\in(T)_{|w|}\setminus\{w\}}K_v\right).\]
$K$ is called minimal if $O_w\ne\emptyset$ for any $w\in T.$ 
\item For $m\ge0,$ we define $E^h_m\subseteq (T)_m\times(T)_m$ by 
\[J^h_m=\{(w,v)\mid w,v\in(T)_m,\ w\ne v\text{ and }K_w\cap K_v\ne\emptyset \} \]
and $\Gamma_n(w)=\{v\in (T)_{|w|}\mid l_{J^h_m}(w,v)\le n\}$ for any $w\in T.$
\end{enumerate}
We simply write $X$ for $(X,\Mc{O})$ in notation if no confusion may occur.
\end{defi}

\begin{defi}[Weight function]\label{cowe}
A function $g:T\to(0,1]$ is called a weight function if it satisfies the following conditions.
\begin{itemize}
\item[(G1)] $g(\phi)=1.$
\item[(G2)] For any $w\in T,$ $g(\pi(w))\ge g(w).$
\item[(G3)] For any $\omega\in\Sigma,$ $\lim_{m\to\infty}g(\omega_m)=0.$
\end{itemize}
\begin{enumerate}
\item For $s>0,$ we define the scale $\Lambda^g_t$ associated $g$ by 
\[ \Lambda^g_s=\begin{cases}
\{w\in T| g(w)\le s< g(\pi(w))\}, & \text{ if }0<s<1, \\
\{\phi\}, & \text{ otherwise.}
\end{cases}\]
We also define $E^g_s\subseteq \Lambda^g_s\times\Lambda^g_s$ by
\[E^g_s=\{(w,v)\mid w,v\in\Lambda^g_s,\ w\ne v\text{ and }K_w\cap K_v\ne\emptyset \}. \]
\item For $x\in X,$ $s>0,$ $M\ge0$ and  $w\in\Lambda^g_s,$ we define
\begin{align*}
\Lambda^g_{s,M}(w)&=\{v\in \Lambda^g_s\mid l_{E^g_s}(w,v)\le M\},&  \Lambda^g_{s,M}(x)&=\Cup_{w\in\Lambda^g_s\text{ and }x\in K_w}\Lambda^g_{s,M}(w),
\end{align*}
and
\[U^g_M(x,s)=\Cup_{w\in\Lambda^g_{s,M}(x)}K_w.\]
\end{enumerate}
\end{defi}

\begin{defi}
 Let $(X,\Mc{O})$ be a compact metrizable space having no isolated point, and $K$ be a partition of $X$. Define
\begin{multline*}
\Mc{D}(X,\Mc{O})=\{d\mid d\text{ is a metric on }X\text{ inducing the topology }\Mc{O}\text{ and }\\
\diam(X,d)=1\}.
\end{multline*}
For $d\in\Mc{D}(X,\Mc{O})$, define $g_d:T\to(0,1]$ by $g_d(w)=\diam(K_w,d)$ for any $w\in T.$
\end{defi}
\begin{prop}[\cite{Kig18},Proposition 5.5(1)]\label{gd}
Let $(X,\Mc{O})$ be a compact metrizable space having no isolated point and $K$ be a partition of $X$. For any $d\in\Mc{D}(X,\Mc{O}),$ $g_d$ is a weight function.
\end{prop}

\begin{proof}
\begin{itemize}
\item[(G1)] By (P1), $g_d(\phi)=\diam(X,d)=1.$
\item[(G2)] By (P1), $K_{\pi(w)}\supseteq K_w$ and hence 
\[g_d(\pi(w))=\diam(K_{\pi(w)},d)\ge\diam(K_w,d)=g_d(w).\]
\item[(G3)] Since $\{g_d(\omega_n)\}_{n\ge0}$ is decreasing, it must converge if $n\to\infty.$ Assume $\lim_{m\to\infty}g_d(\omega_m)=c>0$ for some $\omega\in\Sigma.$ Then there exist $\{x_m\}_{m\ge0}$ and $\{y_m\}_{m\ge0}$ such that $x_m,y_m\in K_{\omega_m}$ and $d(x_m,y_m)\ge c$ because (G2) holds. We can take convergent subsequences $\{x_{m_k}\}_{k\ge0},\ \{y_{m_k}\}_{k\ge0}$ because $X$ is compact. Since $\{K_{\omega_m}\}_{m\ge0}$ is a decreasing sequence of compact sets and $\{x_{m_k}\}_{k\ge n},\{y_{m_k}\}_{k\ge n}\subseteq K_{\omega_{m_n}}$ for any $n,$ it follows that $x,y,\in\cap_{m\ge0}K_{\omega_m}$ where $x=\lim_{k\to\infty}x_{m_k},\ y=\lim_{k\to\infty}y_{m_k}.$
Since (P2) holds, $x=y$ and hence
\[0<c\le\lim_{k\to\infty}d(x_{m_k},y_{m_k})=d(x,y)=0.\]
This is a contradiction.
\end{itemize}
\end{proof}

\begin{ex}[Sierpi\'{n}ski carpet] \label{exSC}
	Let $\{p_i\}_{i=1}^8\subseteq\Mb{C}$ such that
	\begin{align*}
		p_1&=0, & p_2&=\frac{1}{2}, & p_3&=1, & p_4&=1+\frac{1}{2}i, \\
		p_5&=1+i, & p_6&=\frac{1}{2}+i, & p_7&=i, &p_8&=\frac{1}{2}i,
	\end{align*}
	and let $F_i=\frac{1}{3}(z-p_i)+p_i$ for any $i\in[1,8]_\Mb{Z}.$ It is well-known that there exists the unique compact set $X$ such that $\cup_{i=1}^8F_i(X)=X,$ called Sierpi\'{n}ski carpet.
	Let $T=\cup_{n\ge0}([1,8]_\Mb{Z})^n$ where $([1,8]_\Mb{Z})^0=\{\phi\}$ and define $\pi:T\setminus\{\phi\}\to T$ by 
	\[\pi(w)=
	\begin{cases}
		(w_1,w_2,...,w_{n-1}),& \text{ if }w=(w_1,w_2,...w_{n-1},w_n)\in\Cup_{n\ge2}([1,8]_\Mb{Z})^n, \\
		\phi, & \text{ if }w\in[1,8]_\Mb{Z}.
	\end{cases}\]
	We also let $\Mc{A}=\{(w,v)\mid w=\pi(v)\text{ or }v=\pi(w)\},$ then $\Mc{T}=(T,\Mc{A},\phi)$ is a rooted tree. Moreover, for $w\in T$ define $F_w:\Mb{C}\to\Mb{C}$ by 
	\[F_w=
	\begin{cases}
		F_{w_1}\circ F_{w_2}\circ\cdots\circ F_{w_n},& \text{ if }w=(w_1,w_2,...,w_n)\in\Cup_{n\ge1}[1,8]_\Mb{Z}^n, \\
		\id_\Mb{C}, & \text{ if }w=\phi,
	\end{cases}\]
	and define $K:T\to\Mc{C}(X,\Mc{O})$ by $K_w=F_w(X).$ Then $K$ is a partition of $X$ parametrized by $\Mc{T}.$ We also let $d(z,w)=\frac{\sqrt{2}}{2}|z-w|,$ then $d\in \Mc{D}(X,\Mc{O})$ and by Proposition \ref{gd}, $g_d$ is a weight function. 
\end{ex}	
	
We denote $g_d$ by $d$ if no confusion may occur. For example, we will use notation $U^d_M(x,r)$ and $d(w)$ instead of $U^{g_d}_M(x,r)$ and $g_d(w),$ respectively.  \\
For the purpose of stating the main result of~\cite{Kig18}, we introduce some properties of weight functions and metrics.\\
 For the rest of this section, $(X,\Mc{O})$ is a compact metrizable space and $d\in\Mc{D}(X,\Mc{O})$ (in other words, $(X,d)$ is a compact metric space with $\diam(X,d)=1,$ and $\Mc{O}$ is the induced topology). Moreover, $K$ is a partition of $X.$

\begin{defi}\label{cobf}
Let $g$ be a weight function.
\begin{itemize}
\item $g$ is called uniformly finite if 
\begin{equation}\label{uf}
\sup\{\#(\Lambda^g_{s,1}(w))\mid s>0,w\in\Lambda^g_s\}<\infty.
\end{equation}
\item $g$ is called thick (with respect to $K$) if there exists $\alpha>0$ such that for any $w\in T,$ $U^g_1(x,\alpha g(\pi(w)))\subseteq K_w$ for some $x\in K_w.$
\end{itemize}
$d$ is called 
uniformly finite and thick if $g_d$ is 
uniformly finite and thick, respectively.
\end{defi}

\begin{defi}\label{coada}
Let $M\ge1.$ $d$ is called $M$-adapted if there exists $\alpha_1,\alpha_2>0$ such that
\[B_d(x,\alpha_1 r)\subseteq U^{d}_M(x,r)\subseteq B_d(x,\alpha_2 r)\]
for any $x\in X$ and $r\le1.$ $d$ is called adapted if $d$ is $M$-adapted for some $M.$
\end{defi}

\begin{ex}[Sierpi\'{n}ski carpet] \label{exSC2}
Let $X$ is the Sierpi\'{n}ski carpet and define $T,K,F$ and $d$ as in Example \ref{exSC}. We can see that $d(w)=\diam(F_w(X),d)=3^{-|w|},$ and hence $\Lambda^d_s=(T)_m$ for any $s\in(0,1)$ and $m\ge1$ such that $3^{-m}\le s<3^{-(m-1)}.$ 
\begin{itemize}
\item $d(\pi(w))=3d(w)$ for any $w\in T\setminus\{\phi\}.$  
\item If $K_w\cap K_v\ne\emptyset$ and $|w|=|v|,$ then 
\begin{equation}
K_v\in\{K_w+a3^{-|w|}+b3^{-|w|}i\mid (a,b)\in ([-1,1]_\Mb{Z})^2\setminus\{(0,0)\}\} \label{eqSC}
\end{equation}
where $K_w+z=\{x+z\mid x\in K_w\}.$ Therefore $\Lambda^d_{s,1}(w)\le 8$ for any $s\in(0,1)$ and $w\in\Lambda^d_s,$ and hence $d$ is uniformly finite.
\item Let $s\in(0,1)$ and $w\in\Lambda^d_s.$ We also let $x_w=F_w(\frac{2}{3}+\frac{2}{3}i).$ Then $\Lambda^d_{3^{-2}s}=(T)_{|w|+2}$ together with \eqref{eqSC} shows $\Lambda^d_{3^{-2}s,1}(x_w)\subseteq B_d(x_w,\frac{\sqrt{2}}{2}3^{-|w|-1})\subseteq K_w.$ Therefore $d$ is thick.
\item By Equation \eqref{eqSC} and the definition of $K,$ 
\[B_d(x,\frac{\sqrt{2}}{6}s)\subseteq B_d(x,\frac{\sqrt{2}}{2}3^{-m})
\subseteq U^d_1(x,s) \subseteq B_d(x,3\cdot 3^{-m}) \subseteq B_d(x,3s),\]
for any $x\in X,s\in(0,1)$ and $m\ge1$ such that $3^{-m}\le s\le3^{-(m-1)}.$ Therefore $d$ is ($1$-)adapted.
\end{itemize}
\end{ex}

\begin{defi}\label{ciro}
For any $N\ge1,$ $N_2\ge N_1\ge 0$ and $m\ge0,$ define 
\[J^h_{N,m}=\{(w,v)\mid w,v\in (T)_m \text{ such that }0<l_{J^h_m}(w,v)\le N \},\]
and
\begin{multline*}
  \epkw{w}=\inf\biggl\{ \frac{1}{2}\sum_{(x,y)\in J^h_{N,|w|+k}}|f(x)-f(y)|^p\\ \biggl|\ f:(T)_{|w|+k}\to\Mb{R}
\text{ such that }f|_{S^k(\Gamma_{N_1}(w))}\equiv 1,\ f|_{(S^k(\Gamma_{N_2}(w)))^c}\equiv 0 \biggr\}.
\end{multline*}
(Remark that $J^h_{1,m}=J^h_m$.) We also define
\begin{align*}
\mathcal{E}_{p,k}(N_1,N_2,N)&=\sup_{w\in T}\epkw{w},\\ 
\ol{\Mc{E}}_p(N_1,N_2,N)&=\limsup_{k\to\infty}\mathcal{E}_{p,k}(N_1,N_2,N),\\ 
\ul{\Mc{E}}_p(N_1,N_2,N)&=\liminf_{k\to\infty}\mathcal{E}_{p,k}(N_1,N_2,N),\\
\ol{I}_\Mc{E}(N_1,N_2,N)&=\inf\{p|\ol{\Mc{E}}_p(N_1,N_2,N)=0\},\\
\ul{I}_\Mc{E}(N_1,N_2,N)&=\inf\{p|\ul{\Mc{E}}_p(N_1,N_2,N)=0\}.
\end{align*}
\end{defi}

As in~\cite{Kig18}, we define the basic framework as follows.
\begin{defi}\label{defBF1}
Let $(X,d)$ be a metric space and assume $K$ is minimal.  We say $d$ satisfies the basic framework (with respect to $K$) if the following conditions hold.
\begin{itemize}
\item $\sup_{w\in T}\#(S(w))<\infty.$
\item $d$ is uniformly finite, thick, adapted.
\item There exists $r\in(0,1)$ such that $d\asymp r^{|w|}$ for any $w\in T.$  
\end{itemize}
For the rest of this paper, $M_*$ denotes the integer such that $d$ is $M_*$-adapted.
\end{defi}
\begin{thm}[\cite{Kig18}, Theorem 19.4 and 19.9.]\label{Kmain1}
 If $K$ is minimal and $d$ satisfies the basic framework, then
\[\ul{I}_\Mc{E}(N_1,N_2,N)=\ol{I}_\Mc{E}(N_1,N_2,N)=\ard(X,d)\]
for any $N,N_1,N_2$ with $N_2\ge N_1+M_*.$ 
\end{thm}
\color{black}
%
%

As a corollary of Theorem \ref{Kmain1}, we get the result for comparison between $\ard$ and "p-spectral" dimension, defined for any $p>0$, as follows.

\begin{defi}[$p$-spectral dimension]\label{ciro2}
Define
\begin{align*}
\ol{N}_*&=\limsup_{k\to\infty}\sup_{w\in T} \#(S^k(w))^{1/k},&
\ol{R}_p(N_1,N_2,N)&=\limsup_{k\to\infty}\Mc{E}_{p,k}(N_1,N_2,N)^{1/k},\\
&&\ul{R}_p(N_1,N_2,N)&=\liminf_{k\to\infty}\Mc{E}_{p,k}(N_1,N_2,N)^{1/k},
\end{align*}
and upper $p$-spectral dimension $\ol{d}^S_p(N_1,N_2,N)$ and lower $p$-spectral dimension $\ul{d}^S_p(N_1,N_2,N)$ by
\begin{align*}
\ol{d}^S_p(N_1,N_2,N)&=\frac{p\log\ol{N}_*}{\log\ol{N}_*-\log\ol{R}_p(N_1,N_2,N)},\\
\ul{d}^S_p(N_1,N_2,N)&=\frac{p\log\ol{N}_*}{\log\ol{N}_*-\log\ul{R}_p(N_1,N_2,N)}.  
\end{align*}
\end{defi}

\begin{cor}[\cite{Kig18} Theorem 20.8]
Assume $K$ is minimal and $d$ satisfies the basic framework. \color{black} Let $N,N_1,N_2\ge0$ such that $N_2\ge N_1+M_*.$ \color{black}
\begin{enumerate}
\item If $\ul{R}_p(N_1,N_2,N)<1,$ then
\[\ard(X,d)\le \ul{d}^S_p(N_1,N_2,N) \le \ol{d}^S_p(N_1,N_2,N) <p. \]
\item If $\ol{R}_p(N_1,N_2,N)\ge 1,$ then
\[\ard(X,d)\ge \ol{d}^S_p(N_1,N_2,N) \ge \ul{d}^S_p(N_1,N_2,N) \ge p. \]
\end{enumerate}
\end{cor}

\begin{rem}
Originally in~\cite{Kig18}, these theorem are also proved for some broader framework, called proper system of horizontal networks.    
\end{rem}
\section{Extension to $\sigma$-compact metric spaces}
In order to obtain the former results for infinite graphs, we first extend theorems to $\sigma$-compact spaces. To do that, we introduce the notion of a bi-infinite tree.

\begin{defi}
 Let $T$ be a countable set and $\pi_T:T\rightarrow T$ be a map which satisfies the following.
 \begin{itemize}
  \item[(T1)] $\text{For any }w,v \in T, \text{ there exists } n,m \ge 0, \pi^n(w)=\pi^m(v).$
  \item[(T2)] $\text{For any }n\ge 1\text{ and } w\in T, \pi^n(w)\ne w.$
 \end{itemize}
 Then we define $\Mc{A}_{\pi}=\{(w,v)\mid \pi(w)=v\text{ or }\pi(v)=w\}$ and consider the simple graph $(T,\Mc{A}_\pi).$ We call $(T,\pi)$ a bi-infinite tree.\\
Same as a rooted tree, we denote the inverse of $\pi$ by $S$ and write $S(w)$ instead of $S(\{w\}).$ Moreover, we define subtree $T_w=\{v\in T|\pi^n(v)=w\text{ for some }n\ge0\}.$ 
\end{defi}

\begin{lem}\label{hstruc}
 Let $(T,\pi)$ be a bi-infinite tree and fix $w,v \in T.$ 
 Then $n-m$ is constant if $\pi^n(w)=\pi^m(v).$
\end{lem}

\begin{proof}
 Assume $\pi^{n_i}(w)=\pi^{m_i}(v)(i=1,2).$ Without loss of generality, we may assume $n_1\le n_2$, then
 \[ \pi^{m_2}(v)=\pi^{n_2}(w)=\pi^{n_2-n_1}\circ\pi^{n_1}(w)=\pi^{n_2-n_1}\circ\pi^{m_1}(v). \]
This together with (T2) shows $m_2=(n_2-n_1)+m_1;$ it means $ m_2-n_2=m_1-n_1.$
\end{proof}

\begin{defi}
 Let $(T,\pi)$ be a bi-infinite tree.
\begin{enumerate}
\item Let $\phi \in T$. We call the triple $(T,\pi,\phi)$ a bi-infinite tree with a reference point $\phi$,  and for any $w\in T$ we define the height of vertices by $[w]=[w]_\phi=n-m$ such that $\pi^n(w)=\pi^m(\phi)$. We also define $(T)_n$ by $(T)_n=\{w\in T|[w]=n\}$ for any $n\in\mathbb{Z}.$ 
\item Define (descending) geodesics of $T$ by  $\Sigma^*=\{\omega=(\omega_n)_{n\in\Mb{Z}}\mid\omega_n\in (T)_n,\ \pi(\omega_{n+1})=\omega_n\text{ for all }n\in\Mb{Z} \}$ and geodesics passing through $w$ by $\Sigma^*_w=\{\omega\in\Sigma^*|\omega_{[w]}=w\}$ for any $w\in T.$
 \end{enumerate}
\end{defi}
 
\begin{rem}
 (T1) and Lemma \ref{hstruc} ensure that $[w]$ is well-defined. Moreover, for fixed $w,v \in T, [w]_\phi-[v]_\phi$ is constant for every $\phi \in T$ by Lemma \ref{hstruc} (that is, the difference of the height of a bi-infinite tree is determined only by $\pi$ and does not depend on its reference point).
\end{rem}

As the name shows, a bi-infinite tree is a tree.

\begin{prop}
 Let $(T,\pi)$ be a bi-infinite tree, then $(T,\Mc{A}_\pi)$ is a tree.
\end{prop}

\begin{proof}
 For all $w,v\in T$ there exists a path between them by (T1), and also a simple path exists.\\
Next we prove the uniqueness of a simple path. Fix any reference point $\phi\in T$ and think $(T,\mathcal{A},\phi).$ By definition, if $w=\pi(v)$ then $[w]=[v]-1$. Let $(w_0,w_1,...,w_n)$ be a simple path. If $[w_i]>[w_{i+1}]$, then
 $\pi(w_i)=w_{i+1}$ and so $w_{i-1}\ne\pi(w_i),$ it means $[w_{i-1}]>[w_i]$. In the same way, we can see $[w_{i+1}]>[w_i]$ if $[w_i]>[w_{i-1}]$.\\
 Now we let $(w,\pi(w),...,\pi^{n_1}(w)=\pi^{m_1}(v),...,\pi(v),v)$ and $(w,\pi(w),...,\pi^{n_2}(w)=\pi^{m_2}(v),...,\pi(v),v)$ be simple paths. If $n_1<n_2$ then $m_1<m_2$ by Lemma \ref{hstruc}, so the latter simple path take $\pi^{n_1}(w)=\pi^{m_1}(v)$ two times, which is contradiction. The case $n_1>n_2$ is the same. If $n_1=n_2$ then $m_1=m_2$ by Lemma \ref{hstruc}, so the paths are equal. Therefore $(T,\mathcal{A})$ is a tree.
\end{proof}

The root of a bi-infinite tree does not exist, but ``$\pi^\infty(w)$'' is thought to be a virtual root. \\
Now we extend notions of partitions and weight functions to $\sigma$-compact spaces. For the rest of this section, let $\Mc{T}=(T,\pi,\phi)$ be a locally finite bi-infinite tree with a reference point. Note that $(T,\pi)$ is locally finite if and only if $\#(S(w))<\infty$ for any $w\in T.$

\begin{defi}[Partition]
Let $(X,\Mc{O})$ be a $\sigma$ compact metrizable space having no isolated points, and let $\Mc{C}(X,\Mc{O})$ be the collection of nonempty compact subsets of $(X,\Mc{O})$ without single points. A map $K:T\to\Mc{C}(X,\Mc{O})$, where we denote $K(w)$ by $k_w$ for ease of notation, is called a partition of $(X,\Mc{O})$ parametrized by $\Mc{T}$ if it satisfies the following conditions.
\begin{itemize}
\item[(P1)]For any $w\in T,$
\[ \Cup_{v\in S(w)}K_v=K_w.\]
\item[(P2)] For any $\omega\in\Sigma^*,$ $\cap_{m\ge0}K_{\omega_m}$ is a single point.
\item[(P3)] $\cup_{w\in (T)_0}K(w)=X.$
\end{itemize}
We say a partition $K$ is locally finite if it satisfies the following:
 \begin{multline*}
  \text{for any }w\in (T)_0,\text{ there exists an open set }U_w \\
  \text{ which satisfies } K_w\subset U_w \text{ and } \#\{v\in (T)_0|K_v\cap U_w\ne\emptyset\}<\infty. 
\end{multline*}
We define $O_w,J^h_M,\Gamma_n(w)$ and minimality in the same way as compact cases, and similarly use $X$ instead of $(X,\Mc{O}).$
\end{defi}

\begin{rem}
(P3) is the counterpart of $K_\phi=X$ in Definition \ref{copa}. The locally finiteness of a partition is used to reduce local properties of partitions of $\sigma$-compact spaces to compact cases.  
\end{rem}

\begin{defi}[Weight function]
A function $g:T\to(0,\infty)$ is called a weight function if it satisfies the following conditions.
\begin{itemize}
\item[(G1)] $\lim_{n\to\infty}g(\pi^n(\phi))=\infty.$
\item[(G2)] For any $w\in T,$ $g(\pi(w))\ge g(w).$
\item[(G3)] For any $\omega\in\Sigma^*,$ $\lim_{m\to\infty}g(\omega_m)=0.$
\end{itemize}
For $s>0,$ we define the scale $\Lambda^g_s$ associated $g$ by 
\[ \Lambda^g_s=\{w\in T| g(w)\le s< g(\pi(w))\},\]
and define $E^g_s,\Lambda^g_{s,M}(w),\Lambda^g_{s,M}(x),U^g_M(x,s)$ in the same way as Definition \ref{cowe}.
\end{defi}

%

\begin{defi}
 Let $(X,\Mc{O})$ be a $\sigma$-compact metrizable space having no isolated points and $K$ be a partition of $X$.
 \begin{enumerate}
 \item Define
\begin{multline*}
\Mc{D}_\infty(X,\Mc{O})=\{d\mid d\text{ is a metric on }X\text{ inducing the topology }\Mc{O}\text{ and }\\
\diam(X,d)=\infty\}.
\end{multline*}
For $d\in\Mc{D}_\infty(X,\Mc{O})$, define $g_d:T\to(0,\infty)$ by $g_d(w)=\diam(X,d)$ for any $w\in T.$
\item Define
\begin{multline*}
\Mc{M}_\infty(X,\Mc{O})=\{\mu\mid\mu\text{ is a Radon measure on }(X,\Mc{O})\text{ such that }\\
\mu(X)=\infty,\mu\{x\}=0\text{ for any }x\in X\text{ and }\mu(K_w)>0\text{ for any }w\in T\}.
\end{multline*}
For $\mu\in\Mc{M}_\infty(X,\Mc{O}),$ define $g_\mu:T\to(0,\infty)$ by $g_\mu(w)=\mu(K_w)$ for any $w\in T.$
\end{enumerate}
\end{defi}

\begin{prop}\label{propgdgmu}
 Let $(X,\Mc{O})$ be a $\sigma$-compact metrizable space having no isolated point and $K$ be a partition of $X.$
\begin{enumerate}
\item For any $d\in\Mc{D}_\infty(X,\Mc{O}),$ $g_d$ is a weight function.
\item For any $\mu\in\Mc{M}_\infty(X,\Mc{O}),$ $g_\mu$ is a weight function.
\end{enumerate}
\end{prop}

\begin{proof}
(1) Fix $x\in K_\phi$ and any $r>0.$ Since $\diam(X,d)=\infty,$ there exists $y\in X$ such that $d(x,y)>r$, and that there exists $v\in (T)_0$ such that $y\in K_v$ because of (P3). Since (T2) holds, there exists $n_0\ge 0$ such that  $\pi^{n_0}(\phi)=\pi^{n_0}(v)$, and then $g_d(\pi^n(\phi))\ge g_d(\pi^{n_0}(\phi))\ge r$ for any $n\ge n_0$, so (G1) holds. (G2) and (G3) can be proved in the same way as in the proof of Proposition \ref{gd}.\\
(2) Since $\mu$ is locally finite, so (2) immediately follows from general property of a measure and (P2). 
\end{proof}

We denote $g_d$ by $d$ if no confusion may occur. For the rest of this section, $(X,\Mc{O})$ is a $\sigma$-compact metrizable space having no isolated points and $d\in\Mc{D}_\infty(X,\Mc{O}).$  Moreover, $K$ is a partition of $X.$ \\
We introduce properties of $d$ and a weight function $g$ in the same way as Definition \ref{cobf}, \ref{coada} and \ref{defBF1}, and also introduce variables in the same way as Definition \ref{ciro} and \ref{ciro2}. By using bi-infinite tree and these properties, we can extend the theory of~\cite{Kig18} to $\sigma$-compact spaces. In particular, we get the following result.

\begin{thm}\label{sigmain1}
 Assume  $K$ is locally finite and minimal. If $d$ satisfies the basic framework, then for any $N,N_1,N_2$ \color{black}such that $N_2\ge N_1+M_*,$ \color{black}
\begin{enumerate}
\item
\[\ul{I}_\Mc{E}(N_1,N_2,N)=\ol{I}_\Mc{E}(N_1,N_2,N)=\ard(X,d).\]
\item If $\ul{R}_p(N_1,N_2,N)<1,$ then
\[\ard(X,d)\le \ul{d}^S_p(N_1,N_2,N) \le \ol{d}^S_p(N_1,N_2,N) <p. \]
\item If $\ol{R}_p(N_1,N_2,N)\ge 1,$ then
\[\ard(X,d)\ge \ol{d}^S_p(N_1,N_2,N) \ge \ul{d}^S_p(N_1,N_2,N) \ge p. \]
\end{enumerate}
\end{thm}

\begin{proof}
\color{black}Most of the results in \cite{Kig18} do not use the property $T=T_\phi$ nor compactness of $X.$ Therefore if $K$ is locally finite, we can prove $\sigma$-compact version for most of the statements in \cite{Kig18} line by line in the same way. However, modifications of the statements and proofs are required for the following results:

\begin{itemize}
\item Lemma 4.2, 18.3,
\item Proposition 3.8, 4.9(7.2), 5.2, 5.5, 5.7, 5.9, 
\item Theorem 4.7, 7.9, 7.12, 12.9, 15.2, 18.1,
\item Corollary 7.13.
\end{itemize}

Note that for proving our Theorem \ref{sigmain1} (included in  $\sigma$-compact versions of~\cite[Theorem 19.9 and Theorem 20.8 ]{Kig18}), we do not need counterparts of~\cite[Theorem 4.7, Theorem 4.9, Theorem 12.9]{Kig18} and statements in~\cite[Section 7]{Kig18}, so we should prove counterparts of the others. Moreover, we have proved the counterpart of~\cite[Proposition 5.5]{Kig18} as Proposition \ref{propgdgmu}. Recall that $\Mc{T}=(T,\pi,\phi)$ is a bi-infinite tree, $(X,\Mc{O})$ is a $\sigma$-compact space having no isolated points, $K$ is a partition of $X$ parametrized by $\Mc{T}$ and $d\in\Mc{D}_\infty(X,\Mc{O}).$ 

\begin{prop}[\kig{Proposition 3.8}]
Define $\rho_*$ on $\Sigma^*\times\Sigma^*$ by 
\[\rho_*(\omega,\tau)=\begin{cases}
2^{-\max\{m\in\Mb{Z}\mid \omega_m=\tau_m\}}&\text{ if }\omega\ne\tau,\\
0 &\text{ if } \omega=\tau.
\end{cases}\]
Then $(\Sigma^*,\rho_*)$ is a totally disconnected $\sigma$-compact metric space. Moreover, if $\#(S(w))\ge2$ for any $w\in T$, then $(\Sigma^*,\rho_*)$ does not have isolated points. 
\end{prop}

The proof of this statement is standard.

\begin{lem}\label{Klocfin}
   The following conditions are equal:
 \begin{enumerate}
 \item A partition $K$ on $X$ is locally finite.
 \item For any $w\in T,$ there exists an open set $U_w$ such that $K_w\subset U_w$ and $\#\{v\in (T)_{[w]}|K_v\cap U_w\ne\emptyset\}<\infty.$ 
 \item For any $w\in T,\text{ there exists an open set }U_w$ such that $K_w\subset U_w$ and $\#\{v\in (T)_0|K_v\cap U_w\ne\emptyset\}<\infty.$
 \end{enumerate}
\end{lem}
Since $\Mc{T}$ is locally finite, this lemma is easily proven in (1)$\Rightarrow$(3)$\Rightarrow$(2)$\Rightarrow$(1).

\begin{lem}[\kig{Lemma 4.2}]\label{Klocfin2}
Assume $K$ is locally finite, then
  \begin{enumerate}
  \item For any $w\in T,$ $O_w$ is an open set and $O_v\subseteq O_w$ for any $v\in S(w).$
  \item $O_w\cap O_v=\emptyset$ if $w,v\in T$ and $\Sigma_*\cap \Sigma_*=\emptyset.$
  \item If $\Sigma_w^*\cap\Sigma_v^*=\emptyset,$ then $K_w\cap K_v=B_w\cap B_v$ where $B_w=K_w\setminus O_w.$ 
  \end{enumerate}
\end{lem}

\begin{proof}
\begin{enumerate}
\item  Since $K$ is locally finite, 
\[O_w=K_w\setminus(\cup_{v\in (T)_{[w]}:K_v\cap U_w\ne\emptyset}K_v)=U_w\setminus(\cup_{v\in (T)_{[w]}:K_v\cap U_w\ne\emptyset}K_v)\]
 and $O_w$ is open. The rest of the statement follows from (P2).
\item If $[w]\le[v]$, $\Sigma_*\cap \Sigma_*=\emptyset$ implies $\pi^{[v]-[w]}(v)\ne w$ and by (1), $O_w\cap O_v\subseteq O_w\cap O_{\pi^{[v]-[w]}(v)}=\emptyset.$ The case $[v]\le[w]$ is the same.
\item This immediately follows from (2). 
\end{enumerate}
\end{proof}
%
%

\begin{lem}
 Let $w\in T$, Then the triple $\Mc{T}_w=(T_w,\mathcal{A}_\pi|_{T_w\times T_w},w)$ is a rooted tree. Moreover, for any $v\in T_w\setminus\{ w\},\pi_T(v)=\pi_{\Mc{T}_w}(v)$.
\end{lem}

\begin{proof}
 It is trivial to prove $(T_w,\mathcal{A}|_{T_w\times T_w},w)$ be a tree. Let $v\in T_w\setminus\{ w\}$ then there is some $n\ge1$ and $(v,\pi_T(v),...,\pi_T^{(n)}(v)=w)$, which is the geodesic between $w$ and $v$. Therefore $\pi_T(v)=\pi_{\Mc{T}_w}(v).$
\end{proof}

\begin{prop}[\kig{Proposition 5.2}]
Suppose that $g:T\rightarrow(0,\infty)$ satisfies (G1) and (G2). Then $g\in\mathcal{G}(T)$ if and only if $g$ satisfies the following condition:
\begin{itemize}
\item [\rm(G3)'] For any $w\in T,\lim_{n\to\infty}\sup_{v\in T_w\cap(T)_{[w]+n}}g(v)=0.$ 
\end{itemize}
\end{prop}

\begin{proof}
  (G3)' immediately shows (G3). On the other hand, if $g$ is a weight function on $T,$ then $g'$ defined by $g'(v)=g(v)/g(w)$ is a weight function on $\Mc{T}_w=(T_w,\mathcal{A}_\pi|_{T_w\times T_w},w)$. Therefore \cite[Proposition 5.2]{Kig18} shows (G3)'.
\end{proof}

\begin{rem}
  In~\cite{Kig18}, the definition of weight function on a rooted tree is given by the counterpart of (G3)' instead of (G3). 
\end{rem}

For proving  counterparts of statements in~\cite[Section 5]{Kig18}, we use the following Lemma \ref{ue1} derived from locally finiteness. Using this lemma, we can apply original proofs (by using some $T_v$). 
\begin{lem}\label{ue1}
 Let $g$ be a weight function on $T$. Then for any $s\in(0,\infty),w\in\Lambda_s^g,M\ge0,$ there exists $v\in T$ such that $\Lambda_{s,M}^g(w)\subset T_v.$
\end{lem}

\begin{proof}
Let $\tilde{\Lambda}^g_{s,1}(w):=\{\pi^{0\vee([v]-[w])}(v)|v\in\Lambda^g_{s,1}(w)\},$ then $\#\tilde{\Lambda}^g_{s,1}(w)<\infty$ because $K$ is locally finite. And inductively we define 
  \[ \tilde{\Lambda}^g_{s,n+1}(w):=\bigcup\{\tilde{\Lambda}^g_{g(v),1}(v')|\text{for some } v\in\tilde{\Lambda}^g_{s,n}(w),v'\in\Lambda^g_{g(v)}\text{ and }v\in T_{v'}\} \]
  then $\#\tilde{\Lambda}^g_{s,M}(w)<\infty$ for all $M$. \\
  Moreover, for any $v\in\Lambda^g_{s,n}(w)$ there exists $v'\in \tilde{\Lambda}^g_{s,n}(w)$ such that $v\in T_{v'}.$ Therefore $\Lambda^g_{s,M}(w)\subset\cup_{v\in\tilde{\Lambda}^g_{s,M}(w)}T_v\subset T_u$ for some $u$ because of (T1).
\end{proof}

\begin{cor}\label{ue2}
Let $g$ be a weight function on $T$, then for any $x\in X$ and $M\ge0,$ $\#(\Lambda^g_{s,M}(x))<\infty$ for any $s>0$ and $\min_{w\in\Lambda^g_{s,M}(x)}|w|\to\infty$ as $s\to 0.$
\end{cor}

\begin{proof}
$\Lambda^g_{s,M}(x)\subset \Lambda^g_{s,M+1}(w) \subset \Lambda^g_s\cap T_v$ for some $w,v\in T.$ Since $\Sigma^*_v$ is compact, $\Sigma^*_u$ is open for any $u\in T$ and $\Sigma^*_v=\sqcup_{u\in\Lambda^g_s\cap T_v}\Sigma^*_u,$ it follows that $\#(\Lambda^g_{s,M}(x))\le\#(\Lambda^g_s\cap T_v)<\infty.$ The rest of this corollary follows from  $\Lambda^g_{s,M}(x)\subset \Lambda^g_s\cap T_v$ and (G3)'.
\end{proof}

By Lemma \ref{ue1} and Corollary \ref{ue2}, by using $\Lambda^g_{s,M}(x)$ instead of $\Lambda^g_s,$ we can prove following two propositions in the same way as~\cite{Kig18}.

\begin{prop}[\kig{Proposition 5.7}]
Let $g:T\to(0,\infty)$ be a weight function.Then for any $s\in(0,\infty),x\in X,U^g_0(x,s)$ is a neighborhood of $x.$ Furthermore, $\{U^g_M(x,s)\}_{s\in(0,\infty)}$ is a fundamental system of neighborhood of $x$ for any $x\in X.$
\end{prop}

\begin{prop}[\kig{Proposition 5.9}]
For any $M\ge0$ and $x,y\in X,$ define $\delta^g_M(x,y)$ by
\[\delta^g_M(x,y)=\inf\{s|y\in U^g_M(x,s)\},\]
then the infimum attains for any $M\ge0$ and $x,y\in X.$ In particular, for any $M\ge0$ and $s\in(0,\infty),$
\[U^g_M(x,s)=\{y|\delta^g_M(x,y)\le s\}.\]
\end{prop}
\begin{defi}[exponential]
	Let $g$ be a weight function. $g$ is called exponential if there exist $c\in(0,1)$ and $m_0\ge0$ such that
		\begin{align}
		& cg(\pi(w))\le g(w) &&\text{for any }w\in T\text{ and } \label{supexp}\\
		& cg(w)\ge g(v) &&\text{for any }m\ge m_0,\ w\in T\text{ and }v\in S^m(w). \label{subexp}
		\end{align}
		$d\in\Mc{D}_\infty(X,\Mc{O})$ is called exponential if $g_d$ is exponential.
\end{defi}

\begin{defi}[gentle]
  Let $g$ be a weight function on $T.$ function $g:T\to(0,\infty)$ is called gentle with respect to $g$ if there exists $C>0$ such that $f(v)\le Cf(w)$ whenever $K_v\cap K_w\ne\emptyset$ and $K_v, K_w\in \Lambda^g_s$ for some $s>0.$ We write $f\gen g$ if $f$ is gentle with respect to $g.$ 
\end{defi}

\begin{thm}[\kig{Theorem 15.2.}]
Assume $\sup(\#S(w))<\infty$ and  $d$ is exponential, thick, uniformly finite and $d\gen 2^{-[w]}.$ Let $\alpha>0.$ Then there exist a metric $\rho\in \Mc{D}_\infty(X,\Mc{O})$ and a measure $\mu\in \Mc{P}_\infty(X,\Mc{O})$ such that $\rho\qs d$ and $\mu$ is $\alpha$-Ahlfors regular with respect to $\rho$ if and only if there exists an exponential weight function on $T$ such that
\begin{itemize}
\item $g\gen d.$
\item There exists $c>0$ and $M\ge 1$ such that 
\[ cD^g_M(x,y)\le D^g(x,y)\]
for any $x,y\in X,$ where
\begin{multline*}
D^g_m(x,y)=\inf\{\sum_{i=0}^n g(w(i)) \mid n\le m,\ x\in K(w(0)),\ y\in K(w(n))\\ \text{ and }K_{w(i-1)}\cap K_{w(i)}\ne\emptyset\text{ for any }i\in[1,n]_\Mb{Z}\}, 
\end{multline*}
and $D^g(x,y)=\inf_mD^g_m(x,y).$
\item There exists $c>0$ such that for any $w\in T$ and $n\ge0,$
\[c^{-1}g(w)^\alpha\le\sum_{v\in S^n(w)}g(v)^\alpha \le cg(w)^\alpha.\]
\end{itemize}
\end{thm}

\begin{proof}
The ``only if'' part follows from in the same way of~\cite[Theorem 15.2]{Kig18} by replacing $\rho(x,y)$ by $D^g(x,y).$ On the other hand, applying \cite[Theorem 15.2]{Kig18} to $T_w$ and $g(\cdot)/g(w)$ for $w\in (T)_0,$ we obtain a probability measure $\mu_w$ on $K_w$ which is $\alpha$-Ahlfors regular with respect to $(K_w,D^g).$ Define $\mu=\sum_{w\in (T)_0}g(w)\mu_w,$ then $\mu$ satisfies $\mu(w)\asymp g(w)^\alpha$ for any $w\in T,$ and in the same way of the original proof, we can see that $\mu$ is $\alpha$-Ahlfors regular with respect to $D^g.$
\end{proof}

\begin{thm}[\kig{Theorem 18.1}]\label{saigo}
Assume $d$ satisfies the basic framework. Let $M_1\in\Mb{N}$ and $k$ be a sufficiently large number. If there exists $\varphi:T^{(k)}=\cup_{n\in\Mb{Z}}(T)_{nk}\to(0,1]$ such that  
\[\sum_{i=1}^m\varphi(w(i))\ge 1\text{ and } \sum_{v\in S^k(w)}\varphi(v)^p<\frac{1}{2}(\sup_{w\in T}\#(\Gamma_1(w)))^{-2(M_1+M_*)},\]
for any $w\in T^{(k)}$ and any path $(w(1),...,w(m))$ in $J^h_{[w]+k}$ such that 
\begin{itemize}
\item $w(i)\in S^k(\Gamma_{M_1}(w))$ for any $i\in[1,m]_\Mb{Z},$
\item $\Gamma_{M_1}(w(1))\cap S^k(w)\ne\emptyset$ and $\Gamma_{M_1}(w(m))\setminus S^k(\Gamma_{M_1}(w))\ne\emptyset.$
\end{itemize}
Then there exists an exponential metric $\rho\in\Mc{D}_\infty(X,\Mc{O})$ such that $\rho\qs d$ and $\rho$ is $\alpha$-Ahlfors regular.
\end{thm}

This theorem needs the following lemma.

\begin{lem}[\kig{Lemma 18.3}]
Let $k$ be sufficiently large, $\kappa_0\in(0,1)$ and let $f:T^{(k)}\to[\kappa_0,1).$ Then there exists $g:T^{(k)}\to(0,\infty)$ such that
\begin{gather*}
g(u)\ge \kappa_0g(v)\quad\text{for any }(u,v)\in\Cup_{m\in\Mb{Z}}J^h_{mk}, \\ 
f(u)\le \frac{g(u)}{g(\pi^k(u))}\le\max_{v\in\Gamma_{M_*}(u)}f(v)\quad\text{for any }u\in T^{(k)},\text{ and} \\
\sum_{v\in S^k(w)}\left(\frac{g(v)}{g(\pi^k(v))}\right)^p
\le(\sup_{w\in T}\#(\Gamma_1(w)))^{2M_*}\sup_{w'\in \Gamma_{M_*}(w)}\left(\sum_{u\in S^k(w')}f(u)^p \right),
\end{gather*}
for any $p>0$ and $w\in T^{(k)}.$
\end{lem}

\begin{proof}
  Apply \cite[Lemma 18.3]{Kig18} to $T^{(k)}_{\pi^{nk}(\phi)}$ and get $g_n:T^{(k)}_{\pi^{nk}(\phi)}\to (0,1].$ We let $\tilde{g}_n:=\frac{g_n}{g_n(\phi)}$ then $\tilde{g}_n(\phi)=1$ for all $n$. Let $w\in T$, then we take $l,m$ such that $\pi^{lk}(\phi)=\pi^{mk}(w)$ and then for any $n\ge l$,
 \begin{align*}
  \Bigl( \prod_{i=0}^{l-1}\max_{v\in\Gamma_{M_*}(\pi^{ik}(\phi))}f(v)\Bigr)^{-1}\le g_n(\pi^{lk}(\phi))\le\Bigl(\prod_{i=0}^{l-1}f(\pi^{ik}(\phi))\Bigr)^{-1}, \\
  \Bigl( \prod_{j=0}^{m-1}\max_{v\in\Gamma_{M_*}(\pi^{jk}(w))}f(v)\Bigr)^{-1}\le g_n(\pi^{mk}(w))\le\Bigl(\prod_{j=0}^{m-1}f(\pi^{jk}(w))\Bigr)^{-1}.
 \end{align*}
 Therefore there exist $\underline{C}_w,\overline{C}_w$ with $0<\underline{C}_w<\overline{C}_w<\infty$ such that for any $n\ge l,g_n(w)\in[\underline{C}_w \overline{C}_w]$. Using diagonal sequence argument, we get a subsequence of $(g_n)$ and its limit $g$, which also satisfies all required conditions.
\end{proof}

\begin{proof}[Proof of Theorem \ref{saigo}]
 In the same way as proof of \cite[Theorem 18.1]{Kig18}, we obtain $\sigma:T^{(k)}\to [\kappa_0,1).$ We inductively define $\tilde{g}(\phi)=1$ and
  \[ \tilde{g}(w)=\prod_{i=1}^{m(w)}\sigma(\pi^{(i-1)k}(w))\Bigl(\prod_{j=1}^{l(w)}\sigma(\pi^{(j-1)k})(\phi)\Bigr)^{-1}, \]
where $l(w)=\min\{l\mid \pi^{lk}(\phi)=\pi^{mk}(w)\text{ for some }k\}$ and $m(w)=\min\{m\mid \pi^{lk}(\phi)=\pi^{mk}(w)\text{ for some }l\}.$ Note that then $\pi^{l(w)k}(\phi)=\pi^{m(w)k}(w)$ by Lemma \ref{hstruc}. Using this $\tilde{g}$, we can prove this Theorem in the same way as~\cite[Theorem 18.1]{Kig18}.
\end{proof}

We have modify all statements which is necessary for proving Theorem \ref{sigmain1}. Therefore it follows. 
\end{proof}

\section{Ahlfors regular conformal dimension of infinite graphs}
In this section, we give results about the ARC dimensions and the spectral dimensions of metrics on infinite graphs, which are the main results of this paper. To get these results, the cable systems of graphs play an important role. Technically, the main contribution of this paper is to show ARC dimension and a partition of a graph coincides with those of its cable system. Cable systems do NOT appear in statements of main results, but we use them and adapt the results of former sections and lead results for graphs.
Throughout this section, $G$ is a countable (infinite) set, $(G,E)$ is a connected, bounded degree graph and $\Mc{T}=(T,\pi,\phi)$ is a bi-infinite tree with a reference point. 

\subsection{Ahlfors regular conformal dimension of metrics on infinite graphs}
We first denote a class of metrics on $(G,E)$, which we consider in this paper.

\begin{defi}[Fitting metric]
We say a metric $d$ on $G$ fits to $(G,E)$ if it satisfies the following conditions.
\begin{itemize}
\item[(F1)] There exists $C>0$ such that $d(x,y)\le Cd(x,z)$ for any $x,y,z\in G$ with $x\sim y$ and $x\ne z.$
\item[(F2)] For any $\epsilon>0,$ there exists $r>0$, $n\ge1$ and $\{x_i\}_{i=0}^n\subseteq G$ such that
\begin{itemize}
\item $x_i\in B_d(x_0,r)$ for any $i\in[0,n-1]_\Mb{Z}$ and $x_n\not\in B_d(x_0,r),$ 
\item $d(x_i,x_{i-1})\le\epsilon r$ and $x_i\sim x_{i-1}$ for any $i\in[1,n]_\Mb{Z}.$
\end{itemize}
\end{itemize}
If the graph $(G,E)$ is fixed or clear, we simply say $d$ is fitting when $d$ fits to $(G,E).$
\end{defi}

The condition (F1) is a natural condition for metrics on $G.$ For example, the graph distance $l_E$ and ``gently weighted'' graph distances satisfy (F1). Moreover, the effective resistance of a weighted graph with controlled weight, which we will introduce later, also satisfies (F1). The condition (F2) is a little technical, which is needed to evaluate $\ard(G,d).$ 
\begin{ex}
Let $G=\Mb{Z}$ and $E=\{(n,m)\mid |n-m|=1\}.$ For any $k\ge1,$ let $x_i=i$ for $i\in[0,k]_{\Mb{Z}}$ then $l_E(i+1,i)\le\frac{k}{k}$ and $x_k\not\in B_{l_E}(x_0,k),$ so $l_E$ satisfies (F2). On the other hand, let $d(n,m):=|2^n-2^m|$ then for any simple path $(n,n+1,...,n+k),d(n+k-1,n+k)=2^{n+k-1}\ge\frac{d(n,n+k)}{2}.$ Hence in this case $d$ does not satisfy (F2) for $\epsilon\ge1/2.$
\end{ex}
Remark that for any $(G,E),$ $l_E$ satisfies (F2).

\begin{lem}\label{qsfit}
Let $d,\rho$ are metrics on $G$ and $d\qs \rho.$ If $d$ fits to $(G,E),$ then $\rho$ fits to $(G,E).$
\end{lem}

%
For this lemma and later statements, now we recall basic properties of quasisymmetry. Let $(X,d)$ and $(X,\rho)$ be metric spaces. 
  \begin{enumerate}
\item Let $\theta:[0,\infty)\to[0,\infty)$ be a homeomorphism, then the following conditions are equal:
\begin{enumerate}
\item $d$ is $\theta$-quasisymmetric to $\rho.$
\item $\rho(x,z)\le\theta(t)\rho(x,z)$ whenever $d(x,y)\le td(x,z).$
\item $\rho(x,z)<\theta(t)\rho(x,z)$ whenever $d(x,y)< td(x,z).$
\end{enumerate}
\item If $d\qs\rho$ and $\diam(X,d)=\infty,$ then $\diam(X,\rho)=\infty.$
\item $\qs$ is an equivalence relation between metrics on $X.$
\item If $d\qs\rho,$ then both $(X,d)$ and $(X,\rho)$ induce the same topology (in other words, $\mathrm{id}_X$ is a homeomorphism between $(X,d)$ and $(X,\rho)$).    
  \end{enumerate}
(1) follows from monotonicity of $\theta.$ For (2)$\sim$(4), see  \cite[Section 10]{Hei} for example.

\begin{proof}[Proof of Lemma \ref{qsfit}]
Since $d$ satisfies (F1) and $d\qs \rho,$ $\rho(x,y)\le \theta(C)\rho(x,z)$ for any $x,y,z\in G$ with $x\sim y$ and $x\ne z,$ so $\rho$ satisfies (F1).
Next we show $\rho$ satisfies (F2). Fix any $\epsilon>0.$ Let $\delta<1/2$ such that $\theta(2\delta)\theta(3)<\epsilon.$ Since $d$ satisfies (F2), there exists  $\{x_i\}_{i=0}^n\subseteq G$ such that
\begin{itemize}
\item $x_i\in B_d(x_0,r)$ for any $i\in[0,n-1]_\Mb{Z}$ and $x_n\not\in B_d(x_0,r),$ 
\item $d(x_i,x_{i-1})\le\delta r$ and $x_i\sim x_{i-1}$ for any $i\in[1,n]_\Mb{Z}.$
\end{itemize}
Let $i\in[0,n-1]_\Mb{Z}.$ Since $d(x_0,x_n)\ge r$ and $x_i,x_n\in B_d(x_0,(1+\delta)r),$ 
\[\frac{r}{2}\le d(x_0,x_i)\vee d(x_i,x_n)<3r \]
and hence 
\[\rho(x_i,x_{i+1})\le\theta(2\delta)(\rho(x_0,x_i)\vee\rho(x_i,x_n))\le \theta(2\delta)\theta(3)\rho(x_0,x_n).\]
Let $m=\min\{i\mid x_i\not\in B_\rho(x_0,\rho(x_0,x_n))\}$, then $r=\rho(x_0,x_n)$ and $\{x_i\}_{i=0}^m$ satisfies the conditions of (F2) for $\epsilon.$
\end{proof}

Next we introduce partitions of infinite graphs.

\begin{defi}[Partition]\label{defPG}
A map $K:T\to\{A\subseteq G\mid \# (A)<\infty\}$ is called a partition of $(G,E)$ parametrized by $\Mc{T}$ if it satisfies following conditions.
\begin{itemize}\setlength{\leftskip}{8pt}
\item[(PG1)] $ \Cup_{v\in S(w)}K_v=K_w $ for any $w\in T.$
\item[(PG2)] For any $\omega\in\Sigma_*,$ there exist $n_0(\omega)\in\Mb{Z}$ and $x,y\in G$ such that $x\sim y$ and 
$K_{\omega_n}=\{x,y\}$ for any $n\ge n_0(\omega).$
\item[(PG3)] For any $(x,y)\in E,$ there exists $w\in T$ such that $K_w=\{x,y\}.$
\end{itemize}
\end{defi}

For the rest of this section, $K$ is a partition of $(G,E)$ parametrized by $\Mc{T}.$ 

\begin{lem}\label{lambdae}
Let $\Lambda_e=\{w\in T \mid\#(K_w)=2 \text{ and } \#(K_{\pi(w)})>2 \},$ then $\sqcup_{w\in\Lambda_e}\Sigma^*_w=\Sigma^*.$
\end{lem}

\begin{proof}
$\cup_{w\in\Lambda_e}\Sigma^*_w=\Sigma^*$ directly follows from (PG2). By (PG1), $\#(K_{\omega_n})$ is non-increasing for any $\omega\in\Sigma^*,$ so there exists a unique $n\in\Mb{Z}$ such that $\omega_n\in\Lambda_e.$ This shows $\Sigma^*_w\cap\Sigma^*_v=\emptyset$ for any $w,v\in\Lambda_e$ with $w\ne v.$  
\end{proof}

\begin{defi}
	\hspace{0pt} \samepage
\begin{enumerate}
\item We denote $\omega_{n_0(\omega)}$ by $\omega_e.$ where $n_0(\omega)$ is in the condition (PG2). We also define $T_e$ by 
\[T_e=\{w\in T\mid T_w\cap\Lambda_e \ne\emptyset \}=\{w\in T\mid \#(K_{\pi(w)})>2\}, \]
and for $w\in (T\setminus T_e)\cup \Lambda_e$ we define $w_e\in \Lambda_e$ such that $w\in T_{w_e}.$   
\item $K$ is called minimal if $K_w\ne K_v$ for any $w,v\in \Lambda_e$ with $w\ne v.$ 
\end{enumerate}
\end{defi}

\begin{defi}[Discrete weight function]
Recall that $K$ is a partition of $(G,E).$ A function $g:T_e\to (0,\infty)$ is called a discrete weight function (with respect to $K$) if it satisfies following conditions.
\begin{itemize}
\setlength{\leftskip}{8pt}
  \item[(GG1)] For some $w\in T_e,\ \lim_{n\to\infty}g(\pi^n(w))=\infty.$
  \item[(GG2)] For any $w\in T_e,\ g(\pi(w))\ge g(w).$
\end{itemize} 
For $s>0,$ we define the scale $\Lambda^g_s$ associated $g$ by 
\[ \Lambda^g_s=\{w\in T_e\mid g(w)\le s< g(\pi(w))\}, \]
and define $E^g_s,\Lambda^g_{s,M}(w),\Lambda^g_{s,M}(x)$ in the same way as compact cases. We also define $U^g_M(x,s)$ for $M\ge0, x\in G$ and $s>0$ by
\[U^g_M(x,s)=\begin{cases}
\{x\}, &\text{ if }\Lambda^g_{s,M}(x)=\emptyset, \\
\Cup_{w\in\Lambda^g_{s,M}(x)}K_w, &\text{otherwise.}
\end{cases} 
\]
\end{defi}

\begin{rem}
Different from compact and $\sigma$-compact cases, $\Sigma^*$ is not necessarily equal to $\sqcup_{w\in \Lambda^g_s}\Sigma^*_w$ since they are restricted to $T_e.$ The difference also appears in the definition of $U^g_M(x,s).$
\end{rem}


\begin{lem}
 Define
 \[\Mc{D}_\infty(G)=\{d\mid d \text{ is a metric on }G\text{ such that }\diam(G,d)=\infty\}\]
 and let $d\in\Mc{D}_\infty(G).$ We also define $g_d:T_e\to(0,\infty)$ by \smash{$\displaystyle
 g_d(w)=\max_{x,w\in K_w}d(x,y),$} then $g_d$ is a discrete weight function. 
\end{lem}
We denote $g_d$ by $d$ if no confusion may occur.

\begin{defi}
Let $g$ be a discrete weight function.
\begin{itemize}
\item $g$ is called uniformly finite if \eqref{uf} holds.
\item $g$ is called thick (with respect to $K$) if there exists $\alpha>0$ such that for any $w\in T_e,$ $\Lambda^g_{\alpha g(\pi(w)),1}(x)\subset T_w$ for some $x\in K_w.$
\end{itemize}
$d\in\Mc{D}_\infty(G)$ is called 
uniformly finite and thick if $g_d$ is 
uniformly finite and thick, respectively.
\end{defi}

We define adapted in the same way as compact cases, and

\begin{defi}
Let $(G,d)$ be a metric space and assume $K$ is minimal.  We say $(G,d)$ satisfies the basic framework (with respect to $K$) if the following conditions hold.
\begin{itemize}
\item $\sup_{w\in T_e\setminus \Lambda_e}\#(S(w))<\infty.$
\item $d$ is uniformly finite, thick, adapted.
\item There exists $r\in(0,1)$ such that $d\asymp r^{[w]}$ for any $\in T_e.$  
\end{itemize}
\end{defi}

The difference between these definition and those of compact cases are $T_e$s in the notation.


\begin{defi}\label{Tr}
 Let $r\in(0,1).$ For $w\in\Lambda_e$ and $n\ge0$, let
\[\Mf{S}_{w,m}=\{\{(x,k),(y,2^{n(m)}-1-k)\}_{w,m}\mid k\in[0,2^{n(m)}-1]_\Mb{Z}\}\]
where $x,y\in G$, $n(m)\in\Mb{N}$ such that $K_w=\{x,y\}$ and $2^{-n(m)}\le r^m<2^{1-n(m)}.$
Then we define $T_r=T_e\sqcup(\Cup_{w\in\Lambda_e}\bigsqcup_{m\ge 1} \Mf{S}_{w,m})$ and $\pi':T_r\to T_r$ by 
\begin{align*}
&\pi'(w)\\
=&\begin{cases}
  \pi(w), & \text{if }w\in T_e,\\
 v, & \text{if }w\in \Mf{S}_{v,1},\\
\{(x,l),(y,2^{n(m-1)}-1-l)\}_{v,n-1},  & \text{if }w=\{(x,k),(y,2^{n(m)}-1-k)\}_{v,m}\\
&\text{and }[\frac{k}{2^{n(m)}},\frac{k+1}{2^{(n(m))}}]\subseteq[\frac{l}{2^{n(m-1)}},\frac{l+1}{2^{n(m-1)}}].
\end{cases}
\end{align*}
Moreover, we define $K'$ by $K'_w=K_w$ (if $w\in T_e$), $K_v$ (if $w\in \sqcup_{m\ge 1}\Mf{S}_{v,m}$).
\end{defi}

\begin{lem}
 $(T_r,\pi')$ is a by-infinite tree and $K'$ is a partition of $(G,E)$. Moreover, if we write $[w]'$ for the height of $(T_r,\pi',\phi')$ for $\phi'\in T_r$ and $\Lambda'_e$ for  $K'$ version of $\Lambda_e$, then $\Lambda_e=\Lambda'_e$ and we can take $\phi'\in T_r$ such that $[w]=[w]'.$
\end{lem}

\begin{proof}
 For any $w\in T_r\setminus T_e$, by definition of $\pi',$ we have  
\begin{itemize}
  \item $\pi^n(w)\in\Lambda_e$ for some $n>0.$
  \item For any $n>0$, $\pi^n(w)\ne w.$
\end{itemize}
These and  conditions (T1), (T2) of $\pi$ show conditions (T1), (T2) of $\pi'$, so $(T_r,\pi')$ is a bi-infinite tree. Fix $w\in T_e$ and let $\phi'\in S^{[w]}$ then $[w]=[w]'$ and by Lemma 3.1, $[v]=[v]'$ for any $v\in T_e$ because $\pi=\pi'$ on $T_e.$ The rest of this lemma is clear by definition.
\end{proof}

Remark that discrete weight functions, and their properties are given only on $T_e$, so they do not change if we replace $(T,\pi,\phi)$ by $(T_r,\pi',\phi').$ For the rest of this section, assume $(T,\pi,\phi)=(T_r,\pi',\phi').$

\begin{defi}
Now we formally define $\ul{K}$ on $T_r$ by
\[\ul{K}_w=\begin{cases}
K_w,  &\text{ if }w\in T_e,\\
\{x\}, &\text{ if }w=\{(x,0),(y,2^{n(m)}-1)\}_{m,v}\text{ for some }m,v,\\
\emptyset, &\text{ otherwise.}
\end{cases}  \]
We also define $J^h_m\subset (T)_m\times (T)_m$ by
\begin{align*}
J^h_m=J^h_m(K)=\{(w,u)\mid w,u\in(T)_m,\ \ul{K}_w\cap \ul{K}_v\ne\emptyset \text{ or there exists }v\in T_e,\\i\ge 0\text{ such that }w=\{(x,i),(y,2^{n(m-[v])}-1-i)\}_{w,m-[v]},\\u=\{(x,i-1),(y,2^{n(m-[v])}-i)\}_{w,m-[v]}\}.
\end{align*}
\end{defi}
$J^h_m$ can be constructed as follows: define 
\[\hat{J}^h_m=\{(w,v)\mid w,v\in((T)_m\cap T_e)\cup\Lambda_{e,m}, K_w\cap K_v\ne\emptyset \},\] where $\Lambda_{e,m}=\{w_e\mid w\in(T)_m\setminus T_e\}$, and replace each $w\in \Lambda_{e,m}$ by a $2^{n(m-[v])}$-path. We will justify this idea later in the cable system. We define $\ul{I}_\Mc{E}, \ol{I}_\Mc{E}, \ul{R}_p, \ol{R}_p,$ etc. in the same way as Definition \ref{ciro} and \ref{ciro2}.\\
The following is the one of two main theorems of this paper.
\begin{thm}\label{Smain1}
Let $K$ be a minimal partition of $(G,E).$ If $d\in\Mc{D}_\infty(G)$ satisfies the the basic framework in Definition \ref{defBF1} and fits to $(G,E)$, then for any $N,N_1,N_2$ \color{black}such that $N_2-N_1$ is sufficiently large. \color{black}
\begin{enumerate}
\item\[\ul{I}_\Mc{E}(N_1,N_2,N)=\ol{I}_\Mc{E}(N_1,N_2,N)=\ard(G,d).\]
\item If $\ul{R}_p(N_1,N_2,N)<1,$ then
\[\ard(G,d)\le \ul{d}^S_p(N_1,N_2,N) \le \ol{d}^S_p(N_1,N_2,N) <p.\]
\item If $\ol{R}_p(N_1,N_2,N)\ge 1,$ then
\[\ard(G,d)\ge \ol{d}^S_p(N_1,N_2,N) \ge \ul{d}^S_p(N_1,N_2,N) \ge p.\]
\end{enumerate}
\end{thm}

In order to prove the theorem, we introduce the notion of cable system.
\begin{defi}[Cable system]
 Let $\simeq$ be the  minimal equivalence relation on $G\times G\times[0,1]$ which satisfies
\begin{itemize}
\item $((x,y),0)\simeq ((x,z),0)$ for any $x,y,z\in G,$ 
\item $((x,y),t)\simeq ((y,x),1-t)$ for any $(x,y)$ and $t\in[0,1].$
\end{itemize}    
Then we define the cable system $\Mf{C}_G$ of $(G,E)$ by $\Mf{C}_G:=(E\times[0,1])/\simeq.$\\
For $(x,y)\in E,$ we also define $\iota(x,y)=(x,y)\times [0,1]/\simeq.$ Moreover, for any $x\in G,\ \tau(x)=((x,y),0)/\simeq,$ where $(x,y)\in E,$ is well-defined because of the definition of $\simeq.$ We equate $\tau(x)$ with $x$ and regard $G$ as a subset of $\Mf{C}_G.$ 
\end{defi}

\begin{defi}[Induced cable metric]
Let $\alpha\in(0,1]$ and $d\in\Mc{D}_\infty(G).$ We define an induced cable metric $d_{c,\alpha}:\Mf{C}_G\times \Mf{C}_G\to[0,\infty)$ by
\begin{align*}
&d_{\Mf{C},\alpha}(x,y)\\
=&\begin{cases}
|t-s|^\alpha d(x_0,x_1),& \begin{aligned}& \text{ if }(x\simeq ((x_0,x_1),t)\text{ and }y\simeq \\ &((x_0,x_1),s)\text{ for some }(x_0,x_1)\in E,
\end{aligned}
\\[10pt]
\hspace{-1pt}\begin{aligned}
\min \{t^\alpha d(x_1,x_0)\hspace{-1pt}+\hspace{-1pt}d(x_0,y_0)\hspace{-1pt}+\hspace{-1pt}s^\alpha d(y_0,y_1)& \\
\mid x\simeq ((x_0,x_1),t)\text{ and }y\simeq((y_0,y_1),s)\}&
\end{aligned}, \hspace{-5.8pt}
&\text{ otherwise.}\end{cases}
\end{align*}
\end{defi}
 We write $d_{\Mf{C}}$ instead of $d_{\Mf{C},1}.$ Note that 
\begin{lem}
 $(\Mf{C}_G,d_{\Mf{C},\alpha})$ is a metric space.  
\end{lem}

\begin{proof}
If $d_{\Mf{C},\alpha}(x,y)=0$, then $x\simeq ((x_0,x_1),t)\simeq y$ for some $x_0,x_1,t.$ $d_{\Mf{C}}(x,y)=d_{\Mf{C}}(y,x)$ is obvious. Moreover, since $t^\alpha+s^\alpha\ge(t+s)^\alpha$ and the triangle inequality for $d$ holds, we can write 
\begin{multline*}
d_{\Mf{C},\alpha}(x,y)=\min \{\sum_{i=1}^n|s_i-t_i|^\alpha d(x_i,y_i)\mid n\ge 1,\ x_i,y_i\in G\\
\text{ such that } x\simeq ((x_1,y_1),s_1),\ y\simeq ((x_n,y_n),t_n)\\
\text{ and } ((x_i,y_i),t_i)\simeq((x_{i+1},y_{i+1})s_{i+1})\text{ for any }i\in[1,n-1]_\Mb{Z}\}.
\end{multline*} 
(Remark that $(x_i,y_i)$ is not necessarily in $E$.) Hence the triangle inequality for $d_{\Mf{C},\alpha}$ also holds.
\end{proof}
We also note that $d(x,y)=d_{\Mf{C},\alpha}(x,y)$ for any $x,y\in G$ since $d$ satisfies the triangle inequality.\\
For the notation of a partition,  we write $K_w=\{w^+,w^-\}$ for each $w\in\Lambda_e.$

\begin{defi}
Define $\Mc{K}:T=T_r\to \Mf{C}_G$ by
\[\Mc{K}_w=\begin{cases}
  \Cup_{\omega\in\Sigma^*_w}\iota(\omega_e^+,\omega_e^-),& \text{ if }w\in T_e,\\
 (x,y)\times[\frac{k}{2^{n(m)}},\frac{k+1}{2^{n(m)}}],& \text{ if }w=\{(x,k),(y,2^{n(m)}-1-k)\}_{w,m}.
\end{cases}\]
\end{defi}

\begin{lem}\label{lemKMcK}
\begin{enumerate}
\item $\Mc{K}$ is a partition of $(\Mf{C}_G.d_{\Mf{C},\alpha}).$
\item for any $w,v\in T_e,$ $\Mc{K}_w\cap\Mc{K}_v\ne\emptyset$ if and only if $K_v\cap K_w\ne\emptyset.$ 
\end{enumerate}
\end{lem}

\begin{proof}
\begin{enumerate}
\item By the definition of $d_{\Mf{C},\alpha},$ $\Mf{C}_G$ has no isolated point, and $((x,y)\times[0,1],d_{\Mf{C},\alpha})$ is isomorphic to $([0,1],|\bullet|_\Mb{R})$ and compact. Since $\{\omega_e^+,\omega_e^-\}\subseteq K_w$ for any $\omega\in\Sigma^*_w$ and $K_w$ is a finite set, $\Mc{K}_w$ is a compact set for any $w.$
\begin{itemize}
\item[(P1)] By the definition of $\Mc{K},$ (P1) follows from (PG1).
\item[(P2)] For any $\omega\in\Sigma^*$, $\cap_{n\ge 0}K_{\omega_n}=(\omega_e^+,\omega_e^-)\times \cap [\frac{k_n}{2^{m(n)}},\frac{k_n+1}{2^{(m(n))}}]$ for some $\{k_n\}_{n\ge0}$ and is a single point. 
\item[(P3)] By definition of $\Sigma^*,$ $\Sigma^*=\sqcup_{w\in(T)_0}\Sigma^*_w.$ Let $w\not\in T_e,$ then for any $\omega\in\Sigma^*_w,$ $\omega_e=w_e$ and 
\begin{align*}
&\cup_{v\in(T)_m\text{ and }v_e=w_e}\Mc{K}_v\\
=&\cup_{k\in [0,2^{m([w]-[w_e])}-1]_\Mb{Z}}(w_e^+,w_e^-)\times [\frac{k}{2^{m([w]-[w_e])}},\frac{k+1}{2^{m([w]-[w_e])}}]\\
=&(w_e^+,w_e^-)\times[0,1].
\end{align*}
Therefore 
\[\Cup_{w\in(T)_0}\Mc{K}_w=\hspace{-3pt}\Cup_{w\in(T)_0}\Cup_{\omega\in\Sigma^*_w}(\omega_e^+,\omega_e^-)\times[0,1]=\hspace{-3pt}\Cup_{\omega\in\Sigma^*}(\omega_e^+,\omega_e^-)\times[0,1]=\Mf{C}_G.\]
(The last equation follows from (PG2) and (PG3)).
\end{itemize}
\item Let $w\in T_e,$ then (PG1) ensures $K_w\supseteq \Cup_{\omega\in\Sigma^*_w}K_{\omega_e}.$ On the other hand, since (PG1) holds, for any $v\in T$ and $x\in K_v$ we can take $v'\in S(v)$ such that $x\in K_{v'}.$ Inductively use this fact, we can get $\omega\in\Sigma^*_w$ such that $x\in \omega_n$ for any $n,$ therefore $K_w=\Cup_{\omega\in\Sigma^*_w}K_{\omega_e}.$ Hence by definition of $\Mc{K},$ we get the desired result.
\end{enumerate}
\end{proof}

The following proposition plays the key role in the proof of the main theorem.
\begin{prop}\label{sore}
Let $d\in\Mc{D}_\infty(G)$ and fitting to $(G,E).$ Then 
\[1\le\ard(G,d)=\ard(\Mf{C}_G,d_{\Mf{C}}).\]
\end{prop}

\begin{proof}
We first show $\ard(G,d)\ge1.$ Let $\rho$ be a metric on $(G,E)$ and $\Mf{m}$ be a measure on $G$ such that $d\qs\rho$ and $\rho$ is $\alpha$-Ahlfors regular with respect to $\Mf{m}$ for some $\alpha>0.$ Then by Lemma \ref{qsfit}, $\rho$ fits to $(G,E)$ and so for any $n$, there exist $m\ge1$ and $\{x_i\}_{i=0}^m\subseteq G$ such that
\begin{itemize}
\item $x_i\in B_d(x_0,r)$ for any $i\in[0,m-1]_\Mb{Z}$ and $x_m\not\in B_d(x_0,r),$ 
\item $d(x_i,x_{i-1})\le r/2n$ and $x_i\sim x_{i-1}$ for any $i\in[1,m]_\Mb{Z}.$
\end{itemize}
Let $y_0=x_0$ and inductively choose $y_k$ by $y_k=x_{i_k}$ where 
\[i_k=\min\{i\mid x_i\not\in \Cup_{j=0}^{k-1}B_\rho(y_j,\frac{r}{2n})\}.\]
Note that 
\[\min_{0\le j\le k-1}\rho(y_j,y_k)<\frac{2r}{2n} \text{ because }\rho(y_k,x_{i_k-1})<\frac{r}{2n}\]
and so
\[\rho(y_0,y_k)<\frac{k}{n}r\text{ and }\Cup_{j=0}^k B_\rho(y_j,\frac{r}{2n})\subseteq B_\rho(y_0,r)\text { for any }k<n.\]
Therefore we can take $\{y_k\}_{k=0}^n.$ By definition, $\sqcup_{k=0}^{n-1}B_\rho(y_k,r/4n)\subseteq B_\rho(y_0,r),$ so 
\begin{align*}
Cr^\alpha &\ge \Mf{m}(B_\rho(x,r))\ge \sum_{k=0}^{n-1}\Mf{m}(B_\rho(y_k,r/4n)) \\
&\ge \sum_{k=0}^{n-1} C^{-1}((r/4n)\vee r_{y_k})^\alpha \\
&\ge 4^{-\alpha}C^{-1}n^{1-\alpha}r^\alpha,
\end{align*}
for some $C>0.$ Taking sufficiently large $n,$ this inequality contradicts if $\alpha<1.$ Therefore $\ard(G,d)\ge 1.$

Next we show $\ard(G,E)\le \ard(\Mf{C}_G,d_{\Mf{C}}).$ Let $\rho_X$ be a metric on $\Mf{C}_G$ and $\Mf{m_C}$ be a Borel measure on $(\Mf{C}_G,d_{\Mf{C}})$ such that $d_c\qs\rho_X$ and $\rho_X$ is $\alpha$-Ahlfors regular with respect to $\Mf{m_C}.$\\ Define $\rho$ on $G$ by $\rho(x,y)=\rho_X(x,y)$ for any $x,y\in G$ and it is clear that $d\qs \rho$ because $d(x,y)=d_{\Mf{C}}(x,y).$ \\
Define $\Mf{m}$ by 
\[\Mf{m}(\{x\})=\sum_{y\sim x}\Mf{m_C}(\iota(x,y)).\]
Let $r>0$, $x,y,z\in G$ such that $y\in B_\rho(x,r)\setminus\{x\}$ and $y\sim z.$ Then by (F1), there exists $C>1,$
$Cd_{\Mf{C}}(x,y)\ge d_{\Mf{C}}(y,z)\ge d_{\Mf{C}}(y,z')$ and hence $\theta(C)\rho_X(x,y)\ge \rho_X(y,z')$ for any $z'\in \iota(y,z).$ Moreover, if $r>r_{x,\rho}$ then there exists $o_0\in B_\rho(x,r)$ and $d(x,o')\le d(x,o)\le Cd(x,o_0)$ for any $o'$ such that $o'\in\iota (x,o)$ for some $o\in G$ with $o\sim x.$ Hence $o'\in B_{\rho_X}(x,\theta(C)r)$ and so for some $C'>0,$
\begin{align*}
\Mf{m}(B_\rho(x,r)) &=\sum_{y\in B_\rho(x,r)}\sum_{z\sim y} \Mf{m_C} (\iota(y,z))\\
 &\le 2\Mf{m_C} (B_{\rho_X}(x,(1+\theta(C))r)\le 2C'(1+\theta(C))^\alpha r^\alpha. 
 \end{align*}
(Remark that $\Mf{m_C}\{x\}=0$ for any $x\in G$ because of Ahlfors regularity.) On the other hand, for any $y'\in B_{\rho_X}(x,r)$ there exists $y,z\in G$ such that $y'\in\iota(y,z)$ and $d_{\Mf{C}}(x,y)\le d_{\Mf{C}}(x,y')$, and hence $\rho_X(x,y)\le\theta(1)\rho_X(x,y').$ Therefore
\[\Mf{m}(B_\rho(x,r))= \sum_{y\in B_\rho(x,r)}\sum_{z\sim y} \Mf{m_C}(\iota(y,z))\ge B_{\rho_X}(x,\theta(1)r)\ge C'^{-1}\theta(1)^\alpha r^\alpha.\]
We have shown $\rho$ is $\alpha$-Ahlfors regular with respect to $\Mf{m}$, so $\ard(G,E)\le \ard(\Mf{C}_G,d_{\Mf{C}}).$\\
Next we show $\ard(G,E)\ge \ard(\Mf{C}_G,d_{\Mf{C}}).$ Let $\rho$ be a metric on $G$ and $\Mf{m}$ be a measure on $G$ such that $d\qs\rho$ and $\rho$ is $\alpha$-Ahlfors regular with respect to $\Mf{m}.$ Note that $\alpha\ge 1$ because $\ard(G,d)\ge 1.$\\

\textbf{Claim.} $d_{\Mf{C}}\qs\rho_{\Mf{C},1/\alpha}$
\begin{proof}[Proof of claim]
Let $t>0$ and $x,y,z\in \Mf{C}_G$ with $d_{\Mf{C}}(x,y)\le td_{\Mf{C}}(x,z).$ We first consider the conditions
\begin{align}
\text{There exist }e\in E \text{ and }s,t\in(0,1)\text{ such that }
&x=(e,s)\text{ and }y=(e,t), \tag{$y*$}\\
\text{There exist }e\in E \text{ and }s,t\in(0,1)\text{ such that } &x=(e,s)\text{ and }z=(e,t). \tag{$z*$}
\end{align}
Let (Ch) be the condition such that neither ($y*$) nor ($z*$) hold.
Under the condition (Ch), we consider the following condition ($d_*$):
\begin{multline}
\text{ There exist }x_0,x_1,y_0\in G\text{ such that }x_0\sim x_1,\ x_1\sim y_0, \\ 
x\in\iota(x_0,x_1),\ y\in\iota(y_0,x_1)\text{ and } d_{\Mf{C}}(x,y)=d_{\Mf{C}}(x,x_1)+d_{\Mf{C}}(x_1,y). \tag{$d_*$}
\end{multline}
If ($d_*$) does not hold, then
\begin{multline*}
\text{ there exist }x_0,x_1,y_0,y_1\in G\text{ such that }x_1\ne y_1,\ x_0\sim x_1,\ y_1\sim y_0, \\ 
x\in\iota(x_0,x_1),\ y\in\iota(y_0,x_1)\text{ and } d_{\Mf{C}}(x,y)=d_{\Mf{C}}(x,x_1)+d_{\Mf{C}}(x_1,y_1)+d_{\Mf{C}}(y_1,y).
\end{multline*}
We also consider the following similar condition ($\rho_*$):
\begin{multline}
\text{ There exist }x_2,x_3,z_0\in G\text{ such that }x_2\sim x_3,\ x_3\sim z_0, \\ 
x\in\iota(x_2,x_3),\ z\in\iota(x_3,z_0)\text{ and } \rho_{\Mf{C},1/\alpha}(x,z)=\rho_{\Mf{C},1/\alpha}(x,x_3)+\rho_{\Mf{C},1/\alpha}(x_3,z). \tag{$\rho_*$}
\end{multline}
Otherwise,
\begin{multline*}
\hspace{-2.2pt}\text{there exist }x_2,x_3,z_0,z_1\in G\text{ such that }x_3\ne z_1, x_2\sim x_3, z_1\sim z_0, x\in\iota(x_2,x_3),\\
 z\in\iota(x_3,z_0)\text{ and } \rho_{\Mf{C},1/\alpha}(x,z)=\rho_{\Mf{C},1/\alpha}(x,x_3)+\rho_{\Mf{C},1/\alpha}(x_3,z_1)+\rho_{\Mf{C},1/\alpha}(z_1,z). 
\end{multline*}
We first prove the claim with (Ch) in four cases with these conditions. Recall that $d$ fits to $(G,E),$ so there exists $C>0$ such that for any $o,p,q\in G$ with $o\sim p\sim q,$ $d(o,p)\le Cd(p,q)$ and hence $\rho(o,p)\le \theta(C)\rho(p,q).$\\
\textbf{Case 1. Both ($d_*$) and ($\rho_*$) hold.} Let 
\[u=\ccdd{x_2}{x}{x_3}{z}{z_0},\]
then
\begin{align*}
(d_{\Mf{C}}(x,x_1)\vee d_{\Mf{C}}(x_1,y_1)) &\le d_{\Mf{C}}(x,y) \le td_{\Mf{C}}(x,z) \\
 & \le 2t(d_{\Mf{C}}(x,x_3)\vee d_{\Mf{C}}(x_3,z)) \le 2tu(d(x_2,x_3)\vee d(x_3,z_0)).
\end{align*}
Since $x\in\iota(x_0,x_1)\cup\iota(x_2,x_3),$ $l_E(x_1,x_3)\le 2$ and so 
\[(d(x_2,x_3)\vee d(x_3,z_0))\le C^2 (d(x_0,x_1)\wedge d(x_1,y_0)).\]
Hence
\[\ccdd{x_0}{x}{x_1}{y}{y_0}\le 2C^2tu,\]
therefore
\begin{align*}
\rho_{\Mf{C},1/\alpha} (x,y) &\le \rho_{\Mf{C},1/\alpha}(x,x_1)+\rho_{\Mf{C},1/\alpha}(x_1,y)\\
& \le 2(2C^2tu)^{1/\alpha}(\theta(C))^2(\rho(x_2,x_3)\wedge\rho(x_3,z_0))\\
& \le 2(2C^2tu)^{1/\alpha}(\theta(C))^2u^{-1/\alpha}(\rho_{\Mf{C},1/\alpha}(x,x_3)\vee\rho_{\Mf{C},1/\alpha}(x_3,z))\\
& \le \eta_1(t) \rho_{\Mf{C},1/\alpha}(x,z),
\end{align*}
where $\eta_1(t)=2^{(\alpha+1)/\alpha}C^{2/\alpha}(\theta(C))^2 t^{1/\alpha}.$\\
\textbf{Case 2. Neither ($d_*$) nor ($\rho_*$) hold.} Then,
\begin{align*}
d(x_1,y_1)&\le d_{\Mf{C}}(x,y)\le td_{\Mf{C}}(x,z) \\
&\le (1+2C)td(x_3,z_1)\le f_1(C)td(x_1,z_1),
\end{align*}
where $f_1(C)=(1+C+C^2)(1+2C).$ Therefore
\begin{align*}
\rho_{\Mf{C},1/\alpha}(x,y)&\le (1+2\theta(C))\rho(x_1,y_1)\\
&\le (1+2\theta(C))\theta(f_1(C) t)\rho(x_1,z_1) \le \eta_2(t)\rho_{\Mf{C},1/\alpha}(x,z),
\end{align*}
where $\eta_2(t)=f_1(\theta(C))\theta(f_1(C) t).$\\
\textbf{Case 3. Only ($d_*$) holds.} Then
\[(d_{\Mf{C}}(x,x_1)\vee d_{\Mf{C}}(x_1,y_0))\le d_{\Mf{C}}(x,y)\le td_{\Mf{C}}(x,z) \le (1+2C)td(x_3,z_1).\]
Let $u=\ccdd{x_0}{x}{x_1}{y}{y_0},$ then
\begin{align*}
\frac{1}{C}(d(x_0,x_1)\vee d(x_1,y_0)) &\le (d(x_0,x_1)\wedge d(x_1,y_0))\le\frac{t}{u}d_{\Mf{C}}(x,z)\\
&\le (1+2C)\frac{t}{u}d(x_3,z_0) \le f_1(C)\frac{t}{u}d(x_1,z_0),
\end{align*}
and hence
\begin{align*}
\rho_{\Mf{C},1/\alpha}(x,y) &\le 2 u^{1/\alpha}(\rho(x_0,x_1)\vee\rho(x_1,y_0))\\
& \le 2u^{1/\alpha}(\theta(C)\wedge\theta(Cf_1(C)\frac{t}{u}))\rho(x_1,z_0)\\
& \le 2u^{1/\alpha}f_1(\theta(C))(\theta(C)\wedge\theta(Cf_1(C)\frac{t}{u}))\rho_{\Mf{C},1/\alpha}(x,z).
\end{align*}
Since $u\wedge\frac{t}{u}\le\sqrt{t}$ and $u\le1,$ $\rho_{\Mf{C},1/\alpha}(x,y)\le \eta_3(t)\rho_{\Mf{C},1/\alpha}(x,z)$ where\\
$\eta_3(t)=2f_1(\theta(C))(t^{1/2\alpha}\theta(C)\vee\theta(Cf_1(C)\sqrt{t})).$
\\
\textbf{Case 4. Only ($\rho_*$) holds.} Then,
\[d_{\Mf{C}}(x_1,y_1)\le d_{\Mf{C}}(x,y)\le td_{\Mf{C}}(x,z) \le 2t(d_{\Mf{C}}(x,x_3)\vee d_{\Mf{C}}(x_3,z)).\]
Let $u=\ccdd{x_2}{x}{x_3}{z}{z_0}$, then
\begin{equation}
d(x_3,y_1)\le (1+C+C^2)d(x_1,y_1)\le (1+C+C^2)2tu(d(x_2,x_3)\vee d(x_3,z_0),\label{case4}
\end{equation}
and hence
\begin{align*}
&\rho_{\Mf{C},1/\alpha}(x,y)\\
\le& f_1(\theta(C))\rho(x_3,y_1)\\
\le& f_1(\theta(C))\theta(C)\theta((1+C+C^2)2tu)(\rho(x_2,x_3)\wedge\rho(x_3,z_0))\\
\le& f_1(\theta(C))\theta(C)u^{-1/\alpha}\theta((1+C+C^2)2tu)(\rho_{\Mf{C},1/\alpha}(x,x_3)\vee\rho_{\Mf{C},1/\alpha}(x_3,z)).
\end{align*}
Since $Cd(x_3,y_1)\ge d(x_2,x_3)\vee d(x_3,z_0)$ and \eqref{case4} hold, $1\le u^{-1}\le C(1+C+C^2)2t.$ Therefore $\rho_{\Mf{C},1/\alpha}(x,y)\le \eta_4(t)\rho_{\Mf{C},1/\alpha}(x,z)$ where \[\eta_4(t)=f_1(\theta(C))\theta(C)(C+C^2+C^3)^{1/\alpha}\theta((1+C+C^2)2t)t^{1/\alpha}.\]
Finally take $\eta=\eta_1\vee\eta_2\vee\eta_3\vee\eta_4,$ we have shown $d_{\Mf{C}}\qs\rho_{1,\alpha}$ under (Ch).
Next we prove general case by using the result for (Ch). \\
\textbf{Case 5. Both $(y*)$ and $(z*)$ hold.} 
Then $\rho_{\Mf{C},1/\alpha}(x,y)\le\eta_5(t)\rho_{\Mf{C},1/\alpha}(x,z)$ whenever $d_{\Mf{C}}(x,y)\le td_{\Mf{C}}(x,z)$ where $\eta_5(t)=t^\alpha.$\\
\textbf{Case 6. Only $(y*)$ holds.}
We can take $w\in G$ such that $\rho_{\Mf{C},1/\alpha}(x,w)+\rho_{\Mf{C},1/\alpha}(w,z)=\rho_{\Mf{C},1/\alpha}(x,z).$ If $\sqrt{t}d_{\Mf{C}}(w,z)\ge d_{\Mf{C}}(x,w)$ then
\[d_{\Mf{C}}(x,w)\vee d_{\Mf{C}}(y,w)  \le d_{\Mf{C}}(x,w)+d_{\Mf{C}}(x,y) \le (\sqrt{t}+t(1+\sqrt{t}))d_{\Mf{C}}(w,z),\]
and hence
\[\rho_{\Mf{C},1/\alpha}(x,y)\le\rho_{\Mf{C},1/\alpha}(x,w)\vee\rho_{\Mf{C},1/\alpha}(y,w)\le \eta(\sqrt{t}+t(1+\sqrt{t}))\rho_{\Mf{C},1/\alpha}(w,z).\]
Otherwise $d_{\Mf{C}}(x,y)\le (\frac{1+\sqrt{t}}{\sqrt{t}}t)d_{\Mf{C}}(x,w).$ Therefore
\[\rho_{\Mf{C},1/\alpha}(x,y)\le \eta_6(t)(\rho_{\Mf{C},1/\alpha}(w,z)\vee\rho_{\Mf{C},1/\alpha}(x,w))\le \eta_6(t)\rho_{\Mf{C},1/\alpha}(x,z),\]
where $\eta_6(t)= \eta(\sqrt{t}+t(1+\sqrt{t}))\vee ((1+\sqrt{t})\sqrt{t})^{1/\alpha}.$
\\
\textbf{Case 7. Only $(z*)$ holds.}
We can take $w\in G$ such that $d_{\Mf{C}}(x,w)+d_{\Mf{C}}(w,y)=d_{\Mf{C}}(x,y).$ Since
\[ d_{\Mf{C}}(w,y)\le d_{\Mf{C}}(x,y) \le t d_{\Mf{C}}(x,z)\le t(d_{\Mf{C}}(x,w)\vee d_{\Mf{C}}(z,w)),\]
we get $\rho_{\Mf{C},1/\alpha}(w,y)\le\eta(t)(\rho_{\Mf{C},1/\alpha}(x,w)\vee\rho_{\Mf{C},1/\alpha}(z,w)).$ Moreover,
\[ d_{\Mf{C}}(x,w)\vee d_{\Mf{C}}(z,w)\le d_{\Mf{C}}(x,z)+d_{\Mf{C}}(y,z)\le (1+t)d_{\Mf{C}}(x,z),\]
hence $(\rho_{\Mf{C},1/\alpha}(x,w)\vee\rho_{\Mf{C},1/\alpha}(z,w))\le(1+t)^{1/\alpha}\rho_{\Mf{C},1/\alpha}(x,z).$ Therefore
\[\rho_{\Mf{C},1/\alpha}(x,y)\le \rho_{\Mf{C},1/\alpha}(x,w)+\rho_{\Mf{C},1/\alpha}(w,y)\le \eta_7(t)\rho_{\Mf{C},1/\alpha}(x,z)\]
where $\eta_7(t)=t^{1/\alpha}+(1+t)^{1/\alpha}\eta(t)$\\
Let $\eta'=\eta\vee\eta_5\vee\eta_6\vee\eta_7,$ then we have shown $d_{\Mf{C}}$ is $\eta'$-quasisymmetric to $\rho_{\Mf{C},1/\alpha}.$
\end{proof}

Finally, we prove $\rho_{\Mf{C},1/\alpha}$ is $\alpha$-Ahlfors regular. 
Define $\Mf{m_C}$ by $\Mf{m_C}(\{(x,y,t)\in \Mf{C}_G\mid a\le t\le b\})=(b-a)(\Mf{m}(\{x,y\}))$ for any $(x,y)\in E$ and $0\le a< b\le 1.$ First we consider the case $x\in G.$
\begin{itemize}
\item Let $x\in G.$ Since $r_x=r_{x,\rho}\asymp r_y\asymp \rho(x,y)$ for any $(x,y)\in E$ and $\rho$ is $\alpha$-Ahlfors regular with respect to $\Mf{m},$ 
\[V_{\rho_{\Mf{C},1/\alpha},\Mf{m_C}}(x,r)=\hspace{-1pt} \sum_{y:y\sim x}
\left(\frac{r}{\rho_{\Mf{C},1/\alpha}(x,y)}\right)^\alpha \Mf{m}(\{x,y\})
\asymp\hspace{-1pt} \sum_{y:y\sim x}\left(\frac{r}{r_x}\right)^\alpha\hspace{-1pt} r_x^\alpha\asymp r^\alpha\]
for any $x\in G$ and $r\le r_x$ (the last inequality is because $G$ is bounded degree). Next we consider global cases.
\begin{itemize}
\item Let $(y,z)\in E.$ If $y\in B_\rho(x,r)$ for some $r>0,$ then for any $y'\in\iota(y,z),$
\[\rho_{\Mf{C},1/\alpha}(x,y')\le\rho(x,y)+\rho(y,z)<(1+\theta(C))r.\]
\item Let $y'\in B_{\rho_{\Mf{C},1/\alpha}}(x,r),$ then there exists $(y,z)\in E$ such that $y'\in\iota(y,z)$ and $y\in B_\rho(x,r).$
\end{itemize}
Therefore there exists $C_1,C_2>0$ such that
\begin{multline*}
V_{\rho_{\Mf{C},1/\alpha},\Mf{m_C}}(x,r)\ge \frac{1}{2}\sum_{y\in B_\rho(x,r/(1+\theta(C)))}\sum_{z\sim y} \Mf{m_C}(\iota(y,z)) \\
=\frac{1}{2}\sum_{y\in B_\rho(x,r/(1+\theta(C)))}\sum_{z\sim y} \Mf{m}(\{y,z\})\ge V_{\rho,\Mf{m}}(x,r/(1+\theta(C)))\ge C_1r^\alpha,
\end{multline*}
for any $x\in G$ and $r>(1+\theta(C))r_x$, and
\begin{multline*}
V_{\rho_{\Mf{C},1/\alpha},\Mf{m_C}}(x,r)\le \sum_{y\in B_\rho(x,r)}\sum_{z\sim y}\Mf{m}(\{y,z\}) \\
\le (\sup_{y\in G}\#\{z\mid z\sim y\})V_{\rho,\Mf{m}}(x,(1+\theta(C))r) \le C_2r^\alpha,
\end{multline*}
for any $x\in G$ and $r>r_x$. Adjusting constants, we get $V_{\rho_{\Mf{C},1/\alpha},\Mf{m_C}}(x,r)\asymp r^\alpha$ for any $r\in G$ and $r>0.$
\item Let $x\in \Mf{C}_G$ and $y,z\in G$ such that $x\in\iota(y,z).$ We also let $r_0=\rho_{\Mf{C},1/\alpha}(x,y).$ If $r\le r_0/(1+\theta(C)),$ then $B_{\rho_{\Mf{C},1/\alpha}}(x,r)\subseteq\cup_{w\sim y}\iota(w,y)\cup\cup_{w\sim z}\iota(w,z),$ so there exist $C',C_3$ such that
\begin{align*}
&V_{\rho_{\Mf{C},1/\alpha},\Mf{m_C}}(x,r)\\
\le&\sum_{w\sim y} \left(\frac{r}{\rho_{\Mf{C},1/\alpha}(y,w)}\right)^\alpha\Mf{m}(\{w,y\})+\sum_{w\sim z} \left(\frac{r}{\rho_{\Mf{C},1/\alpha}(w,z)}\right)^\alpha\Mf{m}(\{w,z\})\\
\le& \sum_{w\sim y}\left(\frac{\theta(C)r}{r_y}\right)^\alpha C'(r_y^\alpha+(\theta(C)r_y)^\alpha)+\sum_{w\sim z}\left(\frac{\theta(C)^2r}{r_y}\right)^\alpha C'(2(\theta(C)^2r_y)^\alpha)\\
\le& C_3\left(\frac{r}{r_y}\right)^\alpha r_y^\alpha=C_3r^\alpha,
\end{align*}
because $(G,E)$ is bounded degree. Otherwise, since \[
B_{\rho_{\Mf{C},1/\alpha}}(x,r)\subseteq B_{\rho_{\Mf{C},1/\alpha}}(y,r)\cup B_{\rho_{\Mf{C},1/\alpha}}(x,r)\text{ if }r>r_0,\]
there exists $C_4>0$ such that $V_{\rho_{\Mf{C},1/\alpha},\Mf{m_C}}(x,r)\le C_4 r^\alpha$ for any $r>r_0$ and $x\in \Mf{C}_G.$ 
 On the other hand, there exist $C_5,C_6$ such that
\[V_{\rho_{\Mf{C},1/\alpha},\Mf{m_C}}(x,r)\ge \left(\frac{r}{\rho_{\Mf{C},1/\alpha}(y,z)}\right)^\alpha\Mf{m}(\{y,z\})\ge C_5\left(\frac{r}{r_y}\right)^\alpha r_y^\alpha=C_5r^\alpha,\]
for any $x\in \Mf{C}_G$ and $r<r_0(x)$, and
\[V_{\rho_{\Mf{C},1/\alpha},\Mf{m_C}}(x,r)\ge V_{\rho_{\Mf{C},1/\alpha},\Mf{m_C}}(y,r/2)\ge C_6\left(\frac{r}{2}\right)^\alpha,\]
for any $x\in \Mf{C}_G$ and $r\ge 2r_0(x).$ Therefore
\[\left(\frac{C_5\vee C_6}{2^\alpha}\right)r^\alpha\le V_{\rho_{\Mf{C},1/\alpha},\Mf{m_C}}(x,r)\le (C_3\vee(1+\theta(C))^\alpha C_4)r^\alpha,\]
for any $r>0$ and $x\in \Mf{C}_G.$
\end{itemize}
\end{proof}

We can also obtain the following lemma.
\begin{lem}\label{junbi}
Under the same assumption of Theorem \ref{Smain1}, $d_{\Mf{C}}$ satisfies the assumptions of Theorem \ref{sigmain1} with respect to the partition $\Mc{K}_r$ hold.
\end{lem}
\begin{proof}
\begin{itemize}
\item $\sup_{w\in T_r}\#(S(w))\le\sup_{w\in T_e\setminus \Lambda_e}\#(S(w))\vee(1-\frac{\log r}{\log2})<\infty.$
\item (locally finite) Since $(G,E)$ is locally finite and $d$ fits to $(G,E)$, for any $x\in G$, there exists $r>0$ such that $B_{d_{\Mf{C}}}(x,r)\subset\sum_{y:y\sim x}\iota(x,y)$ and then $\#\{w\in(T)_0\mid B_{d_{\Mf{C}}}(x,r)\cap\Mc{K}_w\ne\emptyset\}<\infty.$ Let $U_w:=\Mc{K}_w\cup (\cup_{x\in K_w}B_{d_{\Mf{C}}}(x,r)),$ we get $ \#\{v\in(T)_0\mid U_w\cap\Mc{K}_v\ne\emptyset\}<\infty.$
\item (minimal) Since $K$ is minimal, it directly follows from Lemma \ref{lemKMcK} (2) and the definition of $\Mc{K}.$
\item ($r^{[w]}\asymp d_{\Mf{C}}$) Let $w\in T_e.$ For any $x,y\in \Mc{K}_w$, there exists $x_0,x_1,y_0,y_1\in K_w$ such that $x\in\iota(x_0,x_1),y\in\iota(y_0,y_1)$ and then $d_{\Mf{C}}(x,y)\le d_{\Mf{C}}(x_0,x_1)+d_{\Mf{C}}(x_1,y_1)+d_{\Mf{C}}(y_1,y_0),$ so $d(w)\le d_{\Mf{C}}(w)\le 3d(w)$ and so $d_{\Mf{C}}(w)\asymp r^{[w]}$ for any $w\in T_e.$ Moreover, by the definition of $\Mc{K}_w$ and $d_{\Mf{C}},$ 
\[2^{-1}r^{[w]-[w_e]}d_{\Mf{C}}(w_e)\le d_{\Mf{C}}(w)=\diam(\Mc{K}_w,d_{\Mf{C}})\le r^{[w]-[w_e]}d_{\Mf{C}}(w_e)\]
for any $w\not\in T_e.$ Combining them, we get $r^{[w]}\asymp d_{\Mf{C}}(w)$ for any $w\in T_r.$
\item (uniformly finite) Let $w\in T_e\cup \Lambda^{d_{\Mf{C}}}_s$ for some $s>0.$ Since $d_{\Mf{C}}\asymp d\asymp r^{[w]}$ and $d(w)\le d_{\Mf{C}}(w)$ for $w\in T_e$, there exist $c<1$ and $m>0$ such that if $v\in\Lambda^{d_{\Mf{C}}}_{s,1}(w)\cup T_e$ then $cs\le d(v)\le s$ and hence there exists $v'\in \Lambda^d_s$ such that $v\in T_{v'}$ and $[v]-[v']<m.$ On the other hand, let $v\in \Lambda_e$ such that $K_v\cap K_w\ne \emptyset$ and $d_{\Mf{C}}(v)=d(v)>s.$ Then $s<d(v)\le C d(w)\le Cs$ by (F1), and similarly, there exists $m_1>0$, which is independent of $v$, and $v'\in \Lambda^d_{\Mf{C}s}$ such that $v\in T_{v'}$ and $[v]-[v']<m_1.$ Therefore
  \begin{align*}
  & \#(\Lambda^{d_{\Mf{C}}}_{s,1}(w))\\
=& \#(\Lambda^{d_{\Mf{C}}}_{s,1}(w)\cup T_e)+ \#(\Lambda^{d_{\Mf{C}}}_{s,1}(w)\setminus T_e)\\
\le&  \#(\Lambda^{d_{\Mf{C}}}_{s,1}(w)\cup T_e)+ 2\#(\{v\in\Lambda_e\mid d(v)>s\text{ and }K_v\cap K_w\ne\emptyset \})\\
\le& (\sup_{v\in T_e}\#(S(v)))^m \#(\Lambda^{d}_{s,1}(w'))+ 2(\sup_{v\in T_e}\#(S(v)))^{m_1} \#(\Lambda^{d}_{s,1}(w'')),
  \end{align*}
where $w'\in \Lambda^{d}_s$ and $w''\in\Lambda^{d}_{Cs}$ such that $w\in T_{w'}\subset T_{w''}.$ Since $d$ is uniformly finite, $ \#(\Lambda^{d_{\Mf{C}}}_{s,1}(w))$ is bounded. Moreover, since $ \#(\Lambda^{d_{\Mf{C}}}_{s,1}(w))\le \sup_{x\in G}\#(\{y\mid y\sim x\})+1$ for $s>0$ and $w\in \Lambda^{d_{\Mf{C}}}_s\setminus T_e,$ $d_{\Mf{C}}$ is uniformly finite.
\item(thick) Let $w\in\Lambda^{d_{\Mf{C}}}_s\setminus T_e,$ then $\Mc{K}_w=\{((x,y),t)\mid a\le t\le b\}$ for some $(x,y)\in E,0\le a<b\le 1,$ and then $U^{d_{\Mf{C}}}_1(((x,y),(t+s)/2),(r/8)d_{\Mf{C}}(\pi(w)))\linebreak\subseteq \Mc{K}_w.$ Next we let $w\in T_e.$ Since $d$ is thick and $d\asymp d_{\Mf{C}}$, there exists $\alpha$, independent of $w$, and $x\in K_w$ such that $U^d_1(x,\alpha d_{\Mf{C}}(\pi(w)))\subseteq \Mc{K}_w.$
\begin{itemize}
\item[$\circ$]  If $\Lambda^{d_{\Mf{C}}}_{\alpha d_{\Mf{C}}(\pi(w)),1}(x)\setminus T_e\ne\emptyset,$ then there exists $v\in\Lambda_e$ such that $K_v\cap K_w\ne \emptyset$ and $d(v)>\alpha d_{\Mf{C}}(\pi(w))$, and hence by (F1), there also exists $w'\in\Lambda_e$ such that $w'\in T_w$ and $d(w')>(\alpha/C)d_{\Mf{C}}(\pi(w)).$ Then similar to the former case, $U^{d_{\Mf{C}}}_1(x',(\alpha/4C)d_{\Mf{C}}(\pi(w)))\subseteq \Mc{K}_{w'}\subseteq \Mc{K}_w $ for some $x'\in\Mc{K}_w.$
\item[$\circ$] If $\Lambda^{d_{\Mf{C}}}_{\alpha d_{\Mf{C}}(\pi(w)),1}(x)\subseteq T_e,$ then for any $v\in \Lambda^{d_{\Mf{C}}}_{\alpha d_{\Mf{C}}(\pi(w)),1}(x),$ there exists $v'\in\Lambda^d_{\alpha d_{\Mf{C}}(\pi(w)),1}(x)$ such that $v\in T_{v'}$ because $d(v)\le d_{\Mf{C}}(v)$ and $v\in T_e,$ so $U^{d_{\Mf{C}}}_1(x,\alpha\pi(d_{\Mf{C}}(w)))\in \Mc{K}_w.$ 
\end{itemize}
Therefore $d_{\Mf{C}}$ is thick.
\item (($M_*+1$)-adapted) Let $x\in \Mf{C}_G.$ If $y\in U^{d_{\Mf{C}}}_{M_*+1}(x,r)$, then there exist $w_0,w_1,...,w_{M_*+1}\in \Lambda^{d_{\Mf{C}}}_r$ such that $x\in  K_{w_0},\ y\in K_{w_{M_*+1}}$ and $K_{w_i}\cup K_{w_{i+1}}\ne\emptyset$ for any $i\in[0,M_*]_\Mb{Z}$, so $d_{\Mf{C}}(x,y)\le \sum_{i=0}^{M_*+1}d_{\Mf{C}}(w_i)\le (M_*+2)r$ and hence $U^d_M(x,r)\subseteq B_d(x,(M_*+3)r).$ To show inverse direction, we take $y,z\in G$ such that $x\in\iota(y,z).$ If $r<d(y,z)/C,$ then
\[B_{d_{\Mf{C}}}(x,r)\subseteq\left(\Cup_{w:w\sim z}\iota(w,z)\right) \cup \left(\Cup_{w:w\sim y}\iota(w,y)\right),\]
so $B_{d_{\Mf{C}}}(x,r)\subseteq U^1_d(x,2r)\subseteq U^{M*+1}_d(x,2r)$ by the definition of $\Mc{K}$ on $T\setminus T_e.$ If $r>d(y,z)$, then $B_{d_{\Mf{C}}}(x,r)\subseteq B_{d_{\Mf{C}}}(y,r)\cup B_{d_{\Mf{C}}}(z,r).$ Recall that if $p\in B_{d_{\Mf{C}}}(x,r)$ and $p\in\iota(p_1,p_2)$ for $p\in \Mf{C}_G$ and $p_1,p_2\in G,$ then $\{p_1,p_2\}\subseteq B_{d_{\Mf{C}}}(x,(1+C)r)$, so 
\begin{multline*}
B_{d_{\Mf{C}}}(y,r)\subset \{\iota(p_1,p_2)\mid p_1,p_2\in B_{d_{\Mf{C}}}(y,(1+C)r)\}\\
 \subseteq \Cup\{\Mc{K}_v\mid v\in \Lambda^d_{\alpha(1+C)r,M_*}(y)\subseteq U^{d_{\Mf{C}}}_{M_*}(y,\alpha(1+C)r),
\end{multline*}
because $d$ is $M_*$-adapted. (Note that the last inclusion follows from $d_{\Mf{C}}(w)\ge d(w)$, similar to the proof of thick.)
Since $r>d_{\Mf{C}}(y,z)$, $y,z\in U^{d_{\Mf{C}}}_0(x,r)$ and hence $B_{d_{\Mf{C}}}(x,r)\subseteq U^{d_{\Mf{C}}}_{M_*+1}(y,\alpha(1+C)r).$ Adjusting constants, we get $d_{\Mf{C}}$ is ($M_*+1$)-adapted.
\end{itemize}
\end{proof}
Using these, we can prove Theorem \ref{Smain1}.
\begin{proof}[Proof of Theorem \ref{Smain1}]
Let $(w,v)\in T_m$ and suppose $\Mc{K}_w\cap \Mc{K}_v$ does not intersect with $G.$ Then $\Mc{K}_w\cap\Mc{K}_v\ne\emptyset$ if and only if 
\begin{align*}
w&=\{(x,i),(y,2^{n(m-[w_e])}-1-i)\}_{w,m-[w_e]}\text{ and }\\
v&=\{(x,j),(y,2^{n(m-[v_e])}-1-j)\}_{w,m-[v_e]}\end{align*}
with $w_e=v_e$ and $|i-j|=1.$ Therefore 
\[ J^h_M=J^h_M(K)=J^h_M(\Mc{K}).\]
Hence $\ul{I}_\Mc{E}, \ol{I}_\Mc{E}, \ul{d}^S_p$ and $\ol{d}^S_p$ for $K$ and $\Mc{K}$ coincide respectively. Therefore by Lemma \ref{junbi} and Theorem \ref{sigmain1},
\[\ul{I}_\Mc{E}(N_1,N_2,N)=\ol{I}_\Mc{E}(N_1,N_2,N)=\ard(\Mf{C}_G,d_{\Mf{C}}) \]
for $N_2\ge N_1+M^*+1.$ Combining this with Proposition \ref{sore}, we get
\[\ul{I}_\Mc{E}(N_1,N_2,N)=\ol{I}_\Mc{E}(N_1,N_2,N)=\ard(G,d). \]
We also obtain (2) and (3).
\end{proof}

\subsection{Spectral dimension and Ahlfors regular conformal dimension of weighted graphs} 
In Theorem \ref{Smain1}, we saw a relation between the ARC dimension and the $p$-spectral dimension of associated metrics on graphs. On the other hand, the spectral dimension of the associated random walks on graphs can be determined. In this subsection, we consider the relation between these dimensions. Recall that $(G,E)$ is a connected, bounded degree simple graph and $\Mc{T} = (T,\pi,\phi) = (T_r, \pi', \phi')$ is a bi-infinite tree with a reference point. Throughout this section, let $K$ be a partition of $(G,E)$ parametrized by $\Mc{T}.$    

\begin{defi}[Weighted graph]
Let $\mu$ be a positive symmetric function on $E$, then we call $( G,\mu)$ a weighted graph and $\mu$ a conductance (or a weight) on $(G,E)$. Moreover, we treat $\mu$ as a measure on $ G$ defined by
\[ \mu_x:=\sum_{y:y\sim x}\mu_{xy} \text{ and }\mu(A):=\sum_{x\in A} \mu_x \]
for any $x\in G$ and $A\subset  G.$    
\begin{itemize}
\item(Controlled weight).
We say $( G,\mu)$ has controlled weight, or satisfies condition ($p_0$) if there exists $p_0>0$ such that
\begin{equation}
 p(x,y):=\frac{\mu_{xy}}{\mu_{x}}\ge p_0 \text{ for any }x,y \in G\text{ with }x\sim y. \tag{$p_0$}
\end{equation}
Note that if $(G,\mu)$ has controlled weight, then $\#\{y\mid y\sim x\}\le\lfloor p_0^{-1}\rfloor$ for any $x\in T.$ (It shows that $(G,E)$ must be a bounded degree graph). 
\item(Heat kernel).
We inductively define 
\[p_0(x,y)=\delta_{x,y},\hspace{20pt}p_n(x,y)=\sum_{z\in G}p_{n-1}(x,z)p(z,y).\]
$p_n(x,y)$ is also thought as transition function of associated random walk; that is,
\[ \mathbb{P}^x(X_n=y)=p_n(x,y). \]
Additionally, we define the heat kernel of this random walk (with respect to $\mu$) by $h_n(x,y)=p_n(x,y)/\mu_y.$ It is easy that $h_n(x,y)=h_n(y,x).$
\item(Effective resistance).
For $f\in \mathbb{R}^ G,$ we define
\[\mathcal{E}_\mu(f)=\mathcal{E}(f):=\frac{1}{2}\sum_{(x,y)\in E}(f(x)-f(y))^2\mu_{xy}. \]
For any $A,B\subset  G,$ we also define the effective resistance of $( G,\mu)$ by
\[ R(A,B)=(\inf\{\mathcal{E}(f)|f|_A=1,f|_B=0 \})^{-1} \]
where $\inf\emptyset=\infty.$
We write $R(x,A)$ (resp. $R(x,y)$) instead of $R(\{x\},A)$ (resp. $R(\{x\},\{y\})$). 
\end{itemize}
\end{defi}
It is known that the infimum of $R(A,B)^{-1}$ is attained and that $R(x,y)$ is a distance on $ G$ (for example, see~\cite{Kig01}).\\
For the rest of this paper, $(G,\mu)$ is always a weighted graph and $R$ is the associated effective resistance.
\begin{defi}[Spectral dimension]
Fix $x\in G$ and define 
\[\ol{d}_S(G,\mu)= 2\limsup_{n\to\infty}\frac{-\log p_{2n}(x,x)}{\log n}\quad\text{and}\quad \ul{d}_S(G,\mu)= 2\liminf_{n\to\infty}\frac{-\log p_{2n}(x,x)}{\log n}.\]
\end{defi}
We can see that $\ol{d}_S(G,\mu)$ and $\ul{d}_S(G,\mu)$ are independent of $x.$
We call $\ol{d}_S(G,\mu)$ the upper spectral dimension of $(G,\mu)$, and $\ul{d}_S(G,\mu)$ the lower spectral dimension of $(G,\mu).$ If $\ol{d}_S(G,\mu) = \ul{d}_S(G,\mu),$ then we call $d_S(G,\mu)=\ol{d}_S(G,\mu)$ the spectral dimension of $(G,\mu).$\\

We introduce other notions related to the partition.
\begin{defi}
We say $K$ is connected if for any $w\in T_e$ and $x,y\in K,$ there exists a path between $x$ and $y$ in $K_w,$ in other words, $(K_w,E|_{K_w\times K_w})$ is connected for any $w\in T_e.$  
\end{defi}
\begin{defi}
We define $\ul{\Mc{N}},\ol{\Mc{N}},\ul{\Mc{R}}_p$ by 
\begin{align*}
&\ol{\Mc{N}}=\sup_{w\in T}\limsup_{k\to\infty}\#(\{S^k(\pi^k(w))\})^{1/k},\hspace{10pt}
\ul{\Mc{N}}=\sup_{w\in T}\liminf_{k\to\infty}\#(\{S^k(\pi^k(w))\})^{1/k},\\
&\ul{\Mc{R}}_p(N_1,N_2,N)=\sup_{w\in T}\liminf_{k\to\infty}\epkw{\pi^k(w)}^{1/k}. 
\end{align*}
\end{defi}
Remark that the difference between $\ol{N}_*$ (resp. $\ul{R}_p$) and $\ol{\Mc{N}}$ (resp. $\ol{\Mc{R}}_p$) is the order of the supremum over $w\in T$ and the limit as $k,$ the index of scales, approaches to infinity. By definition, $\ol{N}_*\ge\ol{\mathcal{N}}\ge{\ul{\mathcal{N}}}$ and $\ul{R}_p\ge\ul{\mathcal{R}}_p.$ \\
\begin{lem}\label{tilt}
Assume $\sup_{w\in T_e\setminus \Lambda_e}\#(S(w))<\infty,$ then
  \begin{enumerate}
  \item $\displaystyle\ol{\Mc{N}}=\limsup_{k\to\infty}\#(\{S^k(\pi^k(w))\})^{1/k}$ and $\displaystyle\ul{\Mc{N}}=\liminf_{k\to\infty}\#(\{S^k(\pi^k(w))\})^{1/k}$ for any $w\in T.$
  \item $\displaystyle\ul{\Mc{R}}_p(N_1,N_2,N)=\sup_{l\ge0 }\liminf_{k\to\infty}\epkw{\pi^{k+l}(w)}^{1/k}$ for any $w\in T.$
  \end{enumerate}
\end{lem}
\begin{proof}
Let $N_*=\sup_{w\in T_e\setminus \Lambda_e}\#(S(w))$
  \begin{enumerate}
  \item Let $w\in T$ and $l\ge 0.$ Then for any $k\ge l,$ $\#(S^{k-l}(\pi^k(w)))\le \#(S^k(\pi^k(w)))\linebreak \le N_*^l\#(S^{K-l}(\pi^k(w))).$ So
    \begin{align*}
     N_*^{-1/(k-l)}\#(S^{k-l}(\pi^{k-l}(\pi^l(w))))^{1/(k-l)}&\le N_*^{-1/(k-l)}\#(S^k(\pi^k(w))))^{1/(k-l)}\\%
&\le  \#(S^k(\pi^k(w))))^{1/k} \\
 &\le N_*^{l/k}\#(S^{k-l}(\pi^{k-l}(\pi^l(w)))^{1/(k-l)},
    \end{align*}
because $\#(S^k(\pi^k(w)))^{1/(k-l)-1/k}\le (N_*^k)^{1/k(k-l)}=N_*^{1/(k-l)}.$ Therefore $\limsup_{k\to\infty}(\#\{S^k(\pi^k(w))\})^{1/k}=\limsup_{k\to\infty}(\#\{S^k(\pi^k(\pi^l(w)))\})^{1/k}.$ By (P1), for any $w,v\in T,$ there exists $n,m\ge0$ such that $\pi^l(w)=\pi^n(v)$ and hence
\begin{align*}
\ol{\Mc{N}}&=\sup_{v\in T}\limsup_{k\to\infty}\#(\{S^k(\pi^k(v))\})^{1/k}\\
&=\sup_{l\ge0}\limsup_{k\to\infty}\#(\{S^k(\pi^k(\pi^l(w)))\})^{1/k}\\
&=\limsup_{k\to\infty}\#(\{S^k(\pi^k(v))\})^{1/k}.
\end{align*} 
The case of $\ul{\Mc{N}}$ is the same.
\item Let $w\in T$ and $k\ge l\ge0.$ We also let $f$ be a function on $(T)_{[w]-l}$ such that $f\equiv 1$ on $S^{k-l}(\Gamma_{N_1}(\pi^k(w))$ and $f\equiv 0$ on $S^{k-l}(\Gamma_{N_2}(\pi^k(w)))^c.$ Then for $\ol{f}(w)=f(\pi^l(w))$, $\ol{f}\equiv 1$ on $S^k(\Gamma_{N_1}(\pi^k(w))$, $\ol{f}\equiv 0$ on $S^k(\Gamma_{N_1}(\pi^k(w)).$ Moreover,
  \begin{align*}
    \sum_{(u,v)\in J^h_{N,[w]}}|\ol{f}(u)-\ol{f}(v)|^p &\le \sum_{(u,v)\in J^h_{N,[w]-l}}\sum_{u'\in S^l(u)}\sum_{v'\in S^l(v)}|\ol{f}(u')-\ol{f}(v')|^p \\
& \le (N_*)^{2l}\sum_{(u,v)\in J^h_{N,[w]-l}}|f(u)-f(v)|^p,
  \end{align*}
because $l_{J^h_{[w]}}(u,v)\le N$ implies $l_{J^h_{[w]-l}}(\pi(u),\pi(v)) \le N.$ Therefore 
\[(N_*)^{-2l}\epkw{\pi^k(w)}\le\Mc{E}_{p,k-l,\pi^{k-l}(\pi^l(w))}(N_1,N_2,N),\] and same as the former case, we get 
\[\liminf_{k\to\infty}(\epkw{\pi^k(w)})^{1/k}\le\liminf_{k\to\infty}(\epkw{\pi^k(\pi^l(w))})^{1/k},\]
and 
\begin{align*}
\ul{\Mc{R}}_p(N_1,N_2,N)&=\sup_{v\in T}\liminf_{k\to\infty}(\epkw{\pi^k(v)})^{1/k}\\
&\le \sup_{l\ge0}\liminf_{k\to\infty}(\epkw{\pi^k(\pi^l(w))})^{1/k}\\
&\le \ul{\Mc{R}}_p(N_1,N_2,N).
\end{align*}
  \end{enumerate}
\end{proof}

We will consider the case that the weight is uniformly bounded. In the following theorem, we evaluate $\ol{d}_S$ and $\ul{d}_S$ by a partition.
\begin{thm}\label{Smain2}
Assume $\mu_{xy}\asymp 1$ for any $(x,y)\in E$ and $K$ is minimal and connected. Let $d\in \Mc{D}_\infty(G),$ fitting to $(G,E)$ and satisfying the basic framework. Suppose that $d$ satisfies the following for some $\beta>\alpha\ge1;$ 
\begin{align}
\bullet\hspace{0.3cm}& d(x,y)\asymp 1\text{ for any }(x,y)\in E, \notag \\
\bullet\hspace{0.3cm}& h_{2n}(x,x)\asymp\frac{c}{V_d(x,n^{1/\beta})}\text{ for any }n, \tag{DHK($\beta$)}\\
\bullet\hspace{0.3cm}& \text{There exists }\lambda,C>0\text{ such that}\notag \\ &R(B_d(x,\lambda r),B_d(x,r)^c)V(x,r)\ge Cr^\beta \text{ for any }r>r_x, \tag{ARL($\beta$)}\\
\bullet\hspace{0.3cm}& \text{There exists }C'>0\text{ such that }\notag \\ &R(x,B_d(x,r)^c)V(x,r)\le C'r^\beta\text{ for any }r>r_x, \tag{BRU($\beta$)} \\
\bullet\hspace{0.3cm}& \text{There exists }C''>0\text{ such that }\notag \\ &V_d(x,r)\le C''(r^\alpha/s^\alpha) V_d(x,s)\text{ for any }x\in G\text{ and }r>s>0. \tag{VG($\alpha$)}
\end{align}
Then for any $N,N_1\ge0$ and sufficiently large $N_2=N_2(N_1),$ 
\begin{align*}
\ul{d}_S(G,\mu)&=2\frac{\log \ul{\Mc{N}}}{\log \ul{\Mc{N}}-\log\ul{\Mc{R}_2}(N_1,N_2,N)},\\ 
\ol{d}_S(G,\mu)&=2\frac{\log \ol{\Mc{N}}}{\log \ul{\Mc{N}}-\log\ul{\Mc{R}_2}(N_1,N_2,N)}.
\end{align*}
\end{thm}
The assumption of $d$ seems to be too strong, but we can justify the above assumption in the following way.

\begin{defi}[Volume doubling condition]
Let $(X,d)$ be a metric space and $\mu$ be a measure on $X.$ We say $\mu$ satisfies volume doubling condition with respect to $d$, we will write $\mu$ satisfies ($VD$)$_d$ in short, if there exists $C>0$ such that 
\[V_{d,\mu}(x,2r)\le CV_{d,\mu}(x,r) \hspace{10pt} \text{ for any $x\in X$ and $r>0.$ }\]
\end{defi}

\begin{thm}\label{th-d}
  Assume $(G,\mu)$ satisfies condition {\rm($p_0$)} and {\rm(VD)$_R$} holds. If $R\in\Mc{D}_\infty(G)$ and $V_R(x,r)<\infty$ for any $r>0$ and $x\in G,$ then there exists a fitting metric $d$ such that $d\qs R,$ $d(x,y)\asymp 1$ for any $(x,y)\in E$ and satisfy {\rm(DHK($\beta$)),(ARL($\beta$)),(BRU($\beta$))} and {\rm(VG($\alpha$))} for some $1\le\alpha<\beta.$  
\end{thm}

This theorem is a discrete version of the corresponding result in \cite{Kig12}, and also based on \cite{BCK}. See Section 6 for the proof. Combining these theorems, we get the following corollary.
\begin{cor}\label{Smain3}
Assume $\mu$ satisfies {\rm(VD)}$_R,$ $(G,\mu)$ satisfies $\mu_{xy}\asymp 1$ for any $(x,y)\in E,$ $V_R(x,r)<\infty$ for any $r>0,x\in G$ and $\diam(X,R)=\infty.$ If the metric $d,$ taken in Theorem \ref{th-d} satisfies the basic framework (with respect to some minimal connected partition $K$) and 
\begin{equation}\label{thethe}
 \frac{\log \ul{R}_2(N_1,N_2,N)}{\log \ol{N}_*} \le \frac{\log \ul{\Mc{R}}_2(N_1,N_2,N)}{\log \ul{\Mc{N}}} 
\end{equation}
for some $N,N_1\ge0$ and sufficiently large $N_2>N_1.$ Then
\[ \ard(G,R)\le \ul{d}_S(G,\mu) \le \ol{d}_S(G,\mu)<2.\] 
\end{cor}
The condition \eqref{thethe} holds in natural settings, including Sierpi\'{n}ski carpets or $n$-gaskets. We give an example in Example \ref{exthe2} such that assumptions of Corollary \ref{Smain3} except \eqref{thethe} hold, and neither $\ul{d}^S_2$ nor $\ol{d}^S_2$ coincides with the spectral dimension $d_S(G,\mu).$  It helps to understand the difference between $\ul{\Mc{R}}$ and $\ul{R}.$ 
\begin{proof}[Proof of Corollary \ref{Smain3}]
 Since $d$ satisfies (VG($\alpha$)) and (DHK($\beta$)), it follows that 
\begin{align*}
\ul{d}_S(G,\mu)&\le 2\limsup_{n\to\infty}\frac{\log V_d(x,n^{1/\beta})}{\log n}\\
& \le 2\limsup_{n\to\infty} \frac{\log V_d(x,1)+\log n^{\alpha/\beta}}{\log n}=2\frac{\alpha}{\beta}<2.
\end{align*}
On the other hand, by definition and Theorem \ref{Smain2}, we have
\begin{align*}
 \ul{d}^S_2(N_1,N_2,N)&=\left(1-\frac{\log \ul{R}_2(N_1,N_2,N)}{\log \ol{N}_*}\right)^{-1}, \\
 \ul{d}_s(G,\mu)&=\left(1-\frac{\log \ul{\Mc{R}}_2(N_1,N_2,n)}{\log \ul{\Mc{N}}} \right)^{-1},
\end{align*}
(note that $\diam(G,d)=\infty$ because $R\qs d$) and hence 
 \[\ul{d}^S_2(N_1,N_2,N) \le\ul{d}_S(G,\mu) \le \ol{d}_S(G,\mu)<2. \]
Since $\ul{d}^S_2(N_1,N_2,N)<2,$ by Theorem \ref{Smain1} we have $\ard(G,d)\le\ul{d}^S_2(N_1,N_2,N).$ Moreover, $\ard(G,R)=\ard(G,d)$ because $R\qs d,$ so this shows
\[\ard(G,R) \le\ul{d}_S(G,\mu) \le \ol{d}_S(G,\mu)<2. \]
\end{proof}
For the rest of this section, we prove Theorem \ref{Smain2}.
First we give general lemmas and definition.
\begin{lem}
Let $g$ be a thick discrete weight function, then for any $M\ge1$, there exists $\eta=\eta_M>0$ such that for any $w\in T_e$, there exists $x\in K_w$ such that  $U^d_M(x, \eta g(\pi(w)))\subseteq K_w.$ 
\end{lem}
\begin{proof}
 If $M=1$, the statement follows from the definition of $U_1^d(x,s)$. We prove the rest using induction.\\
 Assume the statement for $M\ge 1$ holds. Let $w\in T_e$ and $x\in K_w$ such that $U^g_M(x,\eta g(\pi(w)))\subseteq K_w.$
 \begin{itemize}
 \item If $\Lambda^g_{\eta g(\pi(w)),0}(x)=\emptyset$, then $U^g_{M+1}(x,\eta g(\pi(w)))=\{x\}\in K_w.$
 \item Otherwise, we can take $v\in\Lambda^g_{\eta g(\pi(w)),0}(x)$ and $y\in K_v$ such that 
\[U^g_M(y,\eta^2g(\pi(w)))\subseteq U^g_M(y,\eta g(\pi(v))) \subseteq K_v,\]
and hence
\[U^g_{M+1}(y,\eta^2g(\pi(w)))\subseteq \{K_u\mid u\in\Lambda^g_{\eta g(\pi(w)),1}(v) \}\subseteq K_w. \]
 \end{itemize}
\end{proof}
\begin{lem}\label{lemada}
Let $d\in \Mc{D}_\infty(G).$ If $d$ is adapted for $M=M_0$, then $d$ is  adapted for any $M\ge M_0.$ 
\end{lem}
\begin{proof}
  Since $d$ is $M_0$ adapted, there exists $\alpha_1,$
\[B_d(x,\alpha_1r)\subseteq U^d_{M_0}(x,r)\subseteq U^d_{M}(x,r).\]
On the other hand, if $y\in U^d_{M}(x,r)\setminus\{x\}$, then there exists $\{w_i\}_{i=0}^n\subseteq \Lambda^d_r$ with $n\le M$ such that $x\in K_{w_0}, y\in K_{w_n}$ and $K_{w_{i-1}}\cap K_{w_i}\ne\emptyset$ for any $i\in [1,n]_{\Mb{Z}}.$ Hence $d(x,y)\le \sum_{i=0}^nd(w_i)\le (M+1)r$, therefore $U^d_{M}(x,r)\subseteq B_d(x,(M+2)r).$  
\end{proof}
\begin{defi}
 For $w\in T_e$ and $M\ge 0,$ define $U_M(w)$ by 
\[U_M(w)=\cup\{K_v\mid v\in \Gamma_M(w)\cap T_e \}.\]
\end{defi}
For the rest of this section, we assume $d$ satisfies the basic framework, $K$ is connected, $\mu_{xy}\asymp 1$ and $d(x,y)\asymp 1$ for any $(x,y)\in E.$  Let $\eta_0>0$ be such that $\eta_0^{-1}r^{[w]}\le d(w) \le \eta_0r^{[w]}$ and write $N_*=\sup_{w\in T}\#(S(w))$ (note that $N_*\le \sup_{w\in T_e\setminus \Lambda_e}\#(S(w))\vee 2r^{-1}<\infty).$\\
Moreover, since $d(x,y)\asymp 1$ for any $(x,y)\in E$ and $d(w)\asymp r^{[w]}$ for $w\in \Lambda_e,$  there exist $m_0,m_1 \in \Mb{Z}$ such that $m_0\le [w] \le m_1$ for any $w\in \Lambda_e.$  
\begin{lem}
 $\sup_{w\in T}\#(\Gamma_1(w))<\infty.$ 
\end{lem}
\begin{proof}
  If $w\in T_e,$ then for $k>2\log \eta_0/\log2,$
\begin{align*}
(\Gamma_1(w))& 
\le\#(\{v\in (T)_{[w]}\mid\text{there exists }v',w'\in\Lambda^d_{\eta_0 r^{[w]}}
\text{ such that }\\
&\hspace{120pt} v\in T_{v'}, w\in T_{w'}\text{ and }K_{v'}\cap K_{w'}\ne\emptyset \})
\\
& \le N_*^k\sup_{w'}\#(\Lambda^d_{\eta_0 r^{[w]},1}(w')).
\end{align*}
Otherwise, $\#(\Gamma_1(w))\le \sup_{x\in G}\#(\{y\mid y\sim x\})+1.$ Since $d$ is uniformly finite and $(G,E)$ is bounded degree, these values are bounded.
\end{proof}
We write $L_*=\sup_{w\in T}\#(\Gamma_1(w)).$
\begin{lem}\label{Smain2lem1}
Let $N_1\ge 0$ and $\lambda\in(0,1),$ then there exists $N_2$ and $\xi>0$ such that for any $x\in G$ and $w\in T_e$ such that $x\in K_w,$
\begin{equation}
U_{N_1}(w)\subset B_d(x,\lambda\xi r^{[w]}),\hspace{10pt} B_d(x,\xi r^{[w]})\subset U_{N_2}(w).
\end{equation}
\end{lem}
\begin{proof}
Since $d$ is $M_*$ adapted, there exists $\eta_1$ such that
\[U_{N_1}(w)\subset U^d_{N_1\vee M_*}(x,\eta_0r^{[w]})\subset B_d(x,\eta_1 r^{[w]}),\]
and 
\[B_d(x,\lambda^{-1}\eta_1r)\subset U^d_{N_1\vee M_*}(x,\lambda^{-1}\eta_1^2r),\]
for any $x\in G$ and $w\in T$ such that $x\in K_w.$ Let $m=\lceil -\log(\lambda^{-1}\eta_0\eta_1^2)/\log r \rceil,$ then
\[U^d_{N_1\vee M_*}(x,\lambda^{-1}\eta'^2r^{[w]})\subset U_{(N_1\vee M_*)+1}(\pi^m(w))\subset U_{N^m_*((N_1\vee M_*)+1)}(w),\]
because $K$ is connected. Adjusting constants, we get desired result.
\end{proof}
\begin{lem}\label{Smain2lem2}
Assume {\rm(VG($\alpha$))} holds. Then $\#(S^{m_1-[w]}(w))\asymp V_d(x,r^{[w]})$ for any $w\in T_e$ and $x\in K_w.$
\end{lem}
\begin{proof}
 Since $\mu\asymp 1$ , $K$ is minimal and $(G,E)$ has a bounded degree,
\[\mu(K_w)\asymp \sum_{v\in T_w\cap\Lambda_e}\mu(K_v)\asymp\#(\{v\mid v\in T_w\cap\Lambda_e \}).\]
Moreover, since $m_0\le [v] \le m_1 $ for any $v\in \lambda_e,$
\begin{align*}
\#(S^{m_1-[w]}(w))&\ge\#(\{v\mid v\in T_w\cap\Lambda_e \})\\
&\ge \#(S^{(m_0-[w])\vee 0})\ge N_*^{-2}\#(S^{m_1-[w]}(w)).
\end{align*}
On the other hand, let $w\in T_e$ and $x\in K_w.$ Then there exist $\eta_1>0$ and $M_*\in\Mb{N}$ such that
\[ K_w \subseteq U^d_{M_*}(x,\eta r^{[w]}) \subseteq B_d(x,\eta_1r^{[w]}),\]
because $d$ is adapted. This together with (VG($\alpha$)) shows there exists $C_1>0$ such that $\mu(K_w)\le C_1V_d(x,r^{[w]}).$ \\Moreover, since $d$ is thick, there exists $\eta_2,\eta_3>0$ and $x'\in K_w$ such that 
\[K_w \supseteq U^d_{M_*}(x',\eta_2\eta^{-1}r^{[w]-1})\supseteq B_d(x',\eta_3r^{[w]}).\]
Note that $d(x,x')\le d(w)\le \eta r^{[w]},$ this together with (VG($\alpha$)) shows there exists $C_2>0$ such that
\[V_d(x,r^{[w]})\le V_d(x',(1+\eta)r^{[w]}) \le C_2V_d(x',\eta_3r^{[w]})\le C_2\mu(K_w).\]
Therefore 
\[\#(S^{m_1-[w]}(w))\asymp \mu(K_w)\asymp V_d(x,r^{[w]}) \]
for any $w\in T_e$ and $x\in K_w.$ 
\end{proof}
\begin{lem}\label{Smain2lem3}
Fix any $w\in T$ such that $[w]\le m_0.$ Then 
\[R(U_{N_1}(\pi^k(w)),U_{N_2}(\pi^k(w))^c)^{-1}\asymp \etkw{\pi^k(w)} \] 
for any $k\ge0.$
\end{lem}
\begin{proof}
 Let $u,v\in (T)_{[w]}.$ If $x\in K_u,y\in K_v$ and $x\sim y,$ then there exists $\omega \in \Sigma^*$ such that $K_{\omega_e}=\{x,y\}. $ Since $[w]\le m_0,$ we have $\omega_{[w]}\in T_e.$ Then $K_{\omega_{[w]}}\cap K_u\ne\emptyset $ and $K_{\omega_{[w]}}\cap K_v\ne\emptyset,$ so we have $l_{J^h_{[w]}}(u,v)\le 1.$
Now, let $f$ be a function on $(T)_{[w]}$ such that $f\equiv 1$ on $S^k(\Gamma_{N_1}(\pi^k(w))$ and $f\equiv 0$ on $S^k(\Gamma_{N_2}(\pi^k(w)))^c.$ Define $\ol{f}$ on $G$ by $\ol{f}(x)=\max_{w:x\in K_w}f(w)$, then $\ol{f}\equiv 1$ on $U_{N_1}(\pi^k(w))$ and $\ol{f}\equiv0$ on $U_{N_2}(\pi^k(w))^c.$ This is  because
\begin{itemize}
\item if $x\in U_{N_1}(\pi^k(w)),$ then there exists $v\in S^k(\Gamma_{N_1}(\pi^k(w)))$ such that $x\in K_v$ by (PG1),
\item if $x\not\in U_{N_2}(\pi^k(w)),$ then for any $v\in S^k(\Gamma_{N_2}(\pi^k(w))),$ it follows that $x\not\in K_v$ by (PG1).
\end{itemize}
Hence, noting that $\mu_{xy}\asymp 1,$ there exists $C>0$ such that the following holds: 
\begin{align*}
  \frac{1}{2}\sum_{x\sim y}|\ol{f}(x)-\ol{f}(y)|^2\mu_{xy}&\le \frac{1}{2}C\sum_{x\sim y}\sum_{u:x\in K_u}\sum_{v:y\in K_v}|f(u)-f(v)|^2\\
&\le \frac{1}{2}C\sum_{(u,v)\in J^h_{[w]}}\sum_{x\in K_u}\sum_{y\in K_v}|f(u)-f(v)|^2\\
&\le  \frac{1}{2}CN_*^{2(m_1-[w])}\sum_{(u,v)\in J^h_{[w]}}|f(u)-f(v)|^2.
\end{align*}
This implies $R(U_{N_1}(\pi^k(w)),U_{N_2}(\pi^k(w))^c)^{-1}\le C N_*^{2(m_1-[w])} \etkw{\pi^k(w)}.$\\
On the other hand, let $h$ be a function on $G$ such that $h\equiv 1$ on $U_{N_1}(\pi^k(w))$ and $h\equiv 0 $ on $U_{N_2}(\pi^k(w))^c.$ Define $\ul{h}(w)=\min_{x\in K_w}h(w),$ then similarly we get $\ul{h}\equiv 1$ on $S^k(\Gamma_{N_1}(w))$ and $\ul{h}\equiv 0$ on $S^k(\Gamma_{N_2}(w))^c.$ \\
Let $x,y\in G.$ If $x\in K_u$ and $y\in K_v$ for some $(u,v)\in J^h_{[w],N}$, then 
\[l_B(x,y)\le (N+1)\sup_{\nu\in (T)_{[w]}}\#(K_\nu)\le 2(N+1)N_*^{m_1-[w]},\] because $K$ is connected. Therefore 
  \begin{align*}
  \frac{1}{2}\sum_{(u,v)\in J^h_{[w],n}}|\ul{h}(u)-\ul{h}(v)|^2&\le \frac{1}{2}\sum_{(u,v)\in J^h_{[w],n}}\sum_{x\in K_u}\sum_{y\in K_v}|h(x)-h(y)|^2\\
&\le \frac{1}{2}\sum_{x,y:l_E(x,y)\le l_0}\sum_{u:x\in K_u}\sum_{v:y\in K_v}|h(x)-h(y)|^2\\
&\le  \frac{1}{2}N_*^{2(m_1-[w])}\sum_{x,y:l_E(x,y)\le l_0}|h(x)-h(y)|^2,
\end{align*}
where $l_0=2(N+1)N_*^{m_1-[w]}.$ Note that if $l_E(x,y)\le l_0$, then $|h(x)-h(y)|^2\le l_0\sum_{i=1}^{n}|h(x_{i-1})-h(x_i)|^2$ for some $n$-path $\{x_i\}_{i=0}^{n}$ between $x$ and $y$ with $n\le l_0,$ so for $E_{x,y,l_0}=\{(p,q)\mid p\sim q, p\in B_{l_E}(x,l_0),\text{ and }q\in B_{l_E}(y,l_0)\}\subseteq E,$
\begin{align*}
 \sum_{x,y:l_E(x,y)\le l_0}|h(x)-h(y)|^2&\le l_0\sum_{x,y:l_E(x,y)\le l_0}\sum_{(p,q)\in E_{x,y,l_0}}|h(p)-h(q)|^2\\
& \le l_0\sum_{p\sim q}\sum_{x\in B_{l_E}(x,l_0)}\sum_{y\in B_{l_E}(y,l_0)}|h(p)-h(q)|^2 \\
& \le l_0(\sup_{x\in G}\#(\{y\mid y\sim x\}))^{2l_0}\sum_{p\sim q}|h(p)-h(q)|^2.
\end{align*}
These inequalities with $\mu_{xy}\asymp 1$ shows
\[C'R(U_{N_1}(\pi^k(w)),U_{N_2}(\pi^k(w))^c)^{-1}\ge\etkw{\pi^k(w)}\text{ for some }C'>0.\]
\end{proof}
\begin{proof}[Proof of Theorem \ref{Smain2}]
 Fix $N_1,N\ge 0$ and  let $w\in T$ such that $[w]\le m_0.$ Then using Lemma \ref{Smain2lem1} and the fact that $d$ is adapted, there exists $\xi$ and $\zeta$ such that for sufficiently large $N_2,$ for any $x\in K_w,$ we have
\[ x\in U_{N_1}(w) \subseteq B_d(x,\lambda\xi r^{[w]})\]
and 
\[B_d(x,\xi r^{[w]})\subseteq U_{N_2}(w)\subseteq U^d_{N_2}(\eta r^{[w]})\subseteq B_d(x,\zeta r^{[w]}).\]   
Hence
\[R(x,B_d(x,\zeta r^{[w]})^c)\ge R(U_{N_1}(w),U_{N_2}(w)^c) \ge R( B_d(x,\lambda\xi r^{[w]}),B_d(x,\xi r^{[w]})^c).\]
These with Lemmas \ref{Smain2lem2} and \ref{Smain2lem3} show that for any $k\ge 0,$ there exist $C_1,C_2$ such that
\begin{align*}
&C_1 R(x,B_d(x,\zeta r^{[w]-k})^c)V_d(x,\zeta r^{[w]-k})\\
\ge& (\etkw{\pi^k(w)})^{-1}\#(S^{m_1-[w]+k}(\pi^k(w)))\\
\ge& C_2 R( B_d(x,\lambda\xi r^{[w]}),B_d(x,\xi r^{[w]})^c)V_d(x,\xi r^{[w]-k}).
\end{align*}
This together with (ARL($\beta$)),(BRU($\beta$)) implies that there exist $\delta>0$ such that
\begin{align*}
-k\beta\log r-\delta &\le \log\#(S^{m_1-[w]+k}(\pi^k(w)))-\log \etkw{\pi^k(w)}\\
& \le -k\beta\log r+\delta 
\end{align*}
and hence by Lemma \ref{tilt},
\[ \log\ul{\Mc{N}}+\beta\log r=\liminf_{k\to\infty} \log\etkw{\pi^k(w)}^{1/k}\]
because $N_*^{m_1-[w]}\#(S^{k}(\pi^k(w)))\ge\#(S^{m_1-[w]+k}(\pi^k(w)))\ge\#(S^{k}(\pi^k(\pi^l(w)))).$ This equation also holds for $\pi^l(w)$ with $l\ge 0$, so again using Lemma \ref{tilt}, we obtain
\[ \log\ul{\Mc{N}}+\beta\log r=\sup_{l\ge0}\liminf_{k\to\infty} \log\etkw{\pi^k(\pi^l(w))}^{1/k}\hspace{-1.2pt}=\log\ul{\Mc{R}_2}(N_1,N_2,N).\]
Now, by (DHK) and (VG)$_\alpha$,

\[\frac{\ul{d}_S}{2}=\liminf_{n\to\infty} \frac{\log V_d(x,n^{1/\beta})}{\log n}=\liminf_{r\to\infty} \frac{\log V_d(x,r^{1/\beta})}{\log r}=\liminf_{k\to\infty} \frac{\log V_d(x,r^{-k})}{\log r^{-\beta k}}\]
and by Lemma \ref{Smain2lem2},
\[\ul{d}_S=2\frac{\liminf_{k\to\infty}\frac{1}{k}\log V_d(x,r^{-k})}{-\beta \log r}=2\frac{\log \ul{\Mc{N}}}{\log \ul{\Mc{N}}-\log\ul{\Mc{R}_2}(N_1,N_2,N)}.\]
In the same way, we also get $\ol{d}_S=2\frac{\log \ol{\Mc{N}}}{\log \ul{\Mc{N}}-\log\ul{\Mc{R}_2}(N_1,N_2,N)}.$
\end{proof}

\section{Examples}
We first give an example that the ARC dimension can be calculated using Theorem \ref{Smain1}.
\begin{ex}\label{exthe1}
  Let $f(n):\Mb{Z}_+\to\Mb{Z}_+$ such that $f(n)\le n$ for any $n.$ For $n\ge0,$ define $B_n,L_n,X_n\in \Mb{R}^2$ by
  \begin{align*}
    B_n&= [2^n,2^{n+1}]\times[0,2^n], \\
 L_n&= \Cup_{j\in\Mb{Z}}\left(\{(x,y)\mid x=2^{n-f(n)}j\} \cup \{(x,y)\mid y=2^{n-f(n)}j\}\right),\\
X_n&= B_n\cap L_n.
  \end{align*}
We also define $X,G,E$ by
\begin{align*}
  X&=\{(t,0)\mid 0\le t\le 1\}\Cup\left(\Cup_{n\ge0}X_n\right), \\
 G&=X\cap\Mb{Z}^2,\\
 E&=\{(p,q)\in G\times G \mid d_2(p,q)=1\},
\end{align*}
where $d_2$ is the Euclidean metric in $\Mb{R}^2.$ See Figure \ref{figex1} or Figure \ref{figex2}.
\begin{figure}[tb]
\centering
\begin{minipage}{0.48\columnwidth}
 \centering
 \includegraphics[width=\columnwidth]{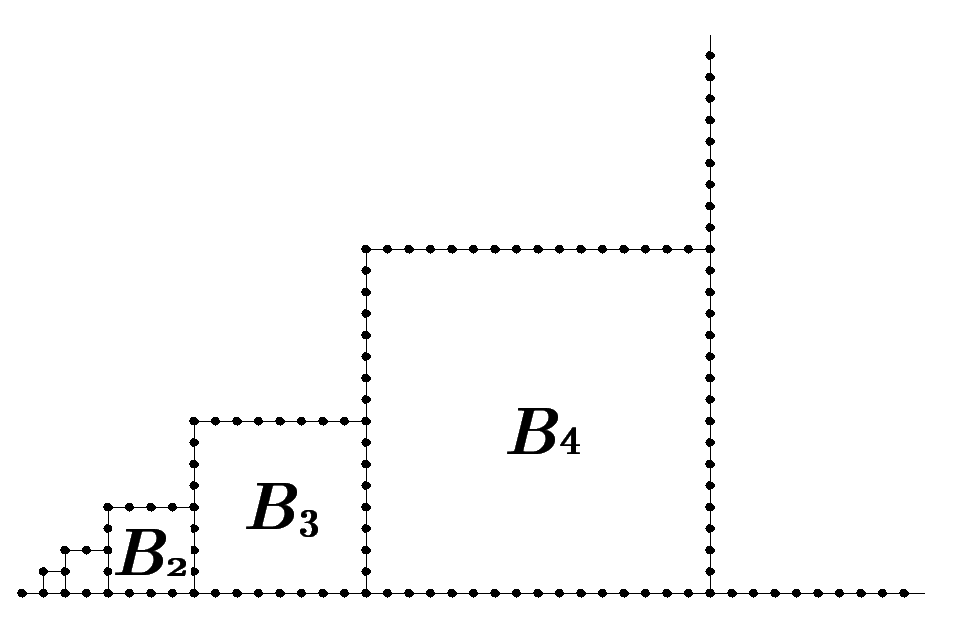}
 \caption{$f(n)\equiv 0$}
 \label{figex1}
\end{minipage}
\begin{minipage}{0.48\columnwidth}
 \centering
 \includegraphics[width=\columnwidth]{ex01.png}
 \caption{$f(n)=\lfloor \frac{n}{2} \rfloor$}
 \label{figex2}
\end{minipage}
\end{figure}
Next we introduce a partition of $(G,E).$ For $m,a,b\in\Mb{Z},$ define 
\begin{align*} \Mc{S}_{m,a,b}&=\{(x,y)\mid 2^ma\le x+y\le 2^m(a+1),\ 2^mb\le x-y\le 2^m(b+1) \}, \\
T_{-m}&=\{ \Mc{S}_{m,a,b}\mid \mathrm{int}(\Mc{S}_{m,a,b})\cap X\ne\emptyset\},\\
T&=\cup_{m\in\Mb{Z}}T_m,
\end{align*}
 and for $w\in T_m,$ define $\pi(w)$ as the unique element in $T_{m-1}$ such that $w\subseteq \pi(w)$ as subsets of $\Mb{R}^2.$ Then $(T,\pi)$ is a bi-infinite tree, and $T_m=(T)_m$ by taking $\phi=\Mc{S}_{0,0,0}.$ We also define $K:T\to\{\text{ finite subsets of $G$ }\}$ by
\[K_w=
\begin{cases}
  w\cap G\ (\text{as subsets of $\Mb{R}^2$}), &\text{ if }[w]\le 0,\\
 \pi^{[w]}(w)\cap G (\text{as subsets of $\Mb{R}^2$}), &\text{ if }[w]>0.
\end{cases}\]
Then $K_w$ is a partition of $(G,E)$. Moreover, $\Lambda_e=(T)_0$ and $T=T_{1/2}.$\\
Now we let $d=l_E,$ and calculate $\ard(G,d).$ 
\end{ex}
\begin{prop}\label{dimgauge}\hspace{0pt}\samepage
  \begin{enumerate}
  \item If  $\limsup_{n\to\infty} f(n)=\infty,$ then $\ard(G,d)=2.$
  \item If  $\limsup_{n\to\infty} f(n)<\infty,$ then $\ard(G,d)=1.$
  \end{enumerate}
\end{prop}

\begin{proof}
(1) First we check $d,K$ satisfy assumptions of Theorem \ref{Smain1}.
\begin{itemize}
\item By definition, $\#(S(w))\le 4$ for any $w\in T$ and $K$ is minimal. It is easily seen that $d$ fits to $(G,E)$.
\item $d(w)=2^{-m}=(1/2)^m$ for any $m\le 0$ and $w\in (T)_m.$
\item (uniformly finite) Similar to Example \ref{exSC}, 
  \[\Lambda^d_s=\begin{cases}
    (T)_{-m} & \text{ if } 2^m\le s<2^{m+1}, \\
\emptyset & \text{ if } s<1.
  \end{cases}\]
Hence $\#(\Lambda^d_{s,1}(w))\le\#(\{v\in (T)_{[w]}\mid v\cap w\ne\emptyset \text{ as subsets of }\Mb{R}^2\})\le 9$ for any $s>0$ and $w\in \Lambda^d_s.$ This shows $d$ is uniformly finite. 
\item (thick) Let $w=\Mc{S}_{m,a,b}\in (T)_{-m}$ for some $m\ge0.$
  \begin{itemize}
  \item[$\circ$] If $m\ge1,$ then $\Lambda^d_{d(\pi(w))/8,1}(x_w)=\Lambda^d_{2^{m-2},1}(x_w)=S^2(w)$ for $x_w=(2^{m-1}(a+b+1),2^{m-1}(a-b)).$ 
  \item[$\circ$] If $m=0,$ then either $(\frac{a+b}{2},\frac{a-b}{2})$ or $(\frac{a+b+1}{2},\frac{a-b+1}{2})$ belongs to $K_w.$ Let $x_w$ be such a point, then $\Lambda^d_{d(\pi(w))/4,1}(x_w)=\emptyset.$
 \end{itemize}
Hence $d$ is thick.
\item ($1$-adapted) Similar to Lemma \ref{lemada}, $U^d_1(x,r)\subseteq B_d(x,3r).$ On the other hand, if $r\ge1,$ then $U^d_1(x,r)=U^d_1(x,2^n)\supseteq B_d(x,2^n)\supseteq B_d(x,r/2)$ for some $n,$ hence $d$ is $1$-adapted. (See Figure \ref{figpf2}).
\end{itemize}
Therefore $d,K$ satisfy the assumptions of Theorem \ref{Smain1}. Now we adapt Theorem $\ref{Smain1}$ and show $\ard(G,d)=2.$\\
 The first step is to show $\ard(G,d)\ge 2.$ Since $\sup_n f(n)=\infty,$ for any $k\ge0,$ there exists $m\in \Mb{N}$ and $w=\Mc{S}_{m,a,b}\in (T)_{-m}$ such that 
\begin{align}
&\{S_{m-k,i,j}\mid i\in [2^k(a\hspace{-0.5pt}-\hspace{-0.5pt}2)\hspace{-0.5pt}-\hspace{-0.5pt}1,2^k(a\hspace{-0.5pt}+\hspace{-0.5pt}2)\hspace{-0.5pt}+\hspace{-0.5pt}1]_\Mb{Z},\ j\in[2^k(b\hspace{-0.5pt}-\hspace{-0.5pt}2)\hspace{-0.5pt}-\hspace{-0.5pt}1,2^k(b\hspace{-0.5pt}+\hspace{-0.5pt}2)\hspace{-0.5pt}+\hspace{-0.5pt}1]_\Mb{Z} \}\notag\\
&\subseteq (T)_{-(m-k)}.\label{aaaa}
\end{align}  

\begin{figure}[tb]
\centering
\begin{minipage}{0.48\linewidth}
 \centering
 \includegraphics[width=\columnwidth]{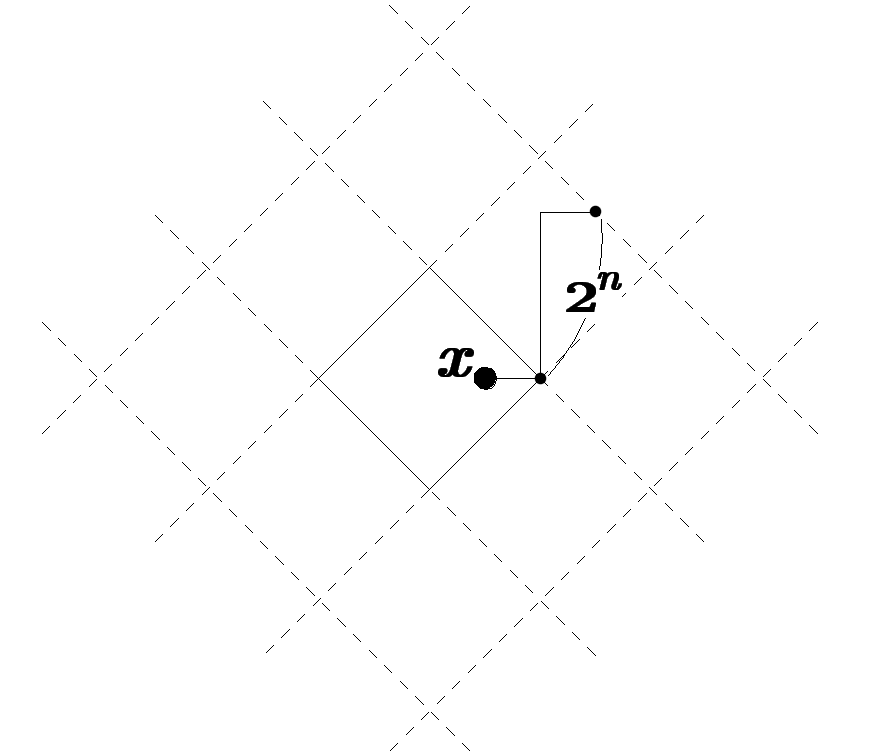}
 \caption{$B_d(x,2^n)\subseteq U^d_1(x,2^n)$}
 \label{figpf2} 
\end{minipage}
\begin{minipage}{0.48\linewidth}
 \centering
 \includegraphics[width=\columnwidth]{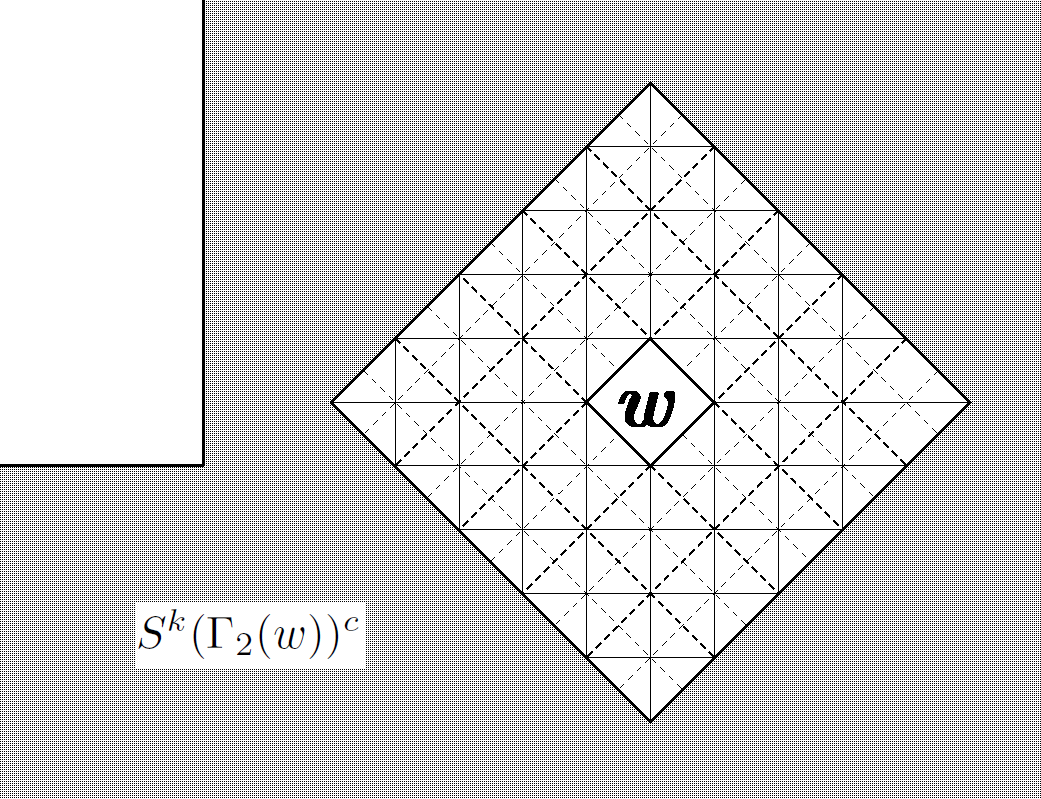}
 \caption{$w$ satisfying \eqref{aaaa}}
 \label{figpf1}
\end{minipage}
\end{figure}

Let $g$ be a function on $(T)_{-(m-k)}$ such that $g\equiv 1$ on $S^k(\Gamma_0(w))$ and $g\equiv 0$ on $S^k(\Gamma_2(w))^c.$ We also let $\tilde{g}=(g\vee0)\wedge1,$ then for any $p\ge 1,$ there exists $C_p>0$ such that 
\begin{align*}
 & \sum_{(u,v)\in E^h_{-(m-k)}}|g(u)-g(v)|^p
 \ge \sum_{(u,v)\in E^h_{-(m-k)}}|\tilde{g}(u)-\tilde{g}(v)|^p\\
 \ge&\sum_{i\in[2^ka,2^k(a+1)]_\Mb{Z}}\sum_{j\in[2^k(b-2),2^kb]_\Mb{Z}}|\tilde{g}(\Mc{S}_{-(m-k),i,j})-\tilde{g}(\Mc{S}_{-(m-k),i,j-1})|^p\\
\ge &\sum_{i}(2^{k+1}+1)^{1-p}\ge C2^{(2-p)k}.
\end{align*}
(This inequality follows from Jensen's inequality, together with $\tilde{g}(\Mc{S}_{-(m-k),i,2^kb})=1$ and $\tilde{g}(\Mc{S}_{-(m-k),i,2^k(b-2)-1})=0$ for any $i\in[2^ka,2^k(a+1)]_\Mb{Z}.$) Moreover for $\ p<1,$
\begin{align*}
 & \sum_{(u,v)\in E^h_{-(m-k)}}|g(u)-g(v)|^p
\ge \sum_{(u,v)\in E^h_{-(m-k)}}|\tilde{g}(u)-\tilde{g}(v)|^p\\
 \ge&\sum_{j\in[2^k(b-2),2^kb]_\Mb{Z}}|\tilde{g}(\Mc{S}_{-(m-k),2^ka,j})-\tilde{g}(\Mc{S}_{-(m-k),2^ka,j-1})|\ge 1. \\
\end{align*}
 Therefore $\lim_{k\to\infty}\Mc{E}_{p,k}(0,2,1)>0$ for any $p\le 2,$ hence $\ard(G,d)\ge2.$\\
On the other hand, define $g=g_w$ on $E^h_{-(m-k)}$ by 
\begin{align*}
 & g(\Mc{S}_{m-k,i,j})\\
=& \left( \frac{(2^k(a\hspace{-0.4pt}+\hspace{-0.4pt}2)\hspace{-0.4pt}-\hspace{-0.4pt}i)\wedge(i\hspace{-0.4pt}-\hspace{-0.4pt}2^k(a\hspace{-0.4pt}-\hspace{-0.4pt}2))\wedge(2^k(b\hspace{-0.4pt}+\hspace{-0.4pt}2)\hspace{-0.4pt}-\hspace{-0.4pt}j)\wedge(j\hspace{-0.4pt}-\hspace{-0.4pt}2^k(b\hspace{-0.4pt}-\hspace{-0.4pt}2))}{2^k}\vee 0\right) \wedge 1.
\end{align*}
Then $f\equiv 1$ on $S^k(\Gamma_0(w)),$ $f\equiv 0$ on $S^k(\Gamma_2(w))^c$ and 
\[ \sum_{(u,v)\in E^h_{-(m-k)}}|f(u)-f(v)|^p \le \sum_{v\in S^k(\Gamma_2(w))}8\cdot2^{-kp}\le C'2^{(2-p)k} \]
for some $C'>0,$ hence $\Mc{E}_{p,k,w}(0,2,1)\le C'2^{(2-p)k}. $ Moreover, for any $v\in T,$ this upper bound holds by the definition of $T$ and $K.$ Therefore $\lim_{k\to\infty}\Mc{E}_{p,k}(0,1,2)=0$ for any $p>2$ and hence $\ard(G,d)=2.$ \\
\vspace{-6pt}\\
(2) Let $\Mf{m}(A)=\#(A)$ for any $A\subset G,$ and $G_n=X_n\cap G$ for any $n\ge 0.$ Then
\[\Mf{m}(B_d(x,r)\cap G_n)\le 2(2^{f(n)}+1)(\diam(B_d(x,r)\cap G_n, d)+1),\] 
because $G_n$ consists of $2(2^{f(n)}+1)$ segments whose length are $2^n.$ Hence there exists $\bar{C}$ such that for any $x\in G$ and $r\ge 1,$
\[ r\le V_d(x,r) \le 1+\sum_{n\ge0}\Mf{m}(B_d(x,r)\cap G_n)\le \bar{C}r, \]
because $\sum_{n\ge0}\diam(B_d(x,r)\cap G_n, d)\le 2r$ and $\sup_n f(n)<\infty.$
Therefore $d$ is $1$-Ahlfors regular. On the other hand, $\ard(G,d)\ge1$ by Proposition \ref{sore} and hence $\ard(G,d)=1.$
\end{proof}
\begin{rem}
If we use a partition parallel to axes, that is, a partition $K'$ defined by $\Mc{S}'_{m,a,b}=[2^ma,2^m(a+1)]\times [2^mb,2^m(b+1)]$ in a similar way to $K,$ then $K'$ is not minimal. For example, both $S'_{0,0,0}$ and $S'_{0,1,0}$ include a edge $((1,0),(1,1))\in E.$ So we need some modification to apply Theorem \ref{Smain1} to $d,K'.$ 
\end{rem}
The next example is that $d_s(G,\mu)\ne \ul{d}_2^S=\ol{d}_2^S$ although $d$ satisfies (DHK($\beta$)), (ARL($\beta$)) and (BRU($\beta$)). 

\begin{ex}\label{exthe2}
Let $f:\Mb{N}\to\{0,1\},$  $G_0=\{0,1,\crt\}\in\Mb{C}$ and $E_0=\{(x,y)\in G_0\times G_0\mid x\ne y\}.$ For $n\in\Mb{N},$ we inductively define $|G_{n-1}|_\infty=\max_{z\in G_{n-1}}|z|,$ and
\begin{align*}
  F_{n,1}(z)&=z, & F_{n,2}(z)&=z+|G_{n-1}|_\infty,\\
 F_{n,3}(z)&=z+\left(\crt\right)|G_{n-1}|_\infty,& F_{n,4}(z)&=z+2|G_{n-1}|_\infty, \\
 F_{n,5}(z)&=z+\left(1+\crt\right)|G_{n-1}|_\infty, & F_{n,6}(z)&=z+2\left(\crt\right)|G_{n-1}|_\infty,
\end{align*}
\[F_n(z)=\begin{cases}
\cup_{i=1}^3F_{n,i}, &\text{ if }f(n)=0, \\
\cup_{i=1}^6F_{n,i}, &\text{ if }f(n)=1. \end{cases}\]
Now define $G_n=F_n(G_{n-1})$ and
\begin{multline*}
E_{n}=\{(x,y)\in G_n\times G_n\mid\text{there exist }x',y'\in G_{n-1} \text{ and }i\ge0\\ \text{ such that }(x',y')\in E_{n-1}\text{ and }x=F_{n,i}(x'),\ y=F_{n,i}(y')\}.
\end{multline*}

\begin{figure}[tb]
\centering
\begin{minipage}{0.4\linewidth}
 \centering
 \includegraphics[width=0.8\columnwidth]{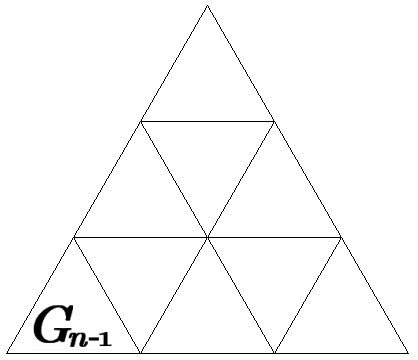}
 \caption{$G_n$ (if $f(n)=1$).}
 \label{figpf41} 
\end{minipage}
\begin{minipage}{0.4\linewidth}
 \centering
 \includegraphics[width=0.8\columnwidth]{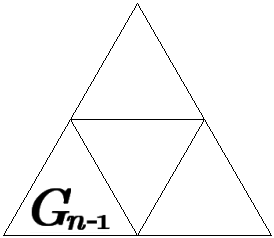}
 \caption{$G_n$ (if $f(n)=0$). }
 \label{figpf42}
\end{minipage}
\end{figure}

Note that $|G_n|_\infty=2^{n-m(n)}\cdot3^{m(n)}$ where $m(n)=\#(\{ k\mid k\le n,f(k)=1\}).$
Let $G=\cup_{n\ge 0}G_n$ and $E=\cup_{n\ge0}E_n.$ We also let $\mu\equiv 1$ on $E$ and consider the effective resistance $R$ of $(G,\mu).$

\begin{figure}[tb]
 \centering
 \includegraphics[width=0.5\columnwidth]{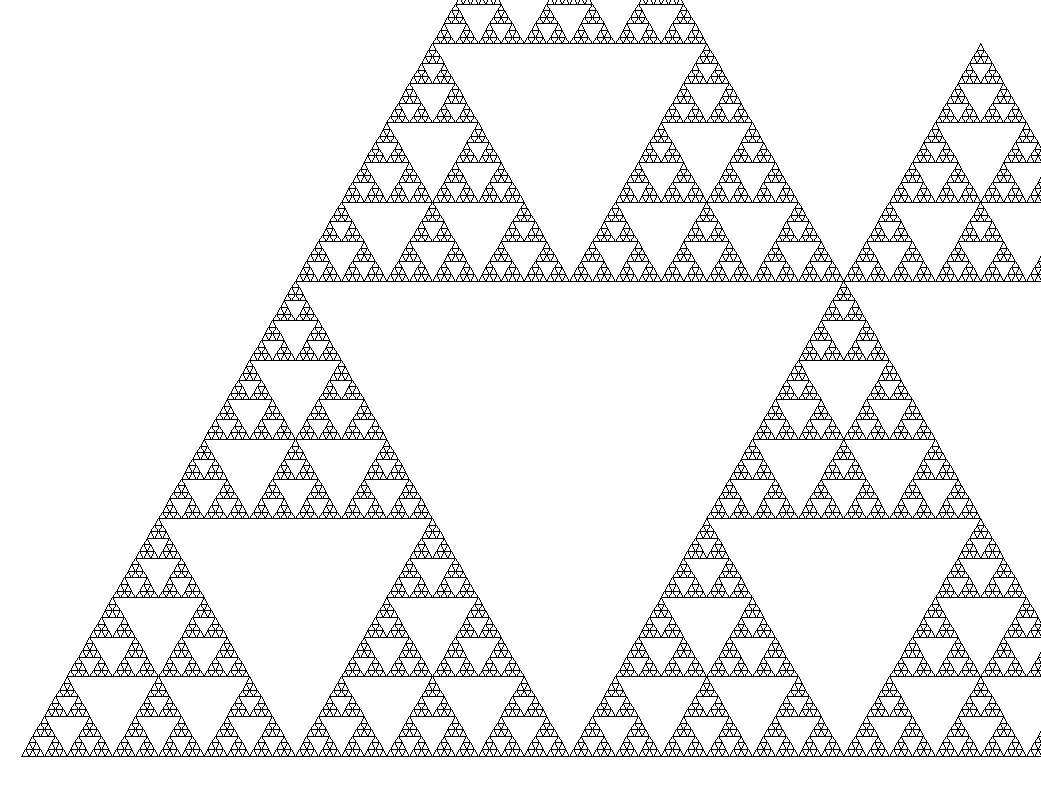}
 \caption{$(G,E)$ (for some $f$).}
\label{figpf5}
\end{figure}
 
Note that 
\[\begin{cases}
R(x,y)^{-1}\ge 1,& \text{ for any }(x,y)\in E,\\
R(x,y)^{-1}\le \Mc{E}(\bold{1}_{\{x\}})\le 6,&\text{ for any }x,y\in G\text{ with }x\ne y , 
\end{cases}\]
so $R$ fits to $(G,E).$ We will check properties of $R$ in order to apply Theorem \ref{th-d}. For the purpose, we first introduce a partition. For $n\ge0$ and $a,b\in\Mb{Z},$ define 
\begin{align*}
 \tr_{0,0,0}&=\{s+\left(\crt\right)t\mid s\ge0,\ t\ge 0,\ s+t\le 1\},\\
 \tr_{n,a,b}&=|G_n|_\infty\bigl(\tr_{0,0,0}+a+\bigl(\crt\bigr)b\bigr),\\
 T_{-n}&=\{\tr_{n,a,b}\mid \tr_{n,a,b}\subseteq \Cup_{m\ge n}F_m\circ F_{m-1}\circ\cdots\circ F_n(\tr_{n,0,0})\},\\
K_{n,a,b}&=\tr_{n,a,b}\cap G\text{ (as subsets of $\Mb{C}$). } 
\end{align*}
For any $n\ge 1,$ we let $T_n=\cup_{w\in T_0}\cup_{x,y\in K_w}\{x,y\}$ and $K_w=w$ for any $w\in T_n.$ Define $T=\sqcup_{n\in\Mb{Z}}T_n$ and $\pi(w)$ for $w\in T_n$ as the unique elements in $T_{n-1}$ such that $K_w\subseteq K_{\pi(w)}.$ Then $(T,\pi)$ is a bi-infinite tree, $(T)_n=T_n$ with $\phi=\tr_{0,0,0}$, $K$ is minimal connected partition and $\Lambda_e=(T)_1.$ If necessary, we replace $T,K$ by $T_r,K'$ for $r\in(0,1)$ in the way of Definition \ref{Tr}.\\
\end{ex}
\begin{lem}\label{lemSGtri}
Let $\Mf{R}(n)=\left(\frac{5}{3}\right)^{n-m(n)}\left(\frac{15}{7}\right)^{m(n)}$ for any $n\ge 0$ and let 
\begin{multline*}
n(x,y)=\min\{n\ge0 \mid\text{there exist }w,v\in(T)_{-n}\text{ such that }\\x\in K_w,y\in K_v\text{ and } K_w\cap K_v\ne\emptyset \}
\end{multline*}
for $(x,y)\in G.$ Then $R(x,y)\asymp \Mf{R}(n(x,y))$ for any $x,y\in G$ with $x\ne y.$
\end{lem}
\begin{proof}
We first evaluate $R(w).$ By the method of Laplacian on finite set (see \cite{Kig01}), $R(0,|G_n|_\infty)=\frac{2}{3}\Mf{R}(n)$ and 
\begin{multline}\label{Rtri}
\min\{\frac{1}{2}\sum_{(x,y)\in E_n}|f(x)-f(y)|^2\mid f:G_n\to \Mb{R},\ f(0)=1,\\
f(|G_n|_\infty)=f\bigl((\crt)|G_n|_\infty\bigr)=0\}=2\Mf{R}(n)^{-1}.
\end{multline}
Hence for any $\tr_{n,a,b}\in T,$
\[ \frac{2}{3}\Mf{R}(n)\ge R\bigl(\bigl(a+(\crt)b\bigr)|G_n|_\infty,\bigl(a+1+(\crt)b\bigr)|G_n|_\infty\bigr)\ge \frac{1}{6}\Mf{R}(n),\] 
and since $\Mf{R}(n-1)\le\frac{3}{5}\Mf{R}(n),$ we obtain $R(w)\asymp R(n)$ for any $n\ge-1$ and $w\in (T)_{-n}$ by using the sum of a geometric series. Fix any $x,y\in G,$ and let $w,v\in (T)_{-n(x,y)}$ such that $x\in K_w,\ y\in K_v$ and $K_w\cap K_v\ne\emptyset.$ Then by \eqref{Rtri}, there exists $C>0$ such that
\[R(w)+R(v)\ge R(x,y)\ge\frac{\Mf{R(n-1)}}{2\#(\Gamma_2(w'))}\ge C\Mf{R(n)},\]
where $w'\in S(w)$ such that  $x\in K_{w'}.$ Therefore $R(x,y)\asymp \Mf{R}(n(x,y)).$
\end{proof}
This lemma also implies $R$ is adapted (for $M=1$) and $V(x,R(x,y))\asymp\Mf{V}(n(x,y))$ where $\Mf{V}(n)=3^{n-m(n)}\cdot6^{m(n)}.$ This inequality also shows (VD)$_R,$ and $(G,\mu)$ satisfies the conditions of Theorem \ref{th-d}.\\
Next we modify $T$ in order to satisfy $d(w)\asymp r^{[w]}$ for some $r\in(0,1)$, where $d$ is the metric obtained by Theorem \ref{th-d}. For $j\ge0,$ let $n(j)\ge0$ such that $\Mf{R}(n(j))\Mf{V}(n(j))\le\bigl(\frac{90}{7}\bigr)^j<\Mf{R}(n(j)+1)\Mf{V}(n(j)+1)$ and for $j<0,$ set $n(j)=j.$ We consider $\bar{T}=\cup_{j\in\Mb{Z}}(T)_{-n(j)}$, and $\bar{\pi}(w)=\pi^{n(j+1)-n(j)}$ for $w\in (T)_{-n(j)}.$ Then $(\bar{T},\bar{\pi},\tr_{0,0,0})$ is a bi-infinite tree with a reference point, $(\bar{T})_j=(T)_{n(j)}$ and $K|_{\bar{T}}$ is minimal, connected partition, $\Lambda_e=(T)_1.$ Moreover, for any $w\in\cup_{j\ge0} (T)_{-n(j)},$
\[\sup_{x,y\in K_w}R(x,y)V(x,R(x,y))\asymp \Mf{R}(n(j))\Mf{V}(n(j))\asymp \left(\frac{90}{7}\right)^j\]
and hence $d(w)\asymp \bigl( \frac{7}{90} \bigr)^{[w]/\beta}$ for any $w\in T_e$ where $\beta$ is the constant in Theorem \ref{th-d}. Comparing with $R(x,y)V(x,R(x,y)),$ we can also see that $d$ is uniformly finite, thick and adapted because of Equation \eqref{Rtri}.
Now we let 
\[f(n)=\begin{cases}
  1,& \text{ if }l(l^2-1)<n\le l^3 \text{ for some }l\in\Mb{N}, \\
  0,& \text{ otherwise. }
\end{cases}\]
Then, we have the following.
\begin{prop}
$ d_S(G,\mu)=2\log 3/\log5$ and $\ul{d}^S_2(0,2,1)=\ol{d}^S_2(0,2,1)=2\log6/(\log90-\log7).$
\end{prop}
\begin{proof}
 Let $w=\tr_{n(j),0,0}$ for some $j\ge0.$ With the $\Delta$-$Y$ transform (see \cite[Lemma 2.1.15]{Kig01}), we can see that $\Mc{E}_{p,k,\pi^k(w)}(0,2,1)\asymp \Mf{R}(n(k+j))/\Mf{R}(n(j))$ (see Figure  \ref{figpf6}), so
 \begin{align*}
   \lim_{k\to\infty}\frac{1}{k}\log \Mc{E}_{p,k,\pi^k(w)}(0,2,1)
&=\lim_{k\to\infty}\frac{1}{k}\log\frac{\Mf{R}(n(k+j))}{\Mf{R}(n(j))}
=\lim_{k\to\infty}\frac{1}{k}\log\Mf{R}(n(k))\\
&=\lim_{k\to\infty}\frac{1}{k}\left(n(k)\log\frac{3}{5}+m(n(k))\log\frac{7}{15}\right).
 \end{align*}
 
\begin{figure}[tb]
 \centering
 \includegraphics[width=0.8\columnwidth]{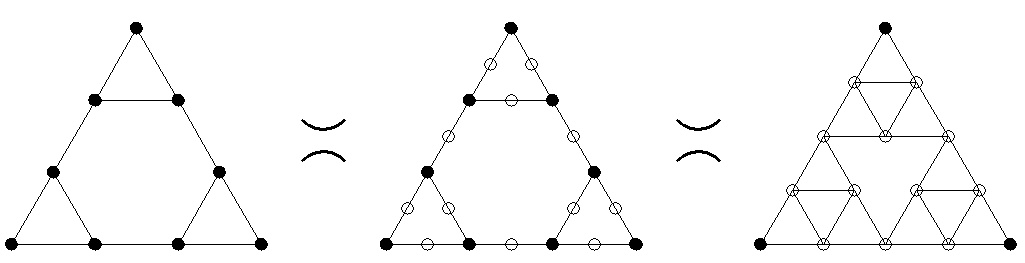}
 \caption{$\Mc{E}_{p,k,\pi^k(w)}(0,2,1)\asymp \Mf{R}(n(k+j))/\Mf{R}(n(j)).$}
\label{figpf6}
\end{figure}
 
Now we consider $\lim_{k\to\infty}n(k)/k.$ By definition, we obtain
\[ k\frac{\log90-\log7}{\log5}\ge n(k) \ge k\frac{\log90-\log7}{\log5}-m(n(k))-C\]
for some $C>0.$ Note that $\lim_{k\to\infty} m(k)/k=\lim_{k\to\infty}k^{-1/3}=0$ because $m(k^3)=\sum_{j=1}^kj=k(k-1)/2$, hence $\lim_{k\to\infty}k/n(k)=\log5/(\log90-\log7).$ Therefore by Lemma \ref{tilt},
\[\ul{\Mc{R}}_2(0,2,1)=\sup_{j\ge0}\lim_{k\to\infty}\frac{1}{k}\log \Mc{E}_{p,k,\pi^k(\tr_{n(j),0,0})}(0,2,1)=\frac{\log90-\log7}{\log5}\log\frac{3}{5}.\]
Similarly we get
\begin{align*}
\ul{\Mc{N}}=\ol{\Mc{N}}=\lim_{k\to\infty}\frac{1}{k}\log\#(S^k(\pi^k(\tr_{0,0,0})))&=\lim_{k\to\infty}\frac{1}{k}\left(n(k)\log3+m(n(k))\log6\right)\\
& =\frac{\log90-\log7}{\log5}\log3.
\end{align*}
Therefore by Theorem \ref{Smain2}, $d_S(G,\mu)=2\log3/(\log3-\log\frac{3}{5})=2\log3/\log5.$\\
On the other hand, since
 \[\sup\{k\mid \text{there exist }a\in\Mb{N}\text{ such that }f(b)=1\text{ for any }b\in[a,a+k]_\Mb{Z}\}=\infty,\] 
it follows that
\[\log\ol{N_*}=\lim_{k\to\infty}\frac{1}{k}\left(\log6^k\vee \log3^{(\log\frac{90}{7}/\log5)k}\right)=\log6\vee \frac{\log90-\log7}{\log5}\log3=\log6,\]
because $\log_{10}6>0.77>0.76>\frac{\log90-\log7}{\log5}\log_{10}3.$ Similarly,
\[\log\ol{R}_2(0,1,2)=\log\ul{R}_2(0,1,2)=\log\frac{7}{15}\vee \frac{\log90-\log7}{\log5}\log\frac{3}{5}=\log\frac{7}{15}.\]
Therefore
\[\ol{d}^S_2(0,2,1)=\ul{d}^S_2(0,2,1)=2\frac{\log6}{\log6-\log\frac{7}{15}}=2\frac{\log6}{\log90-\log7}.\]
\end{proof}
\begin{rem}
In the same way, we can prove that $ d_S(G,\mu)=\ul{d}^S_2(0,2,1)=\ol{d}^S_2(0,2,1)=2\log 3/\log5$ if $f\equiv 0$ (Sierpi\'nski gasket graph) and $ d_S(G,\mu)=\ul{d}^S_2(0,2,1)=\ol{d}^S_2(0,2,1)=2\log6/(\log90-\log7)$ if $f\equiv 1.$ Clearly the assumptions of Corollary \ref{Smain3} holds in these cases.
\end{rem}

\section{Proof of Theorem \ref{th-d}}
To show Theorem \ref{th-d}, we first prepare some condition.
\begin{defi}[uniformly shrinking]
Let $(X,d)$ be a metric space. we call $(X,d)$ is uniformly shrinking if there exists $\alpha\in(0,1)$ such that for any $x\in G$ and $r>0$, $B(x,r)\setminus B(x,\alpha r)\ne\emptyset$ whenever $B(x,r)\ne X$ and $B(x,r)\ne \{ x\}.$
\end{defi}
Uniformly shrinking condition is an extension of uniformly perfect condition to discrete metric spaces, but clearly it does not imply perfectness of a space.

\begin{lem}\label{lem-us}
Let $(X,d),(X,\rho)$ be metric spaces such that $d\qs\rho.$ Moreover, assume $(X,d)$ is uniformly shrinking, then
\begin{enumerate}
\item $(X,\rho)$ is uniformly shrinking.
\item If a measure $\mu$ satisfies \rm{(VD)}$_d$, then also satisfies {\rm(VD)}$_\rho.$
\end{enumerate}
\end{lem}
\begin{proof}
Since $(X,d)$ is uniformly shrinking, there exists $\alpha\in(0,1)$ such that $B(x,r)\setminus B(x,\alpha r)\ne\emptyset$ whenever $B(x,r)\ne X$ and $B(x,r)\ne \{ x\}.$ \\ 
(1) Fix $x\in X$ and $r>0$ such that $B_\rho(x,r)\ne X$ and $B_\rho(x,r)\ne\{x\}.$ Choose $\delta\in(0,1)$ such that $\theta(\delta)<1.$ Let $y_1\in X\setminus B_\rho(x,r).$ If $B_d(x,\alpha\delta d(x,y_1))\ne\{x\},$ then there exists $y_2$ such that $\alpha\delta d(x,y_1)\le  d(x,y_2)< \delta d(x,y_1)$ and then $\lambda_1\rho(x,y_1)\le \rho(x,y_2)<\lambda_2 \rho(x,y_1)$ where $\lambda_1=(1/\theta(\delta^{-1}\alpha^{-1})>0,\lambda_2=\theta(\delta)<1.$ We can inductively choose $y_n$ such that $\lambda_1\rho(x,y_n)\le \rho(x,y_{n+1}) <\lambda_2 \rho(x,y_n)$ whenever $B_d(x,\delta d(x,y_n))\ne\{x\}.$
\begin{itemize}
\item If there exists $n$ such that $\rho(x,y_{n+1})<r\le \rho(x,y_n)$, then $y_{n+1}\in B_\rho(x,r)\setminus B(x,\lambda_1 r).$
\item Assume $r\le\rho(x,y_n)$ and $B_d(x,\delta d(x,y_n))=\{x\}.$ Let $y\in B_\rho(x,r)\setminus\{x\},$ then $\rho(x,y)\ge (1/\theta(\delta^{-1})) \rho(x,y_n)$ because $d(x,y)\ge \delta d(x,y_n).$ Therefore $y\in B_\rho(x,r)\setminus B_\rho(x,r/\theta(\delta^{-1})).$ 
\end{itemize}
Therefore $\rho$ is uniformly shrinking.\\
(2) We first consider the case $B_\rho(x,r)\ne\{x\}.$ Then by (1), there exist $\alpha'\in(0,1)$ and $z\in B_\rho(x,r)$ such that $\rho(x,y)<(2/\alpha')\rho(x,z)$ for any $y\in B_\rho(x,2r)$ where $\alpha'$ is independent of $x$ and $r.$ Therefore 
\[B_\rho(x,2r)\subset B_d(x,2\ol{d}_\rho(x,2r))\subset B_d(x,\lambda_3\ol{d}_\rho(x,r))\]
for some $\lambda_3.$ Moreover, if $d(x,y)<\theta^{-1}(1)\ol{d}_\rho(x,r)$, then $\rho(x,y)<r.$ Since (VD)$_d$ holds, there exists $C>0$ and
\[V_\rho(x,2r)\le V_d(x,\lambda_3\ol{d}_\rho(x,r)) \le C V_d(x,\theta^{-1}(1)\ol{d}_\rho(x,r))\le C V_\rho(x,r).\]
Next we assume $B_\rho(x,r)=\{x\}.$ Choose $y\in B_\rho(x,2r)$ such that $d(x,y)\ge (1/2)\ol{d}_\rho(x,2r).$ Then for any $z\in X\setminus\{x\},d(x,z)\ge \theta^{-1}(1)d(x,y)$ because $B_\rho(x,r)=\{x\}.$ Since (VD)$_d$ holds, there exists $C>0$ such that
\[V_\rho(x,2r)\le V_d(x,2\ol{d}_\rho(x,2r)) \le C V_d(x,\theta^{-1}(1)d(x,y)) \]
and 
\[B_\rho(x,r)=\{x\}=B_d(x,\theta^{-1}(1)d(x,y)), \]
so we get desired condition.
\end{proof}

\begin{lem}\label{le-Rus}
Let $( G,\mu)$ be a weighted graph and $R$ be the associated effective resistance.
 \begin{enumerate}
\item $R(x,y)\ge\max\{\mu_x^{-1},\mu_y^{-1}\}$ for any $x,y\in  G$ such that $x\ne y.$
\item $R(x,y)\le\mu^{-1}_{xy}$ for any $x,y\in G$ such that $x\sim y.$
\item If $( G,\mu)$ satisfies the condition $(p_0)$, then $(X,R)$ is fitting and uniformly shrinking.
\end{enumerate}
\end{lem}
\begin{proof}
For any $x,y\in G$, $\mathcal{E}(\chi_{\{x\}},\chi_{\{x\}})=\mu(x)$ and so $R(x,y)\ge\max\{\mu^{-1}_x,\mu^{-1}_y\}.$ On the other hand, if $x\sim y,$ then $\mathcal{E}(f,f)\ge (1-0)^2\mu_{xy}$ for any $f$ such that $f(x)=1,f(y)=0$, so $R(x,y)\le\mu^{-1}_{xy}.$ (1) and (2) also show that $R$ fits to $(G,E)$ under the condition $(p_0).$ We prove uniformly shrinking condition in two cases. 
\begin{itemize}
\item Assume $B(x,r)\ne X$ and  $\displaystyle r>2\max_{y:y\sim x}\mu^{-1}_{xy}.$ Choose $\{x_n\}_{n=1}^k$ such that $x_1=x,x_k\not\in B(x,r)$ and $x_n\sim x_{n+1}$ for any $n.$ Then for some $n<k,$ $x_n\in B(x,r)$ and $x_{n+1}\not\in B(x,r).$ Note that $x\ne x_n$ because $x\not\sim x_{n+1}.$ Since $( G,\mu)$ satisfies the condition $(p_0),$ there exists $C>0$ such that $R(x_n,x_{n+1})\le \mu_{x_nx_{n+1}}^{-1}\le C\mu_{x_n}^{-1}\le CR(x,x_n),$ so $(1+C)R(x,x_n)\ge r$ and $B(x,r)\setminus B(x,(1+C)^{-1}r)\ne\emptyset.$
\item Assume $B(x,r)\ne\{x\}$ and $\displaystyle r\le 2\max_{y:y\sim x}\mu^{-1}_{xy}.$ By (1), $B(x,\mu_{x}^{-1})=\{x\}$ so $\displaystyle B(x,r)\setminus B(x,2^{-1}\mu^{-1}_x(\max_{y:y\sim x}\mu_{xy})r)\ne\emptyset.$ Since $(p_0)$ holds, $B(x,r)\setminus B(x,C'r)\ne\emptyset$ for some $C'\in (0,1/2].$
\end{itemize}
Take $(1+C)^{-1}\wedge C',$ we get the uniformly shrinking condition.
\end{proof}

We prove Theorem \ref{th-d} in two steps. The first step is to show existence of suitable distance $d$.\\
For the rest of this section, we assume that
\begin{itemize}
\item $( G,\mu)$ be a weighted graph.
\item $\mathcal{E}(f)$ is the associated energy and $R$ is the associated effective resistance to $\mu.$ 
\end{itemize}

\begin{prop}\label{pr-qsmet}
Let $( G,\mu)$ be a weighted graph and assume $R$ satisfies the condition of Theorem \ref{th-d}(1),
\begin{enumerate}
\item then there exists a distance $d$ on $ G$, which satisfies following conditions for some $\beta>\alpha\ge1.$
\begin{itemize}
\item $d\qs R$ 
\item $R(x,y)V_d(x,d(x,y))\asymp d^\beta(x,y)$ for any $x,y\in G.$ \hspace{30pt}{\rm(R($\beta$))}
\item $V_d(x,r)\le (Cr^\alpha/s^\alpha) V_d(x,s)$ for any $x\in G$ and $r>s>0$. \hspace{5pt}{\rm(VG($\alpha$))}
\end{itemize}
\item Let $d$ be a metric in (1). Then we can choose $d$ also satisfies following conditions.
\begin{itemize}
\item $\inf_{x,y\in G}d(x,y)=r_0>0.$\hspace{30pt}{\rm(dL)}
\item There exists $C_+>0$ such that $d(x,y)\le C_+$ for any $x,y\in G$ with $x\sim y.$ \hspace{30pt}{\rm(NdU)}
\item For any $x\in G$ and $r>0,$ $B_d(x,r)$ is a finite set. \hspace{30pt}{\rm(BF)}
\end{itemize}
\end{enumerate}
\end{prop}

To prove this proposition, we prepare a lemma.

 \begin{lem}\label{lemvdr}
Assume $(X,R)$ is uniformly shrinking and {\rm(VD)$_R$} holds. Then there exists a homeomorphism $\eta:[0,\infty)\to [0,\infty)$ such that for any $t>0$,
\[V_R(x,R(x,y))<\eta(t)V_R(x,R(x,z)) \text{ whenever } R(x,y)<tR(x,z) \]
\end{lem}
\begin{proof}
If $t\ge 1$, then there exists $C_1$ and $\tau_1$ such that $V_R(x,R(x,y))<C_1t^{\tau_1}$ whenever $ R(x,y)<tR(x,z)$ because of (VD)$_R$.  Since $R$ is uniformly shrinking, there exists $c\in(0,1)$ such that $B_R(x,r/2)\setminus B_R(x,cr/2)\ne\emptyset$ whenever $B(x, r/2)\ne\{x\}.$ Choose $\xi\in B_R(x,r/2)\setminus B_R(x,cr/2)$. Then by (VD)$_R$, there exists $\gamma$, which is independent of $x,\xi$ and $r$ such that $\gamma V_R(x,cr/4)\le V_R(\xi,cr/4)$. This implies 
\[(1+\gamma)V_R(x,cr/4)\le V_R(x,cr/4)+V_R(\xi,cr/4)\le V_R(x,r).\]
 Therefore there exists $C_2$ and $\tau_2$ such that 
\[V_R(x,R(x,y))< C_2 t^{\tau}V_R(x,R(x,z))\text{ whenever }R(x,y)<tR(x,z)\] 
for $t\le 1$ and we obtain desired $\eta(t):=(C_1\vee C_2)(t^{\tau_1}\vee t^{\tau_2})$.
\end{proof}
\begin{proof}[Proof of Proposition \ref{pr-qsmet}] 
(1)Let $\varphi(x,y)=R(x,y)(V(x,R(x,y))+V(y,R(x,y))$ and $\eta$ be a function given in Lemma \ref{lemvdr}. Note that for any $x,y,z\in G$, $R(x,z)/2\le\max\{R(x,y),R(x,z)\}.$ Assume $R(x,z)/2\le R(x,y),$ then 
\[R(x,z)V_R(x,R(x,z))\le2\eta(2)R(x,y)V_R(x,R(x,y))\le2\eta(2)(\varphi(x,y)+\varphi(x,z))\]
and by (VD)$_R$, there exists $C_1$, which is independent of $x,y,z$, such that 
\[\varphi(x,z)\le 2C_1\eta(2)(\varphi(x,y)+\varphi(x,z)).\]
So, by \cite[Proposition 14.5]{Hei}, there exist a metric $d$ on $ G$ and $\beta>1$ such that $\varphi \asymp d^\beta$ for any $x,y\in  G$, and (VD)$_R$ shows $d^\beta(x,y)\asymp V_R(x,R(x,y))R(x,y)$. Moreover, this inequality and the above lemma shows that for some $C_2>0$, 
\[d(x,y)\le C_2 (R(x,y)V_R(x,y))^{1/\beta}\le C_2\eta^{1/\beta}(t) (R(x,z)V_R(x,z))^{1/\beta}\le C_2^2d(x,z)\]
whenever $R(x,y)<tR(x,z)$. This shows $R\qs d.$\\
Next we show $R(\beta).$ Because $d\qs R$, there exists $c>0$ such that $d(x,z)<cd(x,y)$ whenever $R(x,z)<R(x,y),$ and then $B_R(x,R(x,y))\subset B_d(x,cd(x,y)).$ Using the same way, we get $B_d(x,d(x,y))\subset B_R(x,c'R(x,y))$ for some $c'.$ By Lemma \ref{lem-us}, $\mu$ satisfies both (VD)$_d$ and (VD)$_R.$ By using them and the above conditions, we get $V_R(x,R(x,y))\asymp V_d(x,d(x,y))$ for any $x,y \in  G $ and therefore $R(x,y)V_d(x,d(x,y))\asymp d^\beta(x,y).$ \\
Finally, we show $(VG)_\alpha.$ Recall that by Lemma \ref{lem-us}, there exists $\gamma\in (0,1)$ such that $B_d(x,r)\setminus B_d(x,\gamma r)\ne\emptyset$ whenever $B_d(x,r)\ne\{x\}.$ Fix $x\in G$ and assume $r>s>2r_{x,d},$ then there exist $y\in B_d(x,s)\setminus B_d(x,\gamma s)$ and $z\in B_d(x,r)\setminus B_d(x,\gamma r).$ Therefore
\[\frac{V_d(x,r)}{V_d(x,s)}\le C_3\frac{V_d(x,d(x,z))}{V_d(x,d(x,y))}\le C_4\left(\frac{r}{s}\right)^\beta \frac{R(x,y)}{R(x,z)}\] 
for some $C_3,C_4>0,$ because of (VD)$_d$ and $R(x,y)V_d(x,d(x,y))\asymp d^\beta(x,y)$. To evaluate $R(x,z)$, we take $\tau,\nu>1$ such that $\nu R(o,p)\le R(o,q)$ whenever $\tau d(o,p)\le d(o,q).$ Now we let $y=y_0$ and choose $y_{n}\in B_d(x,\gamma^{-1}\tau d(x,y_{n-1}))\setminus B_d(x,\tau d(x,y_{n-1})),$ inductively. Then there exists $C_5>0$ such that $\nu^n R(x,y)\le R(x,y_n) \le C_5R(x,z)$ for any $n$ with $(\tau^n/\gamma^n)\le (d(x,z)/d(x,y)).$ Therefore there exists $C_6,\iota>0$ such that 
\[\frac{R(x,y)}{R(x,z)}\le C_6\left( \frac{s}{r}\right)^\iota  ,\]
and so
\[\frac{V_R(x,r)}{V_R(x,s)}\le C \left(\frac{r}{s}\right)^{(\beta-\iota)\vee 1} .\]
If $s<2r_{x,d}$, then $B_d(x,s)\supset\{x\}=B_d(x,\inf_{y\ne x}d(x,y))$ and (VD)$_d$ imply that this inequality holds by modifying $C.$ Because $\iota,C$ are determined only by (VD)$_d$ and independent of $x,r$ and $s$, we get the desired inequality.\\  
(2) Lemma\ref{lem-us}, (VD)$_d,$ and $(p_0)$ imply
\[ V_R(x,R(x,y))R(x,y)\ge\frac{\mu_x}{\mu_x}=1 \]
for any $x,y\in G$ and there exists $C_1,C_2>0$ such that
\[V_R(x,R(x,y))R(x,y)\le C_1V_R(x,\mu_x^{-1})\mu^{-1}_{xy}=C_1\frac{\mu_x}{\mu_{xy}}\le C_2\] for any $x,y\in  G$ such that $x\sim y.$ 
Since $V(x,R(x,y))R(x,y)\asymp d^\beta(x,y)$, we get desired inequalities for $d(x,y)$. 
 Since (VD)$_d$ holds, for any $x\in  G,c>0$ and $r>c,$ $B_d(x,r)\subset \sum_{i=1}^n B_d(x_i,c)$ for some $n=n(c)$ and $\{x_i\}_{i=1}^n$ (see~\cite{Hei} for example). Let $c=\inf_{x,y\in G}d(x,y)$ then this imply $\#\{B_d(x,r)\}<\infty.$
\end{proof}
The second step of Proof of Theorem \ref{th-d} is applying the method of~\cite{BCK}, but we need some modification because $d$ is not the graph metric. For the sake of completeness, we do not omit the proofs unless no modification is needed.
Properties of $d$ in Proposition \ref{pr-qsmet} (2) and uniformly shrinking condition help our modification.\\
For the rest of this section, let $d$ be a metric on $ G$ and define
\[\tau(x,r)=\tau_d(x,r)=\min\{n|X_n\not\in B_d(x,r)\}\]
where $X_n$ is associated random walk.

\begin{lem}
Assume $d$ is uniformly shrinking and satisfies {\rm(dL)}, {\rm(NdU)}, {\rm(BF)}, {\rm(R($\beta$))} and {\rm(VG($\alpha$))} for some $\beta>\alpha\ge1.$ Then there exists $\lambda,C>0$ such that
\begin{equation}
 R(B_d(x,\lambda r),B_d(x,r)^c)V(x,r)\ge Cr^\beta \tag{ARL($\beta$)}
\end{equation}

for any $x\in G$ and $r>r_0.$
\end{lem}
\begin{proof}
Fix $x\in G$ and $r>r_0.$ Let $c_*=(C_++r_0)/r_0,A=B_d(x,c_*r)\setminus B_d(x,r).$ Since $z_1\not\sim z_2$ for any $z_1\in B_d(X,r)$ and $z_2\in B_d(x,r+c_+)^c,$ $R(B_d(x,\lambda r), A)=R(B_d(x,\lambda r), B_d(x,r)^c)$ for any $\lambda<1,$ so we consider $R(B_d(x,\lambda r),A).$ Let $f_y$ be functions for $y\in A$ such that $f_y(x)=1,f_y(y)=0$ and $\mathcal{E}(f_y,f_y)=R(x,y).$
We first show there exists $\lambda$ such that $f_y(z)\ge 2/3$ for any $z\in B_d(x,\lambda r)$ and $y\in A.$ Since $f_y$ is harmonic on (finite set) $B_d(x,\lambda r)\setminus\{x\},$
\begin{align*}
(1-f_y(z))^2\le (1-f_y(z_0))^2 & \le C\frac{d(x,z_0)^\beta}{V_d(x,d(x,z_0))}\mathcal{E}(f_y,f_y) \\
& \le C'\left(\frac{d(x,z_o)}{d(x,y)}\right)^\beta \frac{V_d(x,d(x,y))}{V_d(x,d(x,z_0))} 
\end{align*}
for some $C,C'>0,z_0\not\in B_d(x,\lambda r)$ such that $z_0\sim z'$ for some $z'\in B_d(x,\lambda r)$ and any $z\in B_d(x,\lambda r)$ by (R($\beta$)). Since $d(x,z_0)\ge(\lambda r-c_+)\vee r_0\ge c_*^{-1}\lambda r$ and $(VG(\alpha))$ holds, for sufficiently small $\lambda$, $f_y(z)\ge 2/3.$ for any $z\in B_d(x,\lambda r)$ and $y\in A.$ In the same way, we also get $f_y(z)\le 1/3$ for any $z\in B_d(y,\lambda r).$ We assume $\lambda<1/4.$\\
Since (VD)$_d$ holds, there exists $n=n_\lambda>0$ such that for any $x\in  G$ and $r>0$, there exist $\{y_i\}_{i=1}^n\subset A$ such that $\sum_{i=1}^n B_d(y_i,\lambda r) \supset A.$ Let $g=\min_{1\le i\le n}f_i$ and $h=(1\vee(3g-1))\wedge 0.$
Then $h|_{B_d(x,\lambda_r)}\equiv 1, h|_{A}\equiv 0$ and so 
\[R(B_d(x,\lambda r),A)^{-1}\le \mathcal{E}(h) \le \mathcal{E}(3g-1)=9\mathcal{E}(g). \] 
Let $z_1,z_2\in X$ such that $z_1\sim z_2$ and $g(z_1)\ge g(z_2)=h_j(z_2)$ for some $j,$ then
\[(g(z_1)-g(z_2))^2 \le (g(z_1)-h_j(z_2))^2 \le \sum_{i=1}^n (h_i(z_1)-h_i(z_2))^2 \]
and it follows that 
\[\mathcal{E}(g)\le\sum_{i=1}^n \mathcal{E}(h_i).\]
Moreover, (VD)$_d$ implies
\[\mathcal{E}(h_i)=R(x,y_i)^{-1}\le C''\frac{V_d(x,d(x,y_i))}{d(x,y_i)^\beta} \]
for some $C''.$ By these inequalities, we obtain $(ARL(\beta)).$
\end{proof}

In the following theorem and proposition, we can apply the original proof in~\cite{BCK}. 
\begin{thm}{(\cite[Theorem 3.1]{BCK})}
Assume {\rm(VG($\alpha$))} and {\rm(R($\beta$))} hold for $\beta>\alpha\ge1.$ Then there exists $C>0$ such that for any $x\in X$ and $n\in\mathbb{N},$
\begin{equation}
h_n(x,x)\le \frac{C}{V_d(x,n^{(1/\beta)})}. \tag{DUHK($\beta$)}
\end{equation}
\end{thm} 

\begin{prop}{(\cite[Proposition 3.4]{BCK})}\label{eb}
Assume {\rm(BF), (VG($\alpha$))} and {\rm(R($\beta$))} hold for some $\beta>\alpha\ge1.$ If there exists $C,r_0>0$ such that 
\begin{equation}
R(x,B_d(x,r)^c)V(x,r)>Cr^\beta \tag{BRL($\beta$)}
\end{equation}
for any $x\in  G$ and $r>r_0.$ Then 
\begin{equation}
\mathbb{E}^x[\tau(x,r)]\asymp r^\beta \tag{E($\beta$)}
\end{equation}
for any $x\in G$ and $r>r_0.$
\end{prop}
Now we give on-diagonal lower heat kernel estimate. The following type of result is well-known, but we prove it for completeness.
\begin{prop}
 Assume {\rm(E($\beta$))} and {\rm(VG($\alpha$))}, then there exists $C>0$ such that for any $x\in G$ and $n$,
\begin{equation}
   h_{2n}(x,x)\ge\frac{c}{V_d(x,n^{1/\beta})}. \tag{DLHK($\beta$)}
 \end{equation}
\end{prop}
 \begin{proof}
In the same way as \cite[Lemma 3.7]{BCK}, we can prove that there exist $p\in(0,1)$ and $A>0$ such that 
\begin{equation}
\mathbb{P}^x(\tau(x,r)\le n)\le p+An/r^\beta \label{eq1}
\end{equation}
for any $x\in G,r>r_0,$ and $n\in\mathbb{Z}_+.$
Hence
\[ \sum_{y\not\in B_d(x,Cn^{1/\beta})}\hspace{-15pt}h_n(x,y)\mu_y=\mathbb{P}^x(X_n\not\in B_d(x,Cn^{1/\beta}))\le \mathbb{P}^x(\tau(x,Cn^{1/\beta})\le n) \le \frac{1+p}{2} \]
for some $C>0.$ Therefore
\begin{align*}
  h_{2n}(x,x)& =\sum_{y\in G}h_n^2(x,y)\mu_y\ge \sum_{y\in B_d(x,Cn^{1/\beta})}h_n^2(x,y)\mu_y\\
& \ge\frac{1}{V_d(x,Cn^{1/\beta})}\left( \sum_{y\in B_d(x,Cn^{1/\beta})}1\cdot h_n(x,y)\mu_y\right)^2\\
& \ge\left(\frac{1-p}{2}\right)^2\frac{1}{V_d(x,Cn^{1/\beta})} 
\end{align*}
and (VG$(\alpha)$) implies desired result.
\end{proof}
\begin{proof}[Proof of Theorem \ref{th-d}]
By using (VG($\alpha)$), all but (BRU($\beta$)) have shown in above statements (recall that $d(x,y)\asymp 1$ follows from (dL) and (NdU), and note that (BRL($\beta$)) for Proposition \ref{eb} immediately follows from (ARL($\beta$))). Since $d$ is uniformly shrinking, there exists $\alpha\in(0,1)$ such that for any $r>r_x,$ there exists $y\in B_d(x,\alpha^{-1}r)\setminus B_d(x,r)$, so $R(x,B_d(x,r)^c)\le R(x,y)$ with $r\le d(x,y)<\alpha^{-1}r.$ This with (R($\beta$)) and (VG($\alpha$)) implies (BRU($\beta$)).

\end{proof}

\begin{rem}
We have seen that if (BF),(VG($\alpha$)), (R($\beta$)) and (BRL($\beta$)) hold, then (DUHK($\beta$)) and (DLHK($\beta$)) hold without (dL) nor (NdU). This was already used, for example, see~\cite{KM}. 
\end{rem}
We have shown Theorem \ref{th-d}, but we give more precise heat kernel estimates by using $d$. The following Corollary is a discrete version of~\cite[Corollary 15.12]{Kig12}.
\begin{cor} \label{th-qssg}
Let $( G,\mu)$ be a weighted graph satisfying condition $(p_0)$. Moreover, assume $V_R(x,r)<\infty$ for any $x\in  G,r>0$, and $\mathrm{diam}(X,R)=\infty$, then the following conditions are equal.
\begin{enumerate}
\item $\mu$ satisfies {\rm(VD)$_R$}
\item There exists a distance $d$ on $ G$ and constants $\beta, c_1,c_2,c_3$ such that 
  \begin{itemize}
  \item $d\qs R.$
 
  \item For any $x,y\in G$ and $n\in\mathbb{N}$,
\[ h_n(x,y)\le\frac{c_1}{V_d(x,n^{1/\beta})}\exp\biggl(-\Bigl(\frac{d(x,y)^\beta}{c_1n} \Bigr)^\frac{1}{\beta-1}\biggr)\tag{UHK($\beta$)}. \]
  \item For any $x,y\in  G$ and $n\in\mathbb{N}$ such that $d(x,y)\le c_2n^{1/\beta}$,
\[ h_n(x,y)+h_{n+1}(x,y)\ge\frac{c_2}{V_d(x,n^{1/\beta})} \tag{NLHK($\beta$)} \]
   \item For any $x\in G$ and $n\in\mathbb{N},$ $h_n(x,x)\le c_3h_{2n}(x,x)$\hspace{10pt}{\rm(KD)}
  \end{itemize}
\end{enumerate}
\end{cor}
For the last of this paper, we prove it.
\begin{lem}
Assume {\rm(dL), (NdU),} and {\rm(E($\beta$))} hold, then there exist $C>0$ such that for any $n$, $r>0$ and $x\in G,$
\begin{equation}
 \mathbb{P}^x(\tau(x,r)\le n)\le C\exp\left(-\left(\frac{r^\beta}{Cn}\right)^{\frac{1}{\beta-1}}\right). \tag{$\Psi$($\beta$)} 
\end{equation}
\end{lem}
\begin{proof}
Recall that (E($\beta$)) implies \eqref{eq1}.
Let $a>2r_0$ and assume $r^\beta n^{-1}\ge a.$ Also let $l\ge 1$ such that $b:=(r/l)-r_0>0.$ Now we introduce notions 
\[\sigma_0=0, \sigma_{i+1}=\inf \{m\ge \sigma_i:d(X_{\sigma_i},X_m)\} \]
for $i\ge0$ and $\zeta_i=\sigma_i-\sigma_{i-1}$ for $i\ge 1.$ Since 
\eqref{eq1} holds,
\[\mathbb{P}^x(\zeta_{i+1}\le n|\sigma_i)=\mathbb{P}^{X_{\sigma_i}}(\tau(X_{\sigma_i},b)\le n) \le p+An/b^\beta.\]
Remark that $d(X_0,X_{\sigma_l})<(b+r_0)l=r$ and so $\sigma_l=\sum_{i=1}^l\zeta_i<\tau(X,r)$ Therefore by \cite[Lemma3.14,]{Bar98}
\[\log\mathbb{P}^x(\tau(x,r)\le n)\le2\left(\frac{A'ln}{pb^\beta}\right)^{1/2}-l\log\left(\frac{1}{P}\right) =C_1\left(\frac{l^{\beta+1}n}{(r-lr_0)^\beta}\right)^{1/2}-C_2l \]
for some $C_1,C_2.$ Next we consider the following inequality
\begin{equation}
C_1\left(\frac{l^{\beta+1}n}{(r-lr_0)^\beta}\right)^{1/2}\le C_2l/2\label{eq2}
\end{equation}
If $l=1,$ then
\begin{align*}
& C_1\left(\frac{n}{(r-r_0)^\beta}\right)^{1/2}-C_2/2 \\
 \le& 4C_1\left( \frac{n}{r^\beta} \right)^{1/2}-C_2/2 \le 4C_1a^{-1/2}-C_2/2,\\
\end{align*}
hence \eqref{eq2} holds if $a$ is sufficiently large. On the other hand, if $l=\lceil \eta\frac{r}{r_0} \rceil-1 $ for $\eta\in(1/2,1),$ then
\begin{align}
& C_1\left(\frac{l^{\beta+1}n}{(r-lr_0)^\beta}\right)^{1/2}-C_2l/2 \notag \\
\ge & C_1\left(\frac{\eta(r/r_0)-1}{r-(\eta(r/r_0)-1)r_0}\right)^{\beta/2}\left(\eta\frac{r}{r_0}-1\right)^{1/2}n^{1/2}-\eta\frac{C_2}{r_0}r \notag \\
= & C_1\left(\frac{\eta r-r_0}{(1-\eta)r+r_0}\right)^{\beta/2}\left(\frac{\eta r-r_0}{r_0}\right)^{1/2}n^{1/2}-\eta\frac{C_2}{r_0}r \label{eq3}.
\end{align}
Now we fix $\eta$ such that 
\[\frac{C_1}{2^{\beta/2}}\left(\frac{2\eta-1}{1-\eta}\right)^{\frac{\beta+1}{2}}-C_2\frac{1}{1-\eta}>0 \]
(this holds if $(1-\eta)$ is sufficiently small). Therefore if $a\ge r_0/(1-\eta)$ then
\[\eqref{eq3}\ge \frac{C_1}{2^{\beta/2}}\left(\frac{2\eta-1}{1-\eta}\right)^{\frac{\beta+1}{2}}n^{1/2}-C_2\frac{1}{1-\eta}>0 \]
because $r\ge a$ and hence \eqref{eq2} does NOT hold.
Therefore if we fix such $\eta$ and sufficiently large $a,$ especially satisfying $ a\ge 2r_0/(1-\eta),$ then there exists $l$ such that $1\le l<  \lceil \eta\frac{r}{r_0} \rceil -1$, \eqref{eq2} holds, and 
\[C_1\left(\frac{2}{1-\eta}\right)^{\beta/2}\left(\frac{(2l)^{\beta+1}n}{r^\beta}\right)^{1/2}\ge C_1\left(\frac{(l+1)^{\beta+1}n}{(r-(l+1)r_0)^\beta}\right)^{1/2}\ge\frac{C_2}{2}(l+1)\ge \frac{C_2}{2}l \]
because $r-(1+l)r_0\ge(1-\eta)r-r_0\ge(1-\eta)r/2.$ Then 
\[\log\mathbb{P}^x(\tau(x,r)\le n)\le -\frac{C_2}{2}l\] and we obtain ($\Psi$($\beta$)) for $r^\beta n^{-1}\ge a.$ Adjusting the constants if necessary, ($\Psi$($\beta$)) also holds if $r^\beta n^{-1}<a.$
\end{proof}

\begin{prop}
If {\rm(DUHK($\beta$)), (VG($\alpha$)),($\Psi$($\beta$)), (NdU)} and {\rm(dL)} hold, then {\rm(UHK($\beta$))} hold. 
\end{prop}
\begin{rem}
  This type result is well-known (see \cite{Bar98} for example.) Our proof is based on \cite[Section 12.3]{Gri} and \cite[Theorem 17.5]{Kig12}, these are for the Volume doubling settings.  
\end{rem}
\begin{proof}
We only consider cases of $x\ne y.$ Let $x,y\in  G$ and $n\in\mathbb{N}.$ Then for any $r\le d(x,y)/2,$
\begin{align}
h_{2n}(x,y) &\le \frac{\mathbb{P}^x(X_{2n}=y,X_n\not\in B_d(x,r)}{\mu_y} +\frac{\mathbb{P}^x(X_{2n}=y,X_n\not\in B_d(x,r))}{\mu_y} \notag \\
&\le \sum_{i=1}^n \sum_{z\in\ol{\partial}B_d(x,r)}\mathbb{P}^x(\tau(x,r)=i,X_i=z)h_{2n-i}(x,z)\notag\\
&\hspace{10pt}+\sum_{i=1}^n \sum_{z\in\ol{\partial}B_d(y,r)}\mathbb{P}^y(\tau(y,r)=i,X_i=z)h_{2n-i}(y,z)\notag \\
& \le \mathbb{P}^x(\tau(x,r)\le n)\sup_{n\le i \le 2n-1}\sup _{z\in\ol{\partial}B_d(y,r)}h_i(z,y)\notag\\
&\hspace{10pt}+\mathbb{P}^y(\tau(y,r)\le n)\sup_{n\le i \le 2n-1}\sup _{z\in\ol{\partial}B_d(y,r)}h_i(z,y).\label{eq4}
\end{align}
On the other hand, there exists $C>0$ such that
\begin{align*}
h_{2i}(x,y)=\sum_{z\in  G}h_i(x,z)h_i(y,z)\mu_z &\le \left(\sum_z h^2_{i}(x,z)\mu_z\right)^{1/2}\left(\sum_z h^2_{i}(y,z)\mu_z\right)^{1/2}\\
&= h_i(x,x)^{1/2}h_i(y,y)^{1/2} \\
&\le \frac{C}{\sqrt{V_d(x,n^{1/\beta})V_d(y,n^{1/\beta})}} 
\end{align*}
because Chapman-Kolmogorov and Cauchy-Schwarz inequalities, and (DUHK($\beta$)) hold. Remark that for any $\epsilon>0,$ there exists $C_\epsilon$ such that 
\[ 1+\frac{s}{n^{1/\beta}}\le C_\epsilon \exp \left( \epsilon\left(\frac{s^\beta}{n}\right)^{\frac{1}{\beta-1}}\right) \]
for any $n$ and $s>0$ (see [Kig,Lemma 17.7] ). Since $(VG)_\alpha$ holds, for any $y,z\in B_d(x,(\frac{3}{2}+\frac{2c_+}{r_0})d(x,y))$, 
\[h_{2i}(y,z)\le\frac{C_\epsilon'}{V_d(x,n^{1/\beta})}\exp\left(\frac{\alpha}{C_\epsilon''}\epsilon\left(\frac{d^\beta(x,y)}{n}\right)^{\frac{1}{\beta-1}}\right)\]
for some $C_\epsilon',C_\epsilon''>0.$ For odd integers, using the following assumption
\begin{equation}
h_{2i+1}(x,y)=\sum_{z\in G}h_{2i}(x,z)p(z,y)\le\max_{z\sim y}h_{2i}(x,z)\label{eq5}
\end{equation}
and $\ol{(\ol{B_d(x,2d(x,y)/3)})}\subset B_d(x,(\frac{3}{2}+\frac{2c_+}{r_0})d(x,y)),$ we get
\begin{align*}
\sup_{n\le i \le 2n-1}\sup _{z\in\ol{\partial}B_d(y,r)}h_i(x,z) &\le \frac{C_\epsilon'}{V_d(x,n^{1/\beta})}\exp\left(\frac{\alpha}{C_\epsilon''}\epsilon\left(\frac{d^\beta(x,y)}{n}\right)^{\frac{1}{\beta-1}}\right)\\
\sup_{n\le i \le 2n-1}\sup _{z\in\ol{\partial}B_d(y,r)}h_i(y,z)&\le \frac{C_\epsilon''}{V_d(x,n^{1/\beta})}\exp\left(\frac{\alpha}{C_\epsilon''}\epsilon\left(\frac{d^\beta(x,y)}{n}\right)^{\frac{1}{\beta-1}}\right)
\end{align*}
with modifying constants $C_\epsilon',C_\epsilon''.$ Therefore we conclude that
\[h_{2n}(x,y)\le\frac{CC_\epsilon'}{V_d(x,d(x,y))}\left(\left(\frac{\alpha}{C_\epsilon''}\epsilon-(c_*2^\beta)^{\frac{-1}{\beta-1}}\right)\left(\frac{d^\beta(x,y)}{n}\right)^{\frac{1}{\beta-1}}\right) \]
with ($\Psi$($\beta$)). Taking sufficiently small $\epsilon$ and adjusting constants, we get (UHK($\beta$)) for even $n$. Also using \eqref{eq5} and adjusting constants again, we get (UHK($\beta$)) for any $n.$
\end{proof}

\begin{prop}{(\cite[Proposition 3.11]{BCK})}
Assume {\rm(VG($\alpha$)), (R($\beta$)), (DLHK($\beta$))} hold for some $\beta>\alpha\ge1.$ Then {\rm(NLHK($\beta$))} holds.
\end{prop}
The Proof of the above proposition is in the same way as~\cite{BCK}.
\begin{proof}[Proof of Corollary \ref{th-qssg}]
((1)$\Rightarrow$(2)) It has shown by above statements without (KD) . We can also get (KD) by (DLHK($\beta$)), (DUHK($\beta$)) and (VD)$_d.$\\
((2)$\Rightarrow$(1)) (UHK($\beta$)), (NLHK($\beta$)) and (KD) shows (VD)$_d$. By  Lemma \ref{lem-us} and Lemma \ref{le-Rus}, $d$ is uniformly shrinking. Again using Lemma \ref{lem-us}, we get (VD)$_R.$ 
\end{proof}

\end{document}